\documentclass[a4paper]{article}
\usepackage{graphicx, float, xcolor, subcaption, bm}
\usepackage{booktabs, multirow, multicol, amssymb}
\usepackage[shortlabels]{enumitem}
\usepackage[colorinlistoftodos]{todonotes}
\usepackage{appendix}
\usepackage{aligned-overset}
\usepackage[numbers]{natbib}
\usepackage{enumitem}
\usepackage{titlesec}
\usepackage{booktabs}
\usepackage[ruled,vlined]{algorithm2e}
\usepackage{diagbox}
\usepackage{multicol}
\usepackage[hidelinks]{hyperref}       %
\usepackage{blindtext}
\usepackage{url}            %
\usepackage{booktabs}       %
\usepackage{amsfonts}       %
\usepackage{nicefrac}       %
\usepackage{amsthm, amsmath, amssymb}
\usepackage{graphicx, float, xcolor, subcaption, bm}
\numberwithin{equation}{section}

\theoremstyle{plain}
\newtheorem{theorem}{Theorem}[section]
\newtheorem{assumption}{Assumption}[section]
\newtheorem{alg}{Algorithm}[section]
\newtheorem{lemma}{Lemma}[section]
\newtheorem{proposition}{Proposition}[section]
\newtheorem{corollary}{Corollary}[section]
\newtheoremstyle{definition}
{\topsep}
{\topsep}
{}
{}
{\bfseries}
{.}
{.5em}
{}
\theoremstyle{definition}
\newtheorem{definition}{Definition}[section]
\newtheorem{example}{Example}[section]
\newtheorem{remark}{Remark}[section]

\usepackage{lipsum}
\usepackage{amsfonts}
\usepackage{graphicx}
\usepackage{epstopdf}
\usepackage{algorithmic}
\ifpdf
  \DeclareGraphicsExtensions{.eps,.pdf,.png,.jpg}
\else
  \DeclareGraphicsExtensions{.eps}
\fi

\usepackage{etoolbox}
\usepackage{hyperref}
\preto\subequations{\ifhmode\unskip\fi}

\title{A Semismooth Newton based Augmented Lagrangian Method for Nonsmooth Optimization on Matrix Manifolds\thanks{The work was supported in part by National Natural Science Foundation of China (11901338, 61620106010). The research of the third
author was supported in part by the National Natural Science Foundation of China (12071464) and the
Beijing Natural Science Foundation (Z190002).}
}
\author{Yuhao Zhou\footnote{Department of Computer Science and Technology, Tsinghua University, China. 
  (\url{yuhaoz.cs@gmail.com}).}
\and Chenglong Bao\footnote{Yau Mathematical Sciences Center, Tsinghua University, China and Yanqi Lake Beijing Institute of Mathematical Sciences and Applications, China (\url{clbao@mail.tsinghua.edu.cn}).}
\and Chao Ding\footnote{Institute of Applied Mathematics, Academy of Mathematics and System Sciences, Chinese Academy of Sciences, China (\url{dingchao@amss.ac.cn}).}
\and Jun Zhu\footnote{Department of Computer Science and Technology, Tsinghua University, China.  (\url{dcszj@tsinghua.edu.cn}).}
}

\usepackage{amsopn}
\usepackage{ulem}
\usepackage{fullpage}

\newcommand{\peq}{\phantom{{}={}}}
\newcommand{\St}[1]{\mathrm{St}(#1)}
\newcommand{\prox}{\mathrm{prox}}
\newcommand{\feas}{{\mathrm{feas}}}
\newcommand{\trans}{\top}

\newcommand{\R}{\mathbb R}

\newcommand{\N}{\mathbb N}

\DeclareMathOperator{\grad}{grad}
\DeclareMathOperator*{\Hess}{Hess}
\newcommand{\Proj}{\mathbf{P}}

\DeclareMathOperator*{\argmin}{arg\,min}
\newcommand{\norm}[1]{\left \Vert #1 \right \Vert}
\newcommand{\normz}[1]{\Vert #1 \Vert}

\newcommand{\tr}{\text{tr}} %

\newcommand{\inner}[2]{\left \langle #1, #2 \right \rangle}
\newcommand{\innerz}[2]{\langle #1, #2 \rangle}
\newcommand{\id}{{\rm{id}}} %

\newcommand{\bone}{\textbf{1}}

\newcommand{\cA}{\mathcal{A}}

\newcommand{\cC}{\mathcal{C}}
\newcommand{\cD}{\mathcal{D}}

\newcommand{\cI}{\mathcal{I}}

\newcommand{\cK}{\mathcal{K}}

\newcommand{\cL}{\mathcal{L}}

\newcommand{\cM}{\mathcal{M}}
\newcommand{\cN}{\mathcal{N}}

\newcommand{\cP}{\mathcal{P}}

\newcommand{\cX}{\mathcal{X}}

\newcommand{\cZ}{\mathcal{Z}}
\newcommand{\dd}{\mathrm{d}}
\newcommand{\II}{\mathrm{II}}

\renewcommand{\Box}{}

\begin{document}

\maketitle

\begin{abstract}
This paper is devoted to studying an augmented Lagrangian method for solving a class of manifold optimization problems, which have nonsmooth objective functions and nonlinear constraints. Under the constant positive linear dependence condition on manifolds, we show that the proposed method converges to a stationary point of the nonsmooth manifold optimization problem. Moreover, we propose a globalized semismooth Newton method to solve the augmented Lagrangian subproblem on manifolds efficiently. The local superlinear convergence of the manifold semismooth Newton method is also established under some suitable conditions.
We also prove that the semismoothness on submanifolds can be inherited from that in the ambient manifold. Finally, numerical experiments on compressed modes and (constrained) sparse principal component analysis illustrate the advantages of the proposed method.

\vspace{0.6em}

\vspace{0.6em}
\end{abstract}

\section{Introduction}
Manifold optimization is recently growing in popularity as it naturally arises from various applications in many fields,
including phase retrieval~\cite{cai2019fast, waldspurger2015phase}, principal component analysis~\cite{lu2012spca,montanari2015non}, 
matrix completion~\cite{boumal2011rtrmc,vandereycken2013low}, 
medical image analysis~\cite{lai2014folding},
and deep learning~\cite{cho2017riemannian}.
It is concerned with optimization problems with a manifold constraint, 
and has been extensively studied when the objective function is smooth during the past decades~\cite{absil2009optimization,gabay1982minimizing,hu2019brief,hu2018adaptive,wen2013feasible,zhang2016riemannian}.
Nonsmooth manifold optimization is less explored but has drawn increasing attention in recent years~\cite{chen2020proximal,chen2016augmented,huang2019riemannian,lai2014splitting}.
In this paper, we consider the following nonsmooth and nonconvex manifold optimization problem:
\begin{equation}
	\label{eqn:original-problem-0}
	\begin{aligned}
		\min_{x}\  \left \{ f(x) + \psi(h_1(x)) \right \},\ \mathrm{s.t.}\  x \in \cM,\ h_2(x) \leq 0,
	\end{aligned}
\end{equation} 
	 where $\cM$ is a Riemannian manifold, $f : \cM \to \R$, $h_1: \cM \to \R^m$, $h_2: \cM \to \R^q$ are continuously differentiable, and $\psi: \R^m \to \R$ is convex. 
	 Besides the nonsmooth function $\psi$ in \eqref{eqn:original-problem-0}, the inequality constraint can be modeled using the nonsmooth function $\delta_{\cC}(x)$, where $\delta_\cC$ is the indicator function of the set $\cC := \{x : h_2(x)\leq0\}$. From the perspective of algorithm design, the two nonsmooth terms in \eqref{eqn:original-problem-0} are difficult to handle in general.
	 On the other hand, the model~\eqref{eqn:original-problem-0} has many important applications in machine learning and scientific computing. We list some typical examples as follows and refer the readers to~\cite{absil2019collection,chen2020proximal,hu2019brief} for more examples and details. 
\begin{enumerate}
\item {\bf{Compressed modes (CM)~\cite{ozolina2013compressed}.}} The CM problem seeks for sparse eigenfunctions to a class of Hamiltonian operators, as these localized spatial bases play an important role in representing the rapid varying functions in physics and quantum chemistry. Numerically, let $H$ be a discretization of the Hamiltonian operator, 
then the CM problem is formulated as:
	\begin{equation}
	\label{eqn:cm-problem}
	\begin{aligned}
	\min_{Q \in \mathrm{St}(n, r)}\  &\left \{ \tr(Q^\trans H Q) + \mu \norm{Q}_1 \right \},
	\end{aligned}
	\end{equation}
	 where $\St{n, r} := \{ Q \in \R^{n\times r} : Q^\trans Q=I_r \}$ is the Stiefel manifold.
	\item {\bf{Sparse principal component analysis (SPCA)~\cite{zou2006sparse}.}} Let $A\in\R^{p\times n}$ be a data matrix where~$n$ and $p$ are the number of variables and the number of observations, respectively.
	Setting~$f(Q)=-\tr(Q^\trans A^\trans A Q)$, $\psi(Q) = \mu\norm{Q}_1$, $h_1(Q) = Q$, $h_2(Q) \equiv 0$ and $\cM = \St{n, r}$, the SPCA problem has the form \eqref{eqn:original-problem-0}, i.e., it aims at solving
	\begin{equation}
	\label{eqn:spca-problem}
	\begin{aligned}
	\min_{Q \in \mathrm{St}(n, r)}\  &\left \{ -\tr(Q^\trans A^\trans A Q) + \mu \norm{Q}_1 \right \}.
	\end{aligned}
	\end{equation}
	\item {\bf{Constrained SPCA~\cite{lu2012spca}.}} To further enforce the orthogonality among principal components, the constrained SPCA problem imposes additional constraints on each column of $Q$, which has the form
	\begin{equation}
	\label{eqn:cons-spca-problem}
	\begin{aligned}
	\min_{Q \in \mathrm{St}(n, r)}\  &\left \{ -\tr(Q^\trans A^\trans A Q) + \mu \norm{Q}_1 \right \}, \\
	\mathrm{s.t.}\ & \quad |Q_i^\top A^\trans AQ_j| \leq \Delta_{ij}, \ \forall\, i \neq j,
	\end{aligned}
	\end{equation}
	where $Q_i$ denotes the $i$-th column of $Q$ and $\Delta_{ij}\geq 0$ are the predefined tolerances.
\end{enumerate}

It is worth mentioning that \eqref{eqn:original-problem-0} can be regarded as an unconstrained nonsmooth optimization problem on $\cM$ when the inequality constraint is dropped, i.e., $h_2(x) \equiv 0$.
Various methods are designed to solve such problems.
The subgradient methods in the Riemannian setting are studied in \cite{borckmans2014riemannian,dirr2007nonsmooth}.
Riemannian proximal point algorithms are investigated by~\cite{chen2020proximal,ferreira2002proximal,huang2019riemannian}. 
Operator splitting methods like the alternating direction methods of multipliers (ADMM) and the augmented Lagrangian methods (ALM) are also promising on manifolds~\cite{kangkang2019inexact,kovnatsky2016madmm}, in which unconstrained manifold optimization algorithms are used to solve the subproblems.
However, due to the presence of two nonsmooth terms in \eqref{eqn:original-problem-0}, the existing manifold-based algorithms may not be directly applicable for solving \eqref{eqn:original-problem-0} in general cases. 
For example, the direct application of ManPG~\cite{chen2020proximal} or the Riemannian proximal gradient method~\cite{huang2019riemannian} requires dealing with a subproblem with two nonsmooth terms, which is generally difficult to solve. Moreover, by introducing more auxiliary variables, the multiblock ADMM method can be exploited, but there is no convergence guarantee of such algorithm, to the best of our knowledge.

On the other hand, in many situations $\cM$ is embedded in a Euclidean space and can be specified by equality constraints, e.g., the Stiefel manifold or the Oblique manifold. 
In these cases, \eqref{eqn:original-problem-0} can be viewed as a constrained optimization problem in Euclidean spaces~\cite{chen2016augmented,lai2014splitting,lu2012spca,zhu2017nonconvex}, 
which has been widely studied for many years~\cite{andreani2008augmented,boyd2011distributed,chen2017augmented,sun2008rate}. However, due to the complex constraints induced by embedded manifolds, the constraint qualifications may not be satisfied, and the nonlinear methods are not applicable for solving optimization problems over abstract manifolds.
Motivated by the above analysis, we aim at designing numerical algorithms for solving~\eqref{eqn:original-problem-0} by exploiting the intrinsic structure of manifolds.

\subsection*{The Main Contributions} 
In this paper, we propose a manifold-based augmented Lagrangian method to solve \eqref{eqn:original-problem-0}, which consists of two nonsmooth terms. Compared to other existing methods, the proposed algorithm satisfies the manifold constraint automatically at each step and exploits the second-order geometric property of manifolds.
The main idea of the proposed method is to introduce auxiliary variables that split~\eqref{eqn:original-problem-0} into a smooth manifold constrained term, a nonsmooth term and an inequality constrained term. Then, we apply the augmented Lagrangian method to solve the equivalent version of~\eqref{eqn:original-problem-0} and show the global convergence property under some constraint qualifications on manifolds. Using the Moreau-Yosida identity, the augmented Lagrangian subproblem can be converted to a continuous and differentiable manifold optimization problem, but it is not second-order differentiable. This subproblem is inherently different from the subproblems in ManPG~\cite{chen2020proximal} and the Riemannian proximal gradient method~\cite{huang2019riemannian}, which are nonsmooth problems on the tangent space. To solve the augmented Lagrangian subproblem, we propose a globalized version of the semismooth Newton method on manifolds and prove its local superlinear convergence under some reasonable assumptions. 
We also provide a theorem showing that when the manifold $\cM$ is a compact submanifold of another Riemannian manifold $\bar \cM$, 
the semismoothness of a vector field $X$ on $\cM$ can be inherited from the semismoothness of its extension to $\bar \cM$.
Numerical results in compressed modes~\eqref{eqn:cm-problem}, sparse PCA~\eqref{eqn:spca-problem}, and the constrained sparse PCA~\eqref{eqn:cons-spca-problem} show the advantages of the proposed method comparing with existing approaches.

The remaining parts of this article are organized as follows:
Preliminaries on manifolds and backgrounds about previously mentioned optimization methods are presented in Sec.~\ref{sec:background}.
The augmented Lagrangian method together with its convergence analysis are given in Sec.~\ref{sec:alm}.
The globalized semismooth Newton method for dealing with the subproblem, and its global convergence together with the transition to local superlinear convergence are shown in Sec.~\ref{sec:newton}.
The semismooth property on submanifolds and the method for calculating the Clarke generalized covariant derivative are also explored in Sec.~\ref{sec:newton}.
Numerical experiments are reported in Sec.~\ref{sec:exp}.
Finally, this article is concluded in Sec.~\ref{sec:conclusions}.

\begin{center}
	{\small \begin{table}
	\centering
		\caption{Notations used in this work.}
		\label{Table:Notations}
		\begin{tabular}{|c|l|}
			\hline
			\multicolumn{1}{|c|}{\bf Notations} &
			\multicolumn{1}{c|}{\bf Descriptions} \\ \hline
			$[u]_i$ & The $i$-th component of $u\in\R^d$ \\ \hline
			$[M]$ & The set $\{1,2,\cdots,M\}$, where $M$ is a positive integer \\ \hline
			$\cM$ & A complete $n$-dimensional smooth Riemannian manifold \\ 
			\hline
			$T_p\cM$ & The tangent space at $p\in\cM$ \\ \hline
			$T\cM$ & The tangent bundle of $\cM$ \\ \hline
			$\cD(\cM)$ & The set of smooth functions on $\cM$ with compact support \\ \hline 
			$\dd\varphi|_p$ & The differential of the smooth map $\varphi$ at $p\in\cM$  \\ \hline
			$\grad\varphi$ & The gradient of the function $\varphi$ on manifolds\\ \hline
			$\partial \varphi$ & The Clarke subgradient of the function $\varphi$ on manifolds \\ \hline
			$\Hess \varphi$ & The Hessian of the function $\varphi$ on manifolds \\ \hline
			$X,Y$ & Vector fields on manifolds \\ \hline
			$\cX(\cM)$ & The set of all smooth vector fields on $\cM$ \\ \hline
			$\partial X$ & The Clarke generalized covariant derivative of the vector field $X$ \\ \hline 
			$\nabla_X Y$ & The Levi-Civita connection of two vector fields $X$ and $Y$ \\ \hline
			$\nabla X(p; v)$ & The directional derivative of a vector field $X$ at $p$ along $v$ \\ \hline
			$\nabla X(p)$ & The operator $v \mapsto \nabla_v X(p)$ from $T_p\cM$ to $T_p\cM$ \\ \hline
			$P^{s \to t}_\gamma$ & The parallel transport along a curve $\gamma$ from $\gamma(s)$ to $\gamma(t)$ \\ \hline
			$P_{pq}$ & The parallel transport along the geodesic from $p$ to $q$ \\ \hline 
			$\exp_p$ & The exponential map at $p\in\cM$ \\ \hline
			${\mathcal{L}}(T_p\cM)$ & The linear space of all linear operators from $T_p\cM$ to $T_p\cM$ \\ \hline
			$\Proj_p V$ & The projection of a vector $V \in \R^d$ into $T_p \cM$ \\ \hline
			$B_p(r)$ & The open ball $\{ q \in \cM : d(p, q) < r \}$ \\ \hline
			$\|Q \|_1$ & The $\ell_1$-norm of $Q \in \R^{n\times r}$, namely $\sum_{i, j} |Q_{ij}|$ \\ \hline
			$\|Q \|_F$ & The Frobenius norm of the matrix $Q \in \R^{n\times r}$ \\ \hline
			$\|Q \|$ & The operator norm of the matrix $Q \in \R^{n \times r}$ \\ \hline
			$\|x \|_2$, $\| x\|_{\R^n}$ & The $\ell_2$-norm of the vector $x \in \R^n$ \\ \hline
			$\|x \|_\infty$ & The $\ell_\infty$-norm of the vector $x \in \R^n$ \\ \hline
			$\inner{\eta}{\xi}_p$,$\inner{\eta}{\xi}$ & The Riemannian inner product of $\eta, \xi \in T_p\cM$  \\ \hline
			$\norm{\eta}_p$, $\norm{\eta}$ & The norm of $\eta \in T_p\cM$  \\ \hline
		\end{tabular}
	\end{table}}
\end{center}

\section{Background} \label{sec:background}
In this section, we review some concepts of manifolds and briefly discuss some related literature of nonsmooth manifold optimization, nonsmooth nonconvex ALM in Euclidean spaces, and the semismooth Newton method.

\subsection{Preliminaries on Manifolds}\label{sec:background-manifold}
A Hausdorff topological space $\cM$ is said to be an \emph{$n$-dimensional manifold} if it has a countable basis and for each $p \in \cM$ there exist a neighborhood $U$ of $p$, an open subset $\hat U \subset \R^n$ and a map $\varphi: U \to \hat U$ such that $\varphi$ is a homeomorphism. The pair $(U, \varphi)$ is called a \emph{chart}.

Notations used in the remaining part of this article are listed in Table~\ref{Table:Notations}. 
As the Euclidean spaces can be interpreted as the linear manifolds~\cite{absil2009optimization}, 
our notations for manifolds are consistent with those used in Euclidean spaces when the function is defined on $\R^n$, e.g., $\grad\varphi=\nabla\varphi$ if $\varphi$ is defined on $\R^n$. Now, we briefly 
review basic definitions and properties of functions defined on manifolds.
Most of these definitions can be found in, e.g., \cite[Chapter 1-3]{carmo1992riemannian} and \cite[Chapter 3]{absil2009optimization}.
\begin{definition}  %
    A \emph{tangent vector} $\xi_p:\cD(\cM)\to\R$ to a manifold $\cM$ at a point $p$ is a linear operator such that for every $f \in \cD(\cM)$, %
        $\xi_p f = \dot\gamma(0) f:= \left.\frac{\dd(f(\gamma(t)))}{\dd t}\right|_{t=0}$,
    where $\gamma: (-1, 1) \to \cM$ is a smooth curve on $\cM$ with $\gamma(0)=p$.
\end{definition}
The \emph{tangent space} $T_p\cM$ is the space containing all tangent vectors at $p$, which is an $n$-dimensional $\R$-linear space. 
A \emph{Riemannian metric} $\langle\cdot,\cdot\rangle_p$ gives an inner product on $T_p\cM$, which smoothly depends on~$p$.\footnote{The subscript $p$ in $\inner{\cdot}{\cdot}_p$ is usually omitted for simplicity.}
Moreover, a Riemannian metric gives a metric on $\cM$ and the gradient of functions defined on $\cM$.
Below we assume that $\cM$ is equipped with a Riemannian metric.
\begin{definition} %
	Given $p,q\in\cM$, the \emph{distance between $p$ and $q$} is defined as 
	\begin{equation*}
	d(p,q) = \inf \left\{ \ell(\gamma)  :  \ell(\gamma) := \int_0^1 \sqrt{\inner{\dot \gamma(t)}{\dot\gamma(t)}} \dd t \right\},
	\end{equation*}
	where $\inf$ is taken over all piecewise smooth curves $\gamma:[0,1]\to\cM$  with $\gamma(0)=p$ and $\gamma(1)=q$.
\end{definition}
\begin{definition}
    Let $f: \cM \to \cN$ be a smooth map between smooth manifolds $\cM, \cN$, the \emph{differential of $f$ at $p\in\cM$}, denoted by $\dd f|_p$, is a map from $T_p\cM$ to $T_{f(p)}\cN$ such that $(\dd f|_p\eta) g:= \eta(g \circ f)$ for all $g \in \cD(\cN)$ and $\eta \in T_p\cM$.
\end{definition}
\begin{definition}%
	\label{def:riemannian-gradient}
	Let $f:\cM\to\R$ be a smooth function 
	and $p\in\cM$, the \emph{gradient} of $f$ at $p$ is defined as the unique tangent vector $\grad f(p)\in T_p\cM$ that satisfies 
		 $$\xi_p f = \inner{\xi_p}{\grad f(p)},\ \forall\xi_p \in T_p\cM.$$
\end{definition}
The uniqueness of $\grad f(p)$ follows from the Riesz representation theorem.
Define $T\cM := \bigcup_{p \in \cM} T_p \cM$ to be the \emph{tangent bundle} of $\cM$ and a map $X: \cM \to T\cM$ to be a \emph{vector field} on $\cM$ if $X(p) \in T_p\cM$ for all $p\in\cM$. 

\begin{definition}%
	 For any $X,Y\in\cX(\cM)$, a map $\nabla_XY\in\cX(\cM)$ is called the \emph{Levi-Civita connection} if it is an affine connection\footnote{An affine connection is $\cD(\cM)$-linear w.r.t. $X$, $\R$-linear w.r.t. $X, Y$, and satisfies the product rule, see Chapter~5 in \cite{absil2009optimization}.} and satisfies 
    \begin{equation*}
        X\inner{Y}{Z} = \inner{\nabla_XY}{Z} + \inner{Y}{\nabla_XZ}\quad {\rm and} \quad  \nabla_XY - \nabla_YX = XY - YX, 
    \end{equation*}
    for all $X, Y, Z \in \cX(\cM)$.
\end{definition}
The Levi-Civita connection is unique~\cite{carmo1992riemannian} and can define the parallel transport of a vector field.
\begin{definition}
	 A vector field $X$ is \emph{parallel along a smooth curve $\gamma$} if $\nabla_{\dot \gamma} X = 0$. 
\end{definition}
Given a smooth curve $\gamma$ and $\eta \in T_{\dot \gamma(0)}\cM$, there exists a unique parallel vector field $X_\eta$ along $\gamma$ such that $X_\eta(0) = \eta$. 
We define the \emph{parallel transport along $\gamma$} to be $ P_\gamma^{0 \to t} \eta := X_\eta(t)$. 
A curve $\gamma$ is called a \emph{geodesic} if it is parallel to itself, i.e., $\nabla_{\dot \gamma} \dot \gamma = 0$, which implies $P_\gamma^{0 \to t} \dot\gamma(0) = \dot\gamma(t)$. For given initial conditions $\gamma(0) = p \in \cM$, $\dot\gamma(0) = \eta \in T_p\cM$, the geodesic equation $\nabla_{\dot\gamma}\dot\gamma = 0$ has a solution locally.
Let $V_p$ be the set of $\eta \in T_p\cM$ such that $\gamma$ is a geodesic, $\gamma(0) = p$, $\dot\gamma(0) = \eta$ and $\gamma(1)$ exists.
The \emph{exponential map} $\exp_p: V_p \to \cM$ is defined as $\eta \mapsto \gamma(1)$.
When the geodesic from $p$ to $q$ is unique, denoted by $\gamma_{pq}$, we define $P_{pq} := P_{\gamma_{pq}}^{0\to 1}$.
We highlight that the parallel transport $P_\gamma^{0 \to t}$ is a linear isometry, i.e., $P_\gamma^{0 \to t}$ is linear and $\inner{\xi}{\zeta} = \inner{P_\gamma^{0 \to t} \xi}{P_\gamma^{0 \to t} \zeta}$ for $\xi, \zeta \in T_p \cM$~(see Sec.~5.4 in \cite{absil2009optimization} for details).

The exponential map is not always tractable, e.g., it may be expansive to compute or does not have a closed-form solution (since we need to solve a differential equation).
However, it is possible that the convergence properties of an optimization algorithm remains the same when the exponential map is replaced with its first-order approximation~\cite{absil2009optimization,adler2002newton}.
Such an approximation is called a retraction and defined as follows.
\begin{definition}%
\label{def:retraction}
	A $C^2$ map $R:T\cM\to\cM$ is a \emph{retraction} if  $R_p(0) = p$ and $\frac{\dd}{\dd t} R_p(t\eta)|_{t=0} = \eta$ for all $\eta\in T_p\cM$ and $p \in \cM$, where we denote $R_p := R(p, \cdot)$.
\end{definition}
Generally, we only require that for every $p \in \cM$, $R$ is defined on a neighborhood of $(p, 0) \in T\cM$.

\begin{definition}[\cite{de2018newton,oliveira2020a}]
	\label{def:directional-differentiable}
	Let $X$ be a vector field on $\cM$. The \emph{directional derivative} at $p \in \cM$ along $v\in T_p\cM$ is defined as 
	\begin{equation}
		\label{eqn:directionally-differentiable}
		\nabla X(p; v):= \lim_{t \to 0^+} \frac{1}{t} \big [ P_{\exp_p(tv),p}X(\exp_p(tv)) - X(p) \big ] \in T_p\cM. 
	\end{equation}
\end{definition}
We say $X$ is \emph{directionally differentiable} at $p$ if $\nabla X(p; v)$ exists for all $v \in T_p\cM$.
When $X$ is smooth at $p$, we know $\nabla X(p; v) = \nabla_v X(p)$ (see \cite[p.\ 234]{spivak1999comprehensive}).

\begin{definition}%
    Let $f:\cM\to\R$ be a smooth function. The \emph{Hessian of $f$ at $p\in\cM$}, denoted by $\Hess f(p)$, is defined as a linear operator on $T_p\cM$ such that $\Hess f(p)[v] := \nabla_v \grad f(p)$ for all $v\in T_p(\cM)$.
\end{definition}
We refer readers to~\cite{absil2009optimization,carmo1992riemannian, lee2012intro, lee2018introduction} for more details about manifolds.

\subsection{Nonsmooth Manifold Optimization}
As suggested in~\cite{chen2020proximal}, most nonsmooth manifold optimization algorithms can be classified into three categories: subgradient methods, proximal point algorithms and operator splitting methods.

\subsubsection{Subgradient Methods}
The subgradient methods on manifolds~\cite{borckmans2014riemannian,dirr2007nonsmooth} naturally generalize their Euclidean space counterparts.
Suppose $f: \cM \to \R$ is a locally Lipschitz function\footnote{We say a function $f$ on a manifold is locally Lipschitz if $f \circ \varphi^{-1}$ is locally Lipschitz in $U$ for every chart $(U, \varphi)$.} on a manifold $\cM$ and $(U, \varphi)$ is a chart containing $p \in \cM$. 
From ~\cite{azagra2005nonsmooth,dirr2007nonsmooth,hosseini2011generalized}, the \emph{Clarke generalized directional derivative}  of $f$ at $p$, denoted by $f^\circ(p; v)$, is defined by
\[ f^\circ(p; v) := \limsup_{y \to p, t \downarrow 0} \frac{\hat f(\varphi(y) + t \dd\varphi|_pv) - \hat f(\varphi(y))}{t}, \]
where $\hat f := f \circ \varphi^{-1}$ and $\dd\varphi|_p$ is the differential of $\varphi$ at $p$. The \emph{Clarke subgradient} is 
$$\partial f(p) := \{  \xi \in T_p\cM : \inner{\xi}{v} \leq f^\circ (p; v),\ \forall v \in T_p \cM \}.$$
 Indeed, $f^\circ(p; v)$ is the Clarke directional derivative of $\hat f$ in Euclidean spaces and is independent of the choice of $\varphi$. 
 Moreover, the Clarke subgradient $\partial f(p)$ can be obtained from the Euclidean version as shown in Proposition~3.1 of~\cite{yang2014optimality}:
 \begin{equation} \label{eqn:subgradient-from-euclidean}
	\partial f(p) = (\dd \varphi|_p)^{-1}[G_{\varphi(p)}^{-1}\partial \hat f (\varphi(p))],
 \end{equation}
 where $G_{\varphi(p)} \in \R^{n \times n}$ is the metric matrix such that its $(i, j)$-th element is $g_{ij} := \innerz{(\dd \varphi|_p)^{-1}e_i}{(\dd \varphi|_p)^{-1}e_j}$, 
 where $\{ e_i \}_{i \in [n]}$ is the standard basis of $\R^n$, i.e., the $j$-th component of $e_i$ is $\delta_{ij}$.

 The update rule of the subgradient method in the Riemannian setting~\cite{borckmans2014riemannian,dirr2007nonsmooth} is $p_{k + 1} = \exp_{p_k} (t_k v_k)$, where $v_k \in \partial f(p_k)$ and $t_k$ is the stepsize.
These methods are known to be slow in the Euclidean setting. From the experiments in \cite{chen2020proximal}, it is also observed that subgradient based methods are slower than proximal point algorithms and operator splitting methods in the Riemannian setting.

\subsubsection{Proximal Point Methods} \label{sec:background-ppa}
The extension of proximal point algorithms on manifolds is proposed in~\cite{ferreira2002proximal} and the subgradient methods are suggested in~\cite{bacak2016second} to solve the subproblem. In~\cite{ferreira2002proximal}, manifolds with non-positive sectional curvature are considered, which exclude many important applications such as optimization problems on the Stiefel manifold.
Very recently, Chen et al. proposed the proximal gradient method (ManPG) on the Stiefel manifold~\cite{chen2020proximal} with proved convergence. More specifically, it aims at solving the following problem:
\begin{equation}
	\label{eqn:manpg-problem}
	\min_{Q \in \cM}\ \left \{ f(Q) + \psi(Q) \right \},
\end{equation} 
where $\cM = \St{n, r}$ is the Stiefel manifold, $f$ is smooth with Lipschitz gradient, and $\psi$ is convex and Lipschitz.
In each step, the descent direction is determined by solving the subproblem
\begin{equation}
	\label{eqn:manpg-subproblem}
	V_k := \argmin_{V \in T_{Q_k} \cM} \left \{ \inner{\grad f(Q_k)}{V} + \frac{1}{2t}\norm{V}_F^2 + \psi(Q_k + V) \right \}
\end{equation} 
via the regularized semismooth Newton method~\cite{xiao2018regularized}. 
Besides, Huang and Wei extended an accelerated version of the proximal gradient method to manifolds~\cite{huang2019extending}. 
They also proposed a Riemannian proximal gradient method~\cite{huang2019riemannian} to solve~\eqref{eqn:manpg-problem} for general manifolds by replacing the term $\psi(Q_k+V)$ with $\psi(R_{Q_k}(V))$, where $R_{Q_k}$ is a retraction, and analyzed the iteration complexity for convex objectives under some assumptions. However, the direct application of the above three methods for solving \eqref{eqn:original-problem-0} has to deal with the subproblem:
\begin{equation*}
    V_k := \argmin_{V \in T_{Q_k} \cM} 
	\Big \{ \inner{\grad f(Q_k)}{V} + \frac{1}{2t}\norm{V}_F^2 + \underbrace{\psi(Q_k+V) + \delta_{\cC}(Q_k+V)}_{(\text{ or } \psi(R_{Q_k}(V)) + \delta_{\cC}(R_{Q_k}(V)))} \Big \},
\end{equation*}
where $\delta_{\cC}$ is the indicator function of the feasible set corresponding to the inequality constraints in \eqref{eqn:original-problem-0}. In general, the above problem is difficult due to the presence of two nonsmooth terms.

It is worth mentioning that \eqref{eqn:manpg-problem} may be solved in a more efficient way when $\cM$ has specific structures, e.g., $\cM=\cM_1\times\cM_2$ where $\cM_1$ and $\cM_2$ are two manifolds.
Chen et al.~\cite{chen2019alternating} proposed an alternating manifold proximal gradient algorithm to solve \eqref{eqn:manpg-problem}, which alternatively updates the variables in Gauss-Seidel fashion based on the linearized formulation~\eqref{eqn:manpg-subproblem}. It is empirically observed that this effective alternating strategy leads to better performance than ManPG~\cite{chen2019alternating}. Exploring the special structure of the manifold is a promising direction, and we leave it as our future work.

\subsubsection{Operator Splitting Methods}
Operator splitting methods on manifolds split \eqref{eqn:manpg-problem} into several terms, each of which is easier to solve. For example, the manifold ADMM proposed in~\cite{kovnatsky2016madmm} rewrites~\eqref{eqn:manpg-problem} to
\begin{equation}
	\label{eqn:madmm-problem}
	\min_{Q, Z}\ \left \{ f(Q) + \psi(Z) \right \} \quad \mathrm{s.t.} \quad Q = Z,\ Q \in \cM.
\end{equation} 
Then, a two-block ADMM is used to solve it, which has the following update rules:
\[ \begin{aligned}
	Q_{k + 1} & := \argmin_{Q \in \cM} \left \{ f(Q) + \frac{\rho}{2} \norm{Q - Z_k + U_k}_F^2 \right \} ,\\  
	Z_{k + 1} &:= \argmin_{Z} \left \{ \psi(Z) + \frac{\rho}{2} \norm{Q_{k+1} - Z + U_k}_F^2 \right \}, \\
	U_{k+1} & := U_k + Q_{k+1} -Z_{k+1}.
\end{aligned} \]
The $Q$-update requires smooth manifold optimization algorithms and the $Z$-update is the proximal mapping of $\psi$. Besides, an inexact ALM framework to solve~\eqref{eqn:madmm-problem} with some convergence results is considered in~\cite{kangkang2019inexact}.

When the manifold $\cM$ can be  embedded in a Euclidean space, classical nonsmooth nonconvex constrained optimization algorithms can also be explored. 
Lai et al. proposed a splitting method for orthogonality constrained problems (SOC)~\cite{lai2014splitting}, which reformulates~\eqref{eqn:manpg-problem} into
\begin{equation}
	\label{eqn:soc-problem-0}
	\min_{Q, P, R} \left \{ f(P) + \psi(R) \right \} \quad \mathrm{s.t.} \quad P = R,\ Q = P,\ Q^\trans Q = I_r.
\end{equation} 
A three-block ADMM is then used to solve the above problem: 
\[ \begin{aligned}
	P_{k + 1} & := \argmin_{P \in \R^{n\times r}} \left \{ f(P) + \frac{\rho}{2} \norm{P - R_k + \Lambda_k}_F^2
	+ \frac{\rho}{2} \norm{P - Q_k + \Gamma_k}_F^2 \right \}
	, \\
	R_{k + 1} & := \argmin_{R \in \R^{n\times r}} \left \{ \psi(R)
	+ \frac{\rho}{2} \norm{P_{k+1} - R + \Lambda_k}_F^2 \right \}, \\
	Q_{k + 1} & := \argmin_{Q \in \R^{n \times r}}\ \frac{\rho}{2} \norm{P_{k+1} - Q + \Gamma_k}_F^2 \quad \mathrm{s.t.} \quad Q^\trans Q = I_r, \\
	\Lambda_{k+1} & := \Lambda_k + P_{k+1} -R_{k+1}, \quad
	\Gamma_{k+1}  := \Gamma_k + P_{k+1} -Q_{k+1},
\end{aligned} \]
where the $R$-update can be solved by a proximal map, and the $Q$-update has the closed form solution, and the $P$-update can be solved using gradient based methods. 
Although these ADMM-type methods are simple, to the best of our knowledge, 
it is unclear whether they converge to a KKT point of~\eqref{eqn:soc-problem-0}.

Unlike ADMM, ALM usually has theoretical guarantees. In~\cite{chen2016augmented}, Chen et al. proposed the proximal alternating minimized augmented Lagrangian method (PAMAL), which solves the augmented Lagrangian subproblem by the proximal alternating minimization (PAM) scheme~\cite{attouch2010proximal}. In~\cite{zhu2017nonconvex}, the so-called EPALMAL is proposed, where the PALM~\cite{bolte2014proximal} is used for solving the subproblem. Although both PAMAL and EPALMAL have certain convergence guarantee, the analysis is not complete for solving~\eqref{eqn:original-problem-0}.

\subsection{Augmented Lagrangian Methods for Nonsmooth and Nonconvex Problems}
Here, we review some augmented Lagrangian methods in constrained optimization,
which has been studied for many decades~\cite{bertsekas2014constrained}.
We consider the following problem:
\begin{equation}
	\label{eqn:alm-problem}
	\min_{x \in \R^n}\ \{ f(x) + \Phi(x) \} \quad \mathrm{s.t.} \quad g(x) = 0,\ h(x) \leq 0,
\end{equation} 
where $f, g, h$ are smooth and $\Phi$ is lower semicontinuous. It is noted that the original problem~\eqref{eqn:original-problem-0} has the above form when the manifold $\cM$ can be written as
\begin{equation}
\cM = \{x : g_1(x) = 0,\ h_1(x)\leq0\},
\end{equation}
where $g_1$ and $h_1$ are parts of $g$ and $h$, respectively.
When $\Phi \equiv 0$ and the constraints can be divided into $g_1(x) = 0$, $g_2(x) = 0$, $h_1(x) \leq 0$ and $h_2(x) \leq 0$ such that the minimization problem is easier on $\{ x  : g_2(x) = 0,\ h_2(x) \leq 0 \}$, Anderani et al. proposed an Augmented Lagrangian (AL) method~\cite{andreani2008augmented} to solve \eqref{eqn:alm-problem}.
It is shown that any feasible limit point generated by the algorithm is a KKT point under the constant positive linear dependence (CPLD) condition~\cite{qi2000constant}, which is weaker than the linear independence constraint qualification (LICQ) condition. 
However, this method cannot guarantee that any limit point is feasible when the penalty parameter, 
i.e., the coefficient of the quadratic term in the augmented Lagrangian function, tends to infinity as the iteration proceeds (see Theorem 4.1(i) in \cite{andreani2008augmented}).
This infeasiblity phenomenon also exists in other literature such as~\cite{curtis2015adaptive}.

To alleviate this issue, another AL method is proposed in~\cite{lu2012spca}, where two nonmonotone proximal methods are applied to solve the subproblem.
They consider the problem \eqref{eqn:alm-problem} with the additional assumption that $\Phi$ is a convex function.
A feasible point is assumed to be known, and is used to guarantee that the augmented Lagrangian function is uniformly bounded from above at points generated in subproblems. 
Besides, the method in~\cite{lu2012spca} modifies the update rule of the penalty parameter to ensure that the penalty grows faster than Lagrangian multipliers (see, e.g., \eqref{eqn:alm-update-penalty} in Algorithm~\ref{alg:alm-outer}).
Using these two properties, the convergence result that any limit point is a KKT point under Robinson’s constraint qualification is established. Recently, Chen et al.~\cite{chen2017augmented} proposed an AL method to solve~\eqref{eqn:alm-problem} with $\Phi$ possibly being a nonconvex non-Lipschitz function. Under a weak constraint qualification called the relaxed constant positive linear dependence (RCPLD) condition~\cite{andreani2012relaxed}, they provided a global convergence result.

\subsection{Semismooth Newton Methods} \label{sec:background-semismooth}
The subproblem in ALM generally requires tackling a nonsmooth equation, which usually can be efficiently solved by the semismooth Newton method~\cite{mifflin1977semismooth,qi1993nonsmooth,sun2002semismooth}. 
Under suitable assumptions, the semismooth Newton method has the local superlinear convergence rate.
Recently, this method is generalized to solving nonsmooth equations~\cite{de2018newton} on manifolds based on the Clarke generalized covariant derivatives~\cite{ghahraei2017pseudo,rampazzo2007commutators}. 
Below we introduce several definitions for locally Lipschitz vector fields on manifolds.
\begin{definition}[\cite{de2018newton}]
	\label{def:lipschitz}
	Let $L > 0$ and $\Omega \subset \cM$ be given.
	We say a vector field $X: \cM \to T\cM$ on a manifold $\cM$ is 
	\emph{$L$-Lipschitz in $\Omega$} if for each $x, y \in \Omega$, and each geodesic $\gamma$ joining $x, y$, it holds 
	\[ \norm{P_\gamma^{0\to 1} X(x) - X(y)} \leq L \ell(\gamma),  \]
	and we say that $X$ is \emph{locally Lipschitz at $p \in \cM$} if there exist a neighborhood $U_p \ni p$ and a constant $L_p > 0$ such that $X$ is $L_p$-Lipschitz in $U_p$.
	If $X$ is locally Lipschitz at every $p \in \cM$, we say that $X$ is \emph{locally Lipschitz on $\cM$}.
\end{definition}
Since a locally Lipschitz vector field $X$ on manifolds is differentiable almost everywhere~\cite{de2018newton}, we denote $\cD_X$ as the set of its differentiable points and define the \emph{Clarke generalized covariant derivative} as follows:
\begin{definition}[\cite{de2018newton,ghahraei2017pseudo}]
	\label{def:clarke-generalized-covariant}
	Let $X$ be a locally Lipschitz vector field on $\cM$. The \emph{B-derivative} is a set-valued map $\partial_B X: \cM \rightrightarrows \cL(T\cM)$ with 
	\begin{equation}
	\begin{aligned}
	 \partial_B X(p)
	:=  \left\{ H\in {\mathcal{L}}(T_p\cM):~ \exists\, \{p_k\}\subset {\cal D}_X,~ \lim_{k\to +\infty}p_k =p,~ H = \lim_{k\rightarrow +\infty} \nabla X(p_k)  \right\},
	\end{aligned}
	\end{equation}
	where the last limit means that $\norm{\nabla X(p_k)[P_{pp_k}v] - P_{pp_k}Hv} \to 0$ for all $v \in T_p \cM$.
	The \emph{Clarke generalized covariant derivative} is a set-valued map~$\partial X: \cM \rightrightarrows \cL(T\cM)$ such that $\partial X(p)$ is the convex hull of $\partial_B X(p)$.
\end{definition}
The above definitions are consistent with those in $\R^n$ as the tangent spaces can be identified with $\R^n$ so $P_{p_kp}$ and $P_{pp_k}$ are the identical mappings. The properties of the Clarke generalized covariant derivative are similar to those in Euclidean spaces. For example, $\partial_B X(p)$ and $\partial X(p)$ are non-empty compact sets and the maps $\partial_B X, \partial X$ are locally bounded and upper semicontinuous~\cite[Proposition 3.1]{de2018newton}. Having introduced these notions, the Newton method for a locally Lipschitz vector field $X$ on $\cM$ is~\cite{de2018newton}:
\begin{equation}\label{Newton:iter}
p_{k + 1} := \exp_{p_k}(-H_k^{-1} X(p_k)), \quad
\text{where } H_k \in \partial X(p_k).
\end{equation}
To obtain the convergence rate, we have to impose the semismooth property of the vector field $X$.
\begin{definition}[\cite{de2018newton}]
	\label{def:semismooth-clarke}
	Let $X$ be a locally Lipschitz vector field on $\cM$. 
	We say $X$ is \emph{semismooth with order $\mu$ at $p \in \cM$} if 
	it is directionally differentiable in a neighborhood $U$ of $p$, and
	 there exist $C >0, \delta > 0$ such that %
	\begin{equation}
		\label{eqn:semismooth-def}
		\norm{X(p) - P_{qp}[X(q) + H_q \exp^{-1}_qp] }\leq C d(p, q)^{1 + \mu},\ \forall\, q \in B_\delta(p), H_q \in \partial X(q),
	\end{equation}
	where $B_\delta(p) := \{ q \in \cM : d(p, q) < \delta \}$ and $\exp^{-1}_q p$ is the inverse of the exponential map\footnote{This is well-defined in a small neighborhood of $q$~\cite[Proposition 3.2.9]{carmo1992riemannian}.}.
\end{definition}
In~\cite{de2018newton}, it is shown that if $X$ is locally Lipschitz, $X(p_*) = 0$, all elements in $\partial X(p_*)$ are nonsingular and $X$ is semismooth at $p_*$ with order $\mu$, then the Newton iteration~\eqref{Newton:iter} has the local convergence rate $1 + \mu$. This result is similar to that in Euclidean spaces~\cite{mifflin1977semismooth,qi1993nonsmooth,sun2002semismooth}.

\section{An Augmented Lagrangian Framework} \label{sec:alm}
In this section, we present an augmented Lagrangian method to solve~\eqref{eqn:original-problem-0} and establish its convergence result. The method for solving the subproblem is deferred to the next section.
\subsection{Algorithm}
Recall that we consider the following optimization problem:
\begin{equation}
	\label{eqn:original-problem}
	\begin{aligned}
		\min_{x}\  &\{f(x) + \psi(h_1(x)) \}, \quad 
		\mathrm{s.t.}\  x \in \cM,\ h_2(x) \leq 0.
	\end{aligned}
\end{equation}

Throughout this paper, we always make the following assumptions:
\begin{assumption}
	$\cM$ is a complete smooth Riemannian manifold.
\end{assumption}
\begin{assumption}
	$f : \cM \to \R$, $h_1: \cM \to \R^m$, $h_2: \cM \to \R^{q}$ are continuously differentiable, and~$\psi: \R^m \to \R$ is a convex function.
	$f(x) + \psi(y)$ is bounded below for $(x, y) \in \cM \times \R^m$.
\end{assumption}

Note that we can reformulate \eqref{eqn:original-problem} to the following problem:
\begin{equation}
	\label{eqn:reformulated-problem}
	\begin{aligned}
		\min_{x, y, z}\  &\{f(x) + \psi(y) \}, \quad
		\mathrm{s.t.}\  x \in \cM, \
		 y = h_1(x), \
		 z = h_2(x),\ z \leq 0.
	\end{aligned}
\end{equation}
The augmented Lagrangian function $L_\sigma: \cM \times \R^m \times \R_-^q \times \R^m \times \R^q \to \R$ of \eqref{eqn:reformulated-problem} is given by
\begin{equation}
	\begin{aligned} 
\peq	L_\sigma(x, y, z, \lambda, \gamma)  
= f(x) + \psi(y) 
+ \frac{\sigma}{2} \norm{h_1(x) - y + \frac{\lambda}{\sigma}}^2_2 
+ \frac{\sigma}{2}\norm{h_2(x) - z + \frac{\gamma}{\sigma}}_2^2 
- \frac{\norm{\lambda}_2^2 + \norm{\gamma}_2^2}{2\sigma}.
	\end{aligned}
	\label{eqn:L_sigma}
\end{equation}
We note that simultaneously minimizing $L_\sigma$ with respect to $x, y, z$ is equivalent to
\begin{equation}
	\label{eqn:minimizing_L-x}
	 \min_{x \in \cM} \left \{
    f(x) + \psi^\sigma\left (h_1(x) + \frac{\lambda}{\sigma}\right) + \delta_{\R^q_-}^\sigma\left ( h_2(x) + \frac{\gamma}{\sigma} \right)
	\right \}
	,
\end{equation}
where $\psi^\sigma$, $\delta_{\R^q_-}^\sigma$ are the \emph{Moreau-Yosida regularization} of $\psi$, $\delta_{\R^q_-}$, respectively. 
More specifically, it holds that
\begin{align}
    \psi^\sigma(x) &:= \min_{y \in \R^m}  \left \{ \psi(y)  
    + \frac{\sigma}{2} \norm{x - y}^2_2 \right \}, \label{prob:y}\\
    \delta_{\R_-^q}^\sigma(x) &:= \min_{z \in \R^q} \left \{ \delta_{\R_-^q}(z) + \frac{\sigma}{2}\norm{x - z}_2^2 \right \}, \label{prob:z}
\end{align}
where $\delta_{\R_-^q}$ is the indicator function of $\R_-^q$. 
Since the above two functions are crucial in developing our algorithm, we present some important properties.
\begin{proposition}[Theorem 4.1.4 in \cite{hiriart1993convex}]
	\label{prop:moreau-yosida}
	Let $\sigma > 0$ and $f: \R^n \to \R \cup \{ +\infty \}$ be a closed proper convex function, then the Moreau-Yosida regularization $f^\sigma$ is continuously differentiable, and its gradient is 
	\begin{equation}
		\label{eqn:moreau-yosida-gradient}
		\nabla f^\sigma(x) = \frac{1}{\sigma}(x - \prox_{f/\sigma}(x)),
	\end{equation}
	where $\prox_{f / \sigma}$ is the \emph{proximal map} of the function $f / \sigma$:
	 \begin{equation}
		\prox_{f / \sigma} (x) := \argmin_{y} \left \{ f(y) + \frac{\sigma}{2}\norm{x-y}_2^2 \right \}.
	\end{equation}
\end{proposition}

The minimization problems \eqref{prob:y} and \eqref{prob:z} are related to finding the proximal maps $\prox_{\psi / \sigma}$ and $\prox_{\delta_{\R_-^q}}$, which can be easily solved in many cases. 
In addition, since $\delta^\sigma_{\R^q_-}$ and $\psi^\sigma$ are continuously differentiable, \eqref{eqn:minimizing_L-x} is a smooth optimization problem on manifolds. 
These observations suggest the following augmented Lagrangian method, whose framework is similar to \cite{chen2017augmented,lu2012spca}.
\begin{alg}[An augmented Lagrangian method for solving \eqref{eqn:reformulated-problem}]
	\label{alg:alm-outer}
	Choose initial values $x_0 \in \cM$, $\gamma_0 \in \R^{q}_+$, $\lambda_0 \in \R^m$, $\sigma_0 > 0$, $\alpha, \tau \in (0, 1)$, $\rho > 1$ 
	and a sequence $\{ \varepsilon_k \} \subseteq \R_+$ converging to $0$.
	Let $y_0 = \prox_{\psi/\sigma_0} ( h_1(x_0) + \lambda_0 / \sigma_0)$, $z_0 = \Pi_{\R_-^q} ( h_2(x_0) + \gamma_0 / \sigma_0)$, where $\Pi_{\R_-^q}$ is the projection onto $\R_-^q$. %
	Choose a feasible point $x_\feas$ and a constant $\Phi$ such that
	\begin{equation}
		\label{eqn:initial-maximal}
		\Phi \geq \max\{  f(x_\feas) + \psi(h_1(x_\feas)),\ L_{\sigma_0}(x_0, y_0, z_0, \lambda_0, \gamma_0)  \}.
	\end{equation}
	Our algorithm repeats the following steps for $k = 1, 2, \dots$

    \begin{enumerate}[label=(\roman*)]
	\item
	Find $x_k\in\cM$ such that
	\begin{equation}
		\label{eqn:subproblem-stopping-rule}
		\norm{\grad L_k (x_k)} < \varepsilon_k, \quad L_k(x_k) \leq \Phi,
	\end{equation}
	where 
	\begin{equation}
	L_k(x) := f(x) 
	+ \psi^{\sigma_k} \left ( h_1(x) + \frac{\lambda_k}{\sigma_k}  \right )
	+ \delta_{\R_-^q}^{\sigma_k} \left ( h_2(x) + \frac{\gamma_k}{\sigma_k}  \right ).
	\label{eqn:Lk}
	\end{equation}
	\item Update $y$ and $z$ using
    \begin{align}
        \label{eqn:subproblem-y}
        y_k &= \prox_{\psi/\sigma_k}\left ( h_1(x_k) + \frac{\lambda_k}{\sigma_k}  \right ),\\
        \label{eqn:subproblem-z}
        z_k &= \Pi_{\R_-^q} \left ( h_2(x_k) +  \frac{\gamma_k}{ \sigma_k} \right).
    \end{align}
        
	\item %
	Update the multipliers:
    \[ \begin{aligned}
        \lambda_{k+1} &= \lambda_k + \sigma_k (h_1(x_k) - y_k), \quad 
        \gamma_{k+1} = \gamma_k +\sigma_k (h_2(x_k) - z_k).
    \end{aligned} \]

	\item
	Let 
        $     \delta_k = \max\left\{ \norm{h_1(x_k) - y_k}_2,\ \norm{h_2(x_k)-z_k}_2\right\}.
        $ %

		 If $\delta_k \leq \tau \delta_{k-1}$, then $\sigma_{k+1} = \sigma_k$.
		 Otherwise, set
         \begin{equation}
            \label{eqn:alm-update-penalty}
			\sigma_{k+1} = \max \left \{ \rho\sigma_k,\ \norm{\lambda_{k+1}}_2^{1+\alpha},\ \norm{\gamma_{k+1}}_2^{1+\alpha} \right \}. 
		 \end{equation}
	\end{enumerate}
\end{alg}

\subsection{Convergence Analysis}
In this part, motivated by the proof in~\cite{chen2017augmented}, we first give the feasibility result of Algorithm~\ref{alg:alm-outer}.
\begin{theorem}
	\label{thm:feasibility}
	Let $\{ (x_k, y_k, z_k) \}$ be the sequence generated by Algorithm~\ref{alg:alm-outer}.
	Then, we have
	\[ \lim_{k\to \infty} \left( \norm{h_1(x_k) - y_k}_2 + \norm{h_2(x_k) - z_k}_2 \right) = 0. \]
	Consequently, if $x_* \in \cM$ is an accumulation point of $\{ x_k \}$, then $(x_*, h_1(x_*), h_2(x_*))$ is a feasible accumulation point of $\left\{ (x_k, y_k, z_k) \right\}$. %
	Moreover, $\{ x_k \}$ always contains an accumulation point if $\cM$ is compact.
\end{theorem}
\begin{proof}
    First, consider the case where $\{ \sigma_k \}$ is bounded. There exists $k_1 \in \N$ such that $\delta_{k+1} \leq \tau \delta_k$ for any $k \geq k_1$.
    Then, $\lim_{k \to \infty} \delta_k = 0$ since $\tau < 1$. 
    By the definition of $\delta_k$, we know $\lim_{k \to \infty} \norm{h_1(x_k) - y_k}_2 = 0$ and $\lim_{k\to \infty} \norm{h_2(x_k) - z_k}_2 = 0$. 

	In the case where $\{ \sigma_k \}$ is unbounded. By the update rule, $\sigma_k$ is updated for infinitely many times. Then, we can find $k_1 < k_2 < \cdots$ such that
	\[ \begin{aligned}
			\sigma_k = \sigma_{k_i} &= \max\big\{ \rho \sigma_{k_{i-1}}, \norm{\lambda_{k_i}}_2^{1 + \alpha}, \norm{\gamma_{k_i}}_2^{1 + \alpha} \big\},\ \forall \, k_i \leq k < k_{i+1}.
	\end{aligned} \]

	From \eqref{eqn:subproblem-stopping-rule} and the definition of $L_k(x)$, we know that $L_{\sigma_{k}}(x_{k}, y_{k}, z_{k}, \lambda_{k}, \gamma_{k}) \leq \Phi$, where $L_{\sigma_k}$ is defined in \eqref{eqn:L_sigma}. Therefore, we have
	\begin{equation}
		\norm{h_1(x_{k}) - y_{k} + \frac{\lambda_{k}}{\sigma_{k}}}^2_2 
		+ \norm{h_2(x_{k}) - z_{k} + \frac{\gamma_{k}}{\sigma_{k}}}_2^2 
		\leq 2\frac{ \Phi - f(x_{k}) - \psi(y_{k}) }{\sigma_{k}} 
		+ \frac{\norm{\lambda_{k}}_2^2 + \norm{\gamma_{k}}_2^2}{\sigma_{k}^2}.
		\label{eqn:feasibility-bounded}
	\end{equation}

	Since $f + \psi$ is bounded below and $\sigma_{k} \to \infty$, then the first term of \eqref{eqn:feasibility-bounded} converges to $0$. Next, we show that the second term also converges to $0$.
	Notice that $\norm{\lambda_{k_i}}_2^{1+\alpha} \leq \sigma_{k_i}$ and $\sigma_{k_i} \to \infty$, 
	we have $\norm{\lambda_{k_i}}_2 / \sigma_{k_i} \leq \sigma_{k_i}^{-\frac{\alpha}{1 + \alpha}} \to 0$ as $i \to \infty$. A similar result also holds for $\norm{\gamma_{k_i}}_2 / \sigma_{k_i}$. 
	Then by \eqref{eqn:feasibility-bounded} and the definition of $\delta_{k_i}$, we have $\lim_{i \to \infty} \delta_{k_i} = 0$.
	By the update rule of $\lambda_k$, for $k_i < k < k_{i+1}$, it holds that 
	\[\begin{aligned}
			\frac{\norm{\lambda_{k}}_2}{\sigma_{k}} 
			&\leq \frac{\norm{\lambda_{k - 1}}_2}{\sigma_{k_i}} + \norm{h_1(x_{k - 1}) - y_{k - 1}}_2 
			\leq \frac{\norm{\lambda_{k - 1}}_2}{\sigma_{k_i}} + \delta_{k-1}.
	\end{aligned} \]

	Also note that $\delta_k \leq \tau\delta_{k-1}$ for all $k_i \leq k < k_{i+1} - 1$. Then, for $k_i < k < k_{i+1}$ the following inequality holds by induction
	\[\begin{aligned}
			\frac{\norm{\lambda_{k}}_2}{\sigma_{k}} 
			&\leq \frac{\norm{\lambda_{k_i}}_2}{\sigma_{k_i}} + \sum_{j=k_i}^{k-1} \delta_j
			\leq \frac{\norm{\lambda_{k_i}}_2}{\sigma_{k_i}} + \delta_{k_i} \sum_{j=0}^{k - k_i -1} \tau^{j}
			\leq \frac{\norm{\lambda_{k_i}}_2}{\sigma_{k_i}} + \frac{\delta_{k_i}}{1 - \tau}.
	\end{aligned} \]

	Thus, we have $\lim_{k \to \infty} \norm{\lambda_k}_2 / \sigma_k = 0$.
	Similarly, $\lim_{k \to \infty} \norm{\gamma_k}_2 / \sigma_k = 0$. Therefore, from \eqref{eqn:feasibility-bounded}, we conclude that $\lim_{k \to \infty} \delta_k = 0$. \hfill $\Box$
\end{proof}

Next, we consider the convergence result of Algorithm~\ref{alg:alm-outer} by introducing  an extension of the constraint qualifications on manifolds~\cite{yang2014optimality}.
Consider the problem \eqref{eqn:original-problem}, we define the \emph{active set} of a feasible point $x$ to be
$\cA(x) := \{ i \in [q] : [h_2(x)]_i = 0 \}$.
The following constraint qualification can be introduced~\cite{yang2014optimality}:
\begin{definition}[LICQ]
	\label{def:LICQ}
	We say that a feasible point $x \in \cM$ of \eqref{eqn:original-problem} satisfies the \emph{linear independence constraint qualification} (LICQ) if
    $\{ \grad\, [h_2(x)]_i : i \in \cA(x) \}$ are linearly independent in $T_x\cM$.
\end{definition}

The above definition is the same as that in the Euclidean case except that Euclidean gradients are replaced by Riemannian gradients. 
Indeed, there is a weaker constraint qualification:
\begin{definition}[CPLD]
	\label{def:CPLD}
	Let $x \in \cM$ be a feasible point of \eqref{eqn:original-problem} and define
    $S(x) := \{ \grad\, [h_2(x)]_i : i \in \cA(x) \}$.
	We say that $x$ satisfies the \emph{constant positive linear dependence constraint qualification} (CPLD) 
	if for each subset $S_0(x) \subseteq S(x)$ whose elements are linearly dependent with non-negative coefficients, $S_0$ remains linearly dependent in a neighborhood of $x$. 
\end{definition}

The first-order optimality condition of \eqref{eqn:original-problem} can be stated as follows using the LICQ condition, which is a direct consequence of \cite[Theorem 4.1]{yang2014optimality}.
\begin{corollary}
	\label{thm:first-order-optimality}
	Define the Lagrangian of \eqref{eqn:original-problem} as
	\[ \cL(x, \gamma) := f(x) + \psi(h_1(x)) + \gamma^\trans h_2(x), \quad x \in \cM,\ \gamma \geq 0.\]
    Suppose $x_*$ is a local minimum of \eqref{eqn:original-problem} and the LICQ holds at $x_*$, then there exists a multiplier $\gamma_*$ such that the following KKT conditions hold:
    \begin{subequations}
        \label{eqn:original-kkt}
        \begin{align}
        & 0 \in \partial_x \cL(x_*, \gamma_*), \label{eqn:original-kkt-x} \\
        & h_2(x_*) \in \R^q_-,\ \gamma_* \in \R^q_+,\ \gamma_*^\trans h_2(x_*) = 0.
        \end{align}
    \end{subequations}
\end{corollary}
\begin{remark}
	 \label{remark:optimality}
The Lagrangian of the equivalent problem \eqref{eqn:reformulated-problem} is
\[ \tilde \cL(x, y, z, \lambda, \gamma) = f(x) + \psi(y) + \lambda^\trans(h_1(x) - y) + \gamma^\trans(h_2(x) - z). \]
Under the LICQ condition, a necessary optimality condition is the following KKT system~\cite{yang2014optimality}:
    \begin{subequations}
        \label{eqn:reformulated-kkt}
        \begin{align}
			 y = h_1(x),\ z = h_2(x), \label{eqn:reformulated-kkt-feas} \\
			 z \in \R^q_-,\ 
            \gamma \in \R^q_+,\ \gamma^\trans z = 0, \label{eqn:reformulated-kkt-z} \\
            0 \in \partial_y \tilde \cL(x, y, z, \lambda, \gamma), \label{eqn:reformulated-kkt-y} \\
            \mathrm{grad}_x\, \tilde \cL(x, y, z, \lambda, \gamma) = 0. \label{eqn:reformulated-kkt-x}
        \end{align}
	\end{subequations}
\end{remark}

Before presenting the relationship of these two optimality conditions, we give the following chain rule.
\begin{lemma}
	\label{remark:chain-rule}
	Let $p \in \cM$, and $h:\cM\to\R^m$ be a smooth map, $\psi:\R^m\to \R$ be a convex function. Then,
	\begin{equation} \label{eqn:riemannian-chain-rule}
	\partial (\psi \circ h)(p)
	= \left \{    \sum_{i=1}^m \alpha_i \grad\,[h(p)]_i  : \alpha \in \partial \psi(h(p))   \right \}. 
	\end{equation} 
\end{lemma}
\begin{proof}
	For any point $p\in\cM$ and a fixed chart $(U_p, \varphi)$ at $p$, the map $h\circ \varphi^{-1}:\varphi(U_p) \subset \R^n\to\R^m$ is differentiable at $p$, and hence the chain rule in Euclidean spaces (e.g., Theorem 2.3.9 in~\cite{clarke1990optimization}) implies
	\[ \partial (\psi \circ \hat h)(\varphi(p))
	=  \mathrm{co} \left \{    \sum_{i=1}^m \alpha_i \xi_i  : \alpha \in \partial \psi(y)   \right \} = \left \{    \sum_{i=1}^m \alpha_i \xi_i  : \alpha \in \partial \psi(y)   \right \}, \]
	where $\hat h := h \circ \varphi^{-1}$, $y := h(p)$, $\xi_i := \nabla \hat h_i(\varphi(p))$ and ``co'' denotes the convex hull, and the last equality is from the fact that $\partial \psi(y)$ is a compact convex set.
	Note that from \cite[p.~46]{absil2009optimization}, $(\dd \varphi|_{p})^{-1} G^{-1} \xi_i = \grad\,[h (p)]_i$, where $G$ is the metric matrix defined around~\eqref{eqn:subgradient-from-euclidean}.
	Combining with \eqref{eqn:subgradient-from-euclidean}, we know \eqref{eqn:riemannian-chain-rule} holds.
	\hfill $\Box$
\end{proof}
Using the chain rule, the following proposition shows that the two optimality conditions are equivalent. 
\begin{proposition}
	\label{prop:equiv-kkt}
	Conditions \eqref{eqn:original-kkt} and \eqref{eqn:reformulated-kkt} are equivalent: 
	\begin{enumerate}[(i)]
	    \item If $(x,\gamma)$ satisfies \eqref{eqn:original-kkt}, then there exist $y,z,\lambda$ such that $(x,y,z,\lambda,\gamma)$ satisfies \eqref{eqn:reformulated-kkt}; 
	    \item If $(x,y,z,\lambda,\gamma)$ satisfies \eqref{eqn:reformulated-kkt}, then $(x,\gamma)$ satisfies \eqref{eqn:original-kkt}.
	\end{enumerate}
\end{proposition}
\begin{proof}

	When the conditions in \eqref{eqn:original-kkt} hold, we set $y=h_1(x)$, $z=h_2(x)$, and choose $g \in \partial (\psi \circ h_1)(x)$ such that $\grad f(x) + g + \sum_{i=1}^q [\gamma]_i \grad\, [h_2(x)]_i = 0$.
	By \eqref{eqn:riemannian-chain-rule}, there exists $\alpha \in \partial \psi(y)$ such that $g = \sum_{i=1}^m \alpha_i \grad\,[h_1(x)]_i$.
	Therefore, \eqref{eqn:reformulated-kkt-x} holds by setting $\lambda=\alpha$.
	As \eqref{eqn:reformulated-kkt-y} is equivalent to $\lambda \in \partial \psi(y)$, it also holds by the choice of $\lambda$.
	Other conditions directly follow from \eqref{eqn:original-kkt}.

	Conversely, when the conditions in \eqref{eqn:reformulated-kkt} hold, \eqref{eqn:reformulated-kkt-y} and \eqref{eqn:riemannian-chain-rule} imply that $\lambda \in \partial \psi(y)$ and $\sum_{i=1}^m \lambda_i \grad\,[h_1(x)]_i \in \partial (\psi \circ h_1)(x)$. 
	Combining with \eqref{eqn:reformulated-kkt-x}, the condition \eqref{eqn:original-kkt-x} holds. \hfill $\Box$
\end{proof}

Finally, we show that Algorithm~\ref{alg:alm-outer} converges to a KKT point.

\begin{theorem}
    \label{thm:kkt}
    Suppose there exist $K \subseteq \N$ and $(x_*, y_*, z_*) \in \cM \times \R^m \times \R^q_-$ such that
    \[ \lim_{K \ni k \to \infty} \left ( d(x_k, x_*) + \norm{y_k - y_*}_2 + \norm{z_k - z_*}_2 \right ) = 0. \]
    If the CPLD condition holds at $x_*$, then there exist $K_0 \subseteq K$ and $\lambda_* \in \R^m, \gamma_* \in \R^q_+$ such that $\lim_{K_0 \ni k \to \infty} \lambda_{k+1} = \lambda_*$, and the KKT conditions \eqref{eqn:reformulated-kkt} hold at $(x_*, y_*, z_*, \lambda_*, \gamma_*)$. Moreover, when the LICQ holds at $x_*$, we can choose $\gamma_*$ such that $\lim_{K_0 \ni k \to \infty} \gamma_{k + 1} = \gamma_*$.
\end{theorem}
\begin{proof}
    First, from Theorem~\ref{thm:feasibility}, we obtain the feasibility condition \eqref{eqn:reformulated-kkt-feas}.
    From \eqref{eqn:subproblem-y}, \eqref{eqn:subproblem-z} and the property of Moreau-Yosida regularizations in Proposition~\ref{prop:moreau-yosida}, we know that
    \[ \begin{aligned} 
        y_k %
        = h_1(x_k) +  \frac{\lambda_k}{\sigma_k} - \frac{1}{\sigma_k} \nabla \psi^{\sigma_k}\left (h_1(x_k) +  \frac{\lambda_k}{\sigma_k}\right), \text{ and }
        z_k %
        = h_2(x_k) +  \frac{\gamma_k}{\sigma_k} - \frac{1}{\sigma_k} \nabla \delta_{\R_-^q}^{\sigma_k}\left (h_2(x_k) +  \frac{\gamma_k }{ \sigma_k}\right). 
	\end{aligned} \]
    From the definition of $\lambda_{k+1}$ and $\gamma_{k+1}$, combining the above two equations, we have
	\begin{equation}
		\label{eqn:grad_Lk}
		\begin{aligned}
			& \grad L_k(x_k)
			= \grad f(x_k) + \sum_{i=1}^m [\lambda_{k+1}]_i \grad\,[h_1(x_k)]_i 
			+ \sum_{i=1}^q [\gamma_{k+1}]_i \grad\, [h_2(x_k)]_i,  
		\end{aligned}
	\end{equation}
	where the chain rule for gradients on manifolds is used\footnote{
	Let $g: \cM \to \R^m$ and $f: \R^m \to \R$ be two smooth maps,
	we can use Definition~\ref{def:riemannian-gradient} to compute the gradient of $f \circ g$.
	For $p \in \cM$, $\xi_p \in T_p\cM$, choosing any curve $\gamma: (-1, 1) \to \cM$ such that $\gamma(0) = p$ and $\dot \gamma(0) = \xi_p$, we have
	$\inner{\xi_p}{\grad\,(f\circ g)(p)} = \xi_p (f \circ g) 
	= \left.\frac{\dd (f \circ g \circ \gamma)(t)}{\dd t}\right|_{t=0} 
	= \sum_{i=1}^m\alpha_i\left.\frac{\dd (g_i \circ \gamma)(t)}{\dd t}\right|_{t=0} 
	= \sum_{i=1}^m \alpha_i \cdot (\xi_p g_i) 
	= \sum_{i=1}^m \alpha_i \inner{\xi_p}{\grad\,[g(p)]_i}$, where $\alpha := \nabla f(q)$, $q := g(\gamma(t))$ and $g_i$ is the $i$-th component of $g$.
	Thus, $\grad\, (f \circ g)(p) = \sum_{i=1}^m \alpha_i \grad\, [g(p)]_i$.
	}.

    Since $y_k = \prox_{\psi/\sigma_k} (h_1(x_k) +  \lambda_k/\sigma_k)$, then
	\begin{equation}
		\label{eqn:grad_yk}
		0 \in \partial \psi(y_k) - (\sigma_k(h_1(x_k) - y_k) +  \lambda_k) = \partial \psi(y_k) - \lambda_{k+1}. 
	\end{equation}
	Note that $\lim_{K \ni k \to \infty} y_k = y_*$, then $\left\{ y_k \right\}_{k\in K}$ is bounded.
	By the locally boundedness of the subdifferential~\cite[Corollary 24.5.1]{rockafellar1970convex}, $\bigcup_{k \in K} \partial \psi(y_k)$ is also bounded.
	Therefore, $\{ \lambda_{k+1} \}_{k \in K}$ is bounded, so we can choose $K_1 \subseteq K$, $\lambda_* \in \R^m$ such that $\lim_{ K_1 \ni k \to \infty} \lambda_{k+1} = \lambda_*$, which implies $\lambda_*\in\partial\psi(y_*)$.

	On the other hand, since 
	\[ z_k = \Pi_{\R_-^q} \left ( h_2(x_k) +  \frac{\gamma_k}{\sigma_k} \right ) = \argmin_{z \leq 0} \norm{h_2(x_k) +  \frac{\gamma_k}{\sigma_k} - z}^2_2, \]
then from the optimality condition of the above problem, it holds that
	\begin{equation}
		\label{eqn:grad_zk}
		0 = [\sigma_k( h_2(x_k) - z_k) +  \gamma_k]^\trans z_k = \gamma_{k+1}^\trans z_k \quad \mathrm{and} \quad \gamma_{k+1} \geq 0.
	\end{equation}

	Let $\cA_{k} = \{ i \in [q] : [z_k]_i = 0 \}$, from \eqref{eqn:grad_zk}, we know $[\gamma_{k+1}]_i = 0$ for $i \notin \cA_k$. 
	Define $\cA_* := \{ i \in [q] : [z_*]_i = 0 \}$.
	Since $z_k \to z_*$, we also find that for sufficiently large $k$ and $i \notin \cA_*$, $[\gamma_{k+1}]_i = 0$.
	Then, when $k$ is large enough \eqref{eqn:grad_Lk} can be written as 
	\begin{equation}
		\label{eqn:alm-Astar-equ}
			\grad L_k(x_k)
			= \grad f(x_k) + \sum_{i=1}^m [\lambda_{k+1}]_i \grad\,[h_1(x_k)]_i 
			+ \sum_{i \in \cA_*} [\gamma_{k+1}]_i \grad\, [h_2(x_k)]_i.  
    \end{equation}
	Using the Carath\'eodory's theorem of cones~\cite{bertsekas1997nonlinear}, there exist $J_k \subseteq \cA_*$ and $[\hat \gamma_k]_j \geq 0$, where $j \in J_k$, such that $\left\{ \grad\, [h_2(x_k)]_j : j \in J_k \right\}$ are linearly independent and
	\begin{equation}
		\label{eqn:alm-J_k-equ}
			\grad L_k(x_k)
			= \grad f(x_k) + \sum_{i=1}^m [\lambda_{k+1}]_i \grad\,[h_1(x_k)]_i 
			+ \sum_{j \in J_k} [\hat\gamma_{k}]_j \grad\, [h_2(x_k)]_j,
	\end{equation}
	Since $J_k \subseteq [q]$ is a finite set, we can choose $J_*$ and $K_2 \subseteq K_1$ such that $J_k = J_*$ for all $k \in K_2$ and $|K_2| = \infty$.
	Define $M_k = \max\left\{ [\hat\gamma_{k}]_i : i \in J_* \right\}$ for $k \in K_2$.
	When $\{ M_k \}_{k \in K_2}$ is bounded, then we can find $K_3 \subseteq K_2$ and $\gamma_* \in \R^p_+$ such that $[\gamma_*]_i = 0$ for $i \notin J_*$ and $\lim_{ K_3 \ni k \to \infty} [\hat \gamma_k]_i = [\gamma_*]_i$ for $i \in J_*$.
	Using the fact that $\norm{\grad L_k(x_k)} \to 0$, \eqref{eqn:grad_Lk} implies \eqref{eqn:reformulated-kkt-x}.
	Since $J_* \subseteq \cA_*$ and $[\gamma_*]_i = 0$ for $i \notin J_*$, then \eqref{eqn:reformulated-kkt-z} follows.

	When $\{ M_k \}_{k\in K_2}$ is unbounded, we can find $K_4 \subseteq K_2$ and $\hat \gamma \in \R_+^q$ such that $\lim_{ K_4 \ni k \to \infty} [\hat \gamma_k]_i / M_k = [\hat \gamma]_i$ for $i \in J_*$ and $[\hat \gamma]_i = 0$ otherwise. By the definition of $M_k$, we have $\hat\gamma\neq 0$ and $\|\hat\gamma\|_\infty = 1$.
	Dividing \eqref{eqn:alm-J_k-equ} by $M_k$, using the boundedness of $\grad L_k(x_k)$, $\grad f(x_k)$ and $\grad\, [h_1(x_k)]_i$, letting $k \in K_4 \to \infty$ and noticing that $\lambda_{k + 1} \to \lambda_*$, $x_k \to x_*$, we can get
	\[  \sum_{j\in J_*} [\hat \gamma]_j \grad\, [h_2(x_*)]_j = 0.  \]
	Due to $\normz{\hat \gamma}_\infty = 1$ and $\hat \gamma \geq 0$, we know $\left\{ \grad\, [h_2(x_*)]_j : j \in J_* \right\}$ are linearly dependent with non-negative coefficients. However, they are linearly independent near $x_*$, which contradicts to the CPLD assumption. Thus, $M_k$ is always bounded.
	Moreover, when the LICQ holds at $x_*$ but $\{ \gamma_{k + 1} \}$ is unbounded, we can divide \eqref{eqn:alm-Astar-equ} by $\norm{\gamma_{k + 1}}_2$ and yield a contradiction to the LICQ condition. Therefore, $\{ \gamma_{k+1} \}$ is bounded and contains a convergent subsequence.
\hfill $\Box$
\end{proof}

\section{A Globalized Semismooth Newton Method}\label{sec:newton}
Recall that at every step we have to find an approximated stationary point of the following problem.
\begin{equation}\label{sub:obj}
\min_{x\in\cM}~ L_k(x) = f(x) 
+ \psi^{\sigma_k} \left ( h_1(x) + \frac{\lambda_k}{\sigma_k}  \right )
+ \delta_{\R_-^q}^{\sigma_k} \left ( h_2(x) + \frac{\gamma_k}{\sigma_k}  \right ).
\end{equation}
It is known that $L_k$ is continuously differentiable, and from the calculation around \eqref{eqn:grad_Lk}, the gradient of $L_k$ is
\begin{equation}\label{grad_newton}
\grad L_k(x) = \grad f(x)
+ \sum_{i=1}^m \alpha_i(x) \grad\,[h_1(x)]_i
+ \sum_{i=1}^q \beta_i(x) \grad\,[h_2(x)]_i,
\end{equation}
where $\alpha(x)$ and $\beta(x)$ are gradients of the Moreau-Yosida regularization $\psi^{\sigma_k}$ and $\delta^{\sigma_k}_{\R_-^q}$ at $h_1(x) + \lambda_k / \sigma_k$ and $h_2(x) + \gamma_k / \sigma_k$, respectively.
Note that $\alpha(x)$ and $\beta(x)$ are continuous but not differentiable, so the Newton method cannot be applied, and we need the semismooth Newton method.
For simplicity, we consider the following abstract problem
\begin{equation}\label{sub:obj_new}
\min\ \varphi(x),\ \mathrm{s.t.}\ x\in \cM,
\end{equation}
where $\varphi:\cM\to\R$ is continuously differentiable. 
We make  the following assumption. %
\begin{assumption}
	\label{assumption:vector-field-decomposition}
	The vector field $X := \grad \varphi$ is locally Lipschitz and can be factorized into $X = X_1 + X_2  + X_3$ such that  $X_1$ is smooth and $X_2, X_3$ are locally Lipschitz. 
\end{assumption}

Due to the existence of two nonsmooth terms in $X$, we need to extend the Clarke generalized covariant derivatives in the definition of the semismoothness in \eqref{eqn:semismooth-def} to a general set-valued map. 
\begin{definition}
	\label{def:semismooth-general}
	Let $X$ be a vector field on $\cM$ and 
	$\cK: \cM \rightrightarrows \cL(T\cM)$ be an upper semicontinuous\footnote{We say the map $\cK$ is \emph{upper semicontinuous} if for every $\varepsilon > 0$ there exists $ \delta>0$ such that for all $q \in B_\delta(p)$ we have $P_{qp}\cK(q) \subset \cK(p) + \hat B_\varepsilon(0)$, where $\hat B_\varepsilon(0) := \{ V \in \cL(T_p\cM) : \norm{V} < \varepsilon \}$.} 
	set-valued map such that $\cK(p)$ is a non-empty compact subset of $\cL(T_p\cM)$. %
	Suppose $X$ is Lipschitz and directionally differentiable in a neighborhood $U$ of $p \in \cM$.
	We say that $X$ is \emph{semismooth at $p$ with respect to $\cK$} if for every $\varepsilon > 0$, there exists $\delta > 0$ such that for every $q \in B_\delta(p)$ and $H_q \in \cK(q)$,
	\begin{equation}
	\label{eqn:semismooth-general-weak}
	\norm{X(p) - P_{qp}[X(q) + H_q \exp^{-1}_qp] }\leq \varepsilon d(p, q), %
	\end{equation}
	Moreover, we say that $X$ is \emph{semismooth at $p$ with order $\mu \in (0, 1]$ with respect to $\cK$} if there exist $C > 0, \delta > 0$ such that for every $q \in B_\delta(p)$ and $H_q \in \cK(q)$,
	\begin{equation}
	\label{eqn:semismooth-general}
	\norm{X(p) - P_{qp}[X(q) + H_q \exp^{-1}_qp] }\leq C d(p, q)^{1 + \mu}. %
	\end{equation}
	In particular, we say $X$ is \emph{strongly semismooth at $p$ with respect to $\cK$} if $\mu = 1$ in \eqref{eqn:semismooth-general}.
\end{definition}
\begin{remark}
	Definition~\ref{def:semismooth-clarke} and Definition~\ref{def:semismooth-general} coincide when $\cK = \partial X$~\cite[Proposition 3.1]{de2018newton}.
\end{remark}

When one of the nonsmooth terms vanishes, i.e., $X_2=0$ or $X_3 = 0$, we can choose $\cK = \partial_B X_3$ or $\cK = \partial_B X_2$. 
When both terms $X_2$ and $X_3$ are non-trivial, it is difficult to compute $\partial_B (X_2+X_3)$ in general as we only know $\partial_B(X_2+X_3)\subset\partial_B X_2 + \partial_B X_3$. 
In this case, we choose $\cK = \partial_B X_2 + \partial_B X_3$. 
The next proposition guarantees that such a choice does not affect the semismoothness of $X_2 + X_3$. 

\begin{proposition}
	\label{prop:semismooth-sum}
	Let $X, Y$ be vector fields on $\cM$ and $p \in \cM$.
	Suppose $X$ is semismooth at $p$ with order $\mu \in (0, 1]$ with respect to $\cK_X$,
	and $Y$ is semismooth at $p$ with order $\mu$ with respect to $\cK_Y$.
	Then, $X + Y$ is semismooth at $p$ with order $\mu$ with respect to $\cK_X + \cK_Y$. %
\end{proposition}
\begin{proof}
	By Definition~\ref{def:semismooth-general}, there exist $C > 0, \delta > 0$ such that \eqref{eqn:semismooth-general} holds for both $X$ and $Y$.
	Then, for every $q \in B_\delta(p)$, $H_X \in \cK_X(q)$, $H_Y \in \cK_Y(q)$, we have
	\[ \begin{aligned}
	\peq &\norm{(X + Y)(p) - P_{qp}[(X + Y)(q) + (H_X + H_Y) \exp^{-1}_qp] } \\
	 \leq& 
	\norm{X(p) - P_{qp}[X(q) + H_X\exp^{-1}_qp] }
	+ \norm{Y(p) - P_{qp}[Y(q) + H_Y\exp^{-1}_qp] } < 2 Cd(p, q)^{1 + \mu},
	\end{aligned} \] 
	where the first inequality follows from the linearity of $P_{qp}$ and the triangle inequality of the norm.
	Therefore, the vector field $X + Y$ is semismooth with order $\mu$ with respect to $\cK_X + \cK_Y$. \hfill $\Box$
\end{proof}
From the numerical perspective, we impose the assumption on $\cK$ in our following analysis.
\begin{assumption}
	\label{assumption:symmetric-op}
	For every $p \in \cM$, every element in $\cK(p)$ is self-adjoint (symmetric).
\end{assumption}
Since the Riemannian Hessian is self-adjoint and an element in $\partial X(p)$ is the limit of a sequence of self-adjoint operators, then the Clarke generalized covariant derivatives fulfill the above assumption.

Now, we are ready to present the semismooth Newton method for finding $p\in\cM$ such that $X(p)=0$.
\begin{alg}
    \label{alg:semismooth}

    Choose $p_0 \in \cM$, $\bar \nu \in (0, 1]$ and let $\{ \eta_k \} \subset \R_+$ be a sequence converging to $0$. Set $\mu \in (0, 1/2), \delta \in (0, 1), m_{\mathrm{max}} \in \N$, and
	$p, \beta_0, \beta_1 > 0$, 
	
	Our algorithm repeats the following steps for $k = 0, 1, 2, \dots$
    \begin{enumerate}[label=(\roman*)]
	\item Choose $H_k \in \cK(p_k)$ and use the conjugate gradient (CG) method to find $V_k \in T_{p_k}\cM$ such that 
        \begin{equation}
            \label{eqn:newton-stopping}
		 \norm{(H_k + \omega_k I) V_k +X(p_k)} \leq \tilde \eta_k,
		\end{equation}

		where $\omega_k := \norm{X(p_k)}^{\bar \nu}$, $\tilde \eta_k := \min\big\{ \eta_k, \norm{X(p_k)}^{1+\bar \nu} \big\}$.
		Note that CG may fail when $H_k$ is not positive definite, we choose the first-order direction $V_k = -X(p_k)$ in this case.
	\item
	Choose the stepsize by one of the following linesearch methods:

    \begin{enumerate}
        \item[(LS-I)]
	If $V_k$ is not a sufficient descent direction of $\varphi$, i.e. it does not satisfy
    \begin{equation}
        \label{eqn:sufficient-descent}
    \inner{-X(p_k)}{V_k} \geq \min\{ \beta_0, \beta_1 \norm{V_k}^p \} \norm{V_k}^2,
    \end{equation} 
    then, we set $V_k$ to be $-X(p_k)$.

	Next, find the minimum non-negative integer $m_k$ such that
	\begin{equation}\label{linesearch}
	    \varphi(R_{p_k}(\delta^{m_k} V_k)) \leq \varphi(p_k) + \mu\delta^{m_k} \inner{X(p_k)}{V_k}. 
	\end{equation}
	\item[(LS-II)] Find the minimum non-negative integer $m_k \leq m_{\mathrm{max}}$ such that
        \[ \normz{X(R_{p_k}(\delta^{m_k}V_k))} \leq (1-2\mu\delta^{m_k})\norm{X(p_k)} . \]
    If $m_k$ cannot be found, then we set $m_k = m_{\mathrm{max}}$.
    \end{enumerate}
    
	\item
        Set $p_{k+1} = R_{p_k} (\delta^{m_k} V_k)$.
    \end{enumerate}
\end{alg}

We make several remarks for the above numerical algorithm.
\begin{remark}
	LS-I is a standard way to globalize the semismooth Newton method for minimizing smooth functions.
	Note that the CG method may fail, or $V_k$ may not be a descent direction as $H_k$ may not be positive definite. In both cases, Algorithm~\ref{alg:semismooth} reduces to the first-order method. 
	It is noted that the condition \eqref{linesearch} always holds for finite $m_k$ as shown in Theorem~\ref{thm:semismooth-LS-I-stationary}, so the requirement $m_k \leq m_{\max}$ in LS-I is not needed.
	In Theorem~\ref{thm:semismooth-LS-I-rate}, we show that the LS-I globalization method has a superlinear convergence under a ``convexity assumption'' and some regularity conditions.
\end{remark}
\begin{remark}
	LS-II is another way to globalize the Newton method~\cite{dirr2007nonsmooth,mifflin1977semismooth,sun2002semismooth}. 
	In the Riemannian setting, the convergence result is established under the assumption that $\norm{X}^2$ is differentiable~\cite{oliveira2020a}. 
	However, $\|X\|^2 = \|X_1+X_2+X_3\|^2$ is not differentiable in general. Thus, LS-II does not have a convergence guarantee.
	In our experiments, we find both LS-I and LS-II have a similar performance for ``convex problems'' like \eqref{eqn:cm-problem} and LS-II is suitable for ``nonconvex problems'' like \eqref{eqn:spca-problem} and \eqref{eqn:cons-spca-problem}.
\end{remark}
\begin{remark}
	We can use the method in \cite{chen2017augmented} to choose an initial point of Algorithm~\ref{alg:semismooth} to guarantee the condition~\eqref{eqn:subproblem-stopping-rule}.
	For LS-I, since it is a descent method, it suffices to choose $p_0$ such that $L_k(p_0) \leq \Phi$. This can be done by
    \[ p_0 = \begin{cases}
        x_\feas, & L_k(x_{k-1}) > \Phi, \\
        x_{k - 1}, & L_k(x_{k-1}) \leq \Phi,
	\end{cases} \]
	where $x_\feas$ and $x_{k-1}$ is defined in Algorithm~\ref{alg:alm-outer}. The same initialization method is used for the LS-II method. However, the LS-II method has no convergence guarantee as it is not a descent method.
\end{remark}

\subsection{The Analysis of Global Convergence} \label{sec:global-convergence}
We collect two technical lemmas for the subsequent analysis of Algorithm~\ref{alg:semismooth}. 
The second part of the first one can be regarded as an approximated cosine law on manifolds.

\begin{lemma}[Lemma 2.3 and 2.4 in~\cite{daniilidis2018self}]
	\label{lem:cosine-law}
Fix $p \in \cM$, the following properties hold.
\begin{enumerate}[(i)]
    \item There exists $r>0$ such that for every $q\in B_r(p)$, the exponential map $\exp_q$ is a diffeomorphism from $\{ v \in T_q\cM : \|v\| < 2r \}$ to $B_{2r}(q)$.
    \item There exist $K>0$ and $r>0$ such that for every $v, w \in T_q \cM$ with $\|v\|, \|w\| < 2r$, 
	it holds 
	\begin{equation}
		\label{eqn:cosine-law}
		 | d(\exp_q v, \exp_q w)^2 - \normz{v - w}^2 | \leq K \|v\|^2 \|w\|^2.
	\end{equation}
\end{enumerate}
\end{lemma}

The second lemma is an extension of Proposition 2 in \cite{absil2012projection}, which follows from the compactness of $U$ and Taylor's theorem. Since it is the key lemma for extending the analysis of Algorithm~\ref{alg:semismooth} from the exponential map to general retractions,  we give its proof in Appendix~\ref{app:proof-retraction-approx} for completeness.
\begin{lemma}
	\label{lem:retraction-approx}
	Let $U \subset \cM$ be a compact subset and $R$ be a retraction, then there exist $C, r > 0$ such that for any $q \in U$ and $v \in T_q \cM$ with $\| v \| < r$,  the following inequalities hold
	\begin{equation}
		\label{eqn:retraction-approx}
		d(R_{q}v , q) \leq C \| v\| \quad \text{ and } \quad
		d(R_{q}v , \exp_q v) \leq C \| v\|^2.
	\end{equation}
\end{lemma}

The next theorem establishes the global convergence of Algorithm~\ref{alg:semismooth} with LS-I.
\begin{theorem}
    \label{thm:semismooth-LS-I-stationary}
    Let $\{ p_k \}$ be the sequence generated by Algorithm \ref{alg:semismooth} with LS-I. 
    Suppose there exists $\delta > 0$ such that $\Omega := \{ p \in \cM : \varphi(p) \leq \varphi(p_0) + \delta \}$ is compact. Then, the following properties hold:
    \begin{enumerate}[(i)]
		\item For every $k \in \N$, there exists $m_k < \infty$ such that \eqref{linesearch} holds. 
	Moreover, if there exist $k_0 \in \N$ and $\varepsilon_0 > 0$ such that the following holds for every $k > k_0$, 
	\begin{equation}
		\label{eqn:strong-sufficient-descent}
		\innerz{-X(p_k)}{V_k} \geq \varepsilon_0 \normz{V_k}^2,
	\end{equation}
	then there exists an $m_{\max} \in \N$ such that $m_k \leq m_{\max}$ for all $k \in \N$.
        \item Let $\cP$ be the set of accumulation points of $\{p_k\}$, then $\cP$ is a non-empty set and every $p_*\in\cP$ is a stationary point of $\varphi$, i.e., $X(p_*) = 0$.
	\item  $\lim_{k \to \infty} \norm{X(p_k)} = 0$ and $\lim_{k \to \infty} \norm{V_k} = 0$.
    \end{enumerate}

\end{theorem}
\begin{proof}
	(i) Let $r, C$ be the constants such that Lemma~\ref{lem:retraction-approx} holds for $U = \Omega$.
	Since $X$ is locally Lipschitz, for every $p \in \Omega$ we can find $S_p, L_p > 0$ such that $X$ is $L_p$-Lipschitz in $B_{2S_p}(p)$.
	Since $\Omega$ is compact,  we can find a finite set $\{ q_1, \dots, q_T \} \in \Omega$ such that $\bigcup_j B_{S_{q_j}}(q_j) \supset \Omega$. 
	Define $R := \min \{ r, \min_j S_{q_j} \}$, $\tilde L := \max_j L_{q_j}$.
	Since there exists $B_{S_{q_j}}(q_j) \ni p$ for every $p \in \Omega$, we find $B_R(p) \subset B_{2S_{q_j}}(q_j)$. Thus,
	$X$ is $\tilde L$-Lipschitz in $B_R(p)$ for every $p \in \Omega$.
    Fix $p \in \Omega$ with $\varphi(p) < \varphi(p_0) + \delta$ and $V \in T_p\cM$, and define $\gamma(t) = \exp_p(tV)$. 
    Consider the function $\hat \varphi = \varphi \circ \gamma$, we know that 
	$\hat \varphi^\prime(t) = \inner{X(\gamma(t))}{\dot \gamma(t)}$.
	Note that for $0 \leq s, t < R / \norm{V}$,
    \[ \begin{aligned}
        \left | \hat\varphi^\prime(t) - \hat\varphi^\prime(s) \right | 
        &= 
        \left | \inner{X(\gamma(t))}{\dot \gamma(t)} - \inner{X(\gamma(s))}{\dot \gamma(s)} \right |  = 
        \left | \inner{X(\gamma(t)) - P^{s \to t}_\gamma X(\gamma(s))}{\dot \gamma(t)} \right |  \\
        & \leq \norm{X(\gamma(t)) - P^{s \to t}_\gamma X(\gamma(s))}\norm{\dot \gamma(t)}   \leq \tilde L \ell(\gamma|_{[s,t]})\norm{\dot \gamma(t)} = \tilde L\norm{V}^2 |t - s|,
	\end{aligned} \]
	where the last inequality follows from the Lipschitzness of $X$ and $\ell(\gamma|_{[s,t]}) = |t-s|\norm{V}$ and $\norm{\dot \gamma(t)} = \norm{\dot\gamma(0)} = \norm{V}$.\footnote{Since $\gamma$ is a geodesic which is parallel to itself, i.e., $\nabla_{\dot \gamma} \dot \gamma = 0$, then $\frac{\dd}{\dd t} \| \dot \gamma(t) \|^2 = 2\inner{\nabla_{\dot \gamma}\dot\gamma}{\dot\gamma} \equiv 0$. Therefore, we have $\|\dot\gamma(t)\| = \| \dot \gamma(0) \| = \| V \|$ and $\ell(\gamma|_{[s,t]}) = \int_s^t \| \dot \gamma(u) \| \dd u = |t - s| \|\dot\gamma(0)\|$.}
    Thus, $\hat \varphi$ is $\tilde L\norm{V}^2$-smooth.
	Besides, we know $\varphi$ is Lipschitz on the compact set $\Omega$ since $\grad \varphi = X$ is bounded.
	Setting $p = p_k$, $V = V_k$ to be the quantities defined in Algorithm~\ref{alg:semismooth}, 
	$\hat L$ to be the Lipschitz constant of $\varphi$ and $L := \tilde L + 2\hat LC$,
	 we know when $0 \leq t < R / \norm{V_k}$,
	 it holds that $|\varphi(R_{p_k} (tV_k)) - \varphi(\exp_{p_k} (tV_k))| \leq \hat L d(R_{p_k} (tV_k), \exp_{p_k} (tV_k))$, and
	\begin{equation}\label{eqn:est_decrease}
	    \begin{aligned}
		\varphi(p_k) - \varphi(R_{p_k}(t V_k)) 
		\overset{\eqref{eqn:retraction-approx}}&{\geq} \varphi(p_k) - \varphi(\exp_{p_k}(t V_k)) - \hat LC \| tV_k \|^2   \\
		& \geq \inner{-X(p_k)}{t V_k} - \frac{Lt^2}{2} \norm{V_k}^2.
	\end{aligned}
	\end{equation}
	In view of \eqref{eqn:est_decrease}, when $V_k \neq 0$, we have for $0 \leq t < R / \norm{V_k}$,
	\[ \begin{aligned}
		\varphi(p_k) - \varphi(R_{p_k}(t V_k))  
		\begin{cases}\geq t \Big ( 1 - \frac{Lt}{2\min\{ \beta_0, \beta_1 \normz{V_k}^p \}} \Big ) \inner{-X(p_k)}{V_k}, & \text{ if \eqref{eqn:sufficient-descent} holds,} \\
		=t \Big(1-\frac{Lt}{2}\Big)\inner{-X(p_k)}{V_k}, & \text{ if $V_k=-X(p_k)$,}
		\end{cases}
	\end{aligned} \]
	where $\beta_0, \beta_1$ and $p$ are constants defined in Algorithm~\ref{alg:semismooth}. Thus, for a fixed $k \in \N$, \eqref{linesearch} holds whenever
	$\delta^{m_k} \leq \min \{ R / \norm{V_k}, 2\min\{ 1, \beta_0, \beta_1 \normz{V_k}^p \} (1 - \mu) / L \}$.
	On the other hand, \eqref{linesearch} automatically holds for $m_k = 0$ when $V_k = 0$. Thus, $\{ \varphi(p_k) \}$ is a non-increasing sequence, which implies $\{p_k\}\subset\Omega$.
	
	Next, we show that $\sup_k \normz{V_k} < \infty$. We assume on the contrary that there exists a subsequence $\{ k_j \}$ such that $\normz{V_{k_j}} \to \infty$. 
	The compactness of $\Omega$ and the continuity of $X$ imply that $\sup_k \normz{X(p_k)} < \infty$.
	Since $V_{k_j}$ satisfies \eqref{eqn:sufficient-descent} or $V_{k_j} = -X(p_{k_j})$, it holds that $\normz{V_{k_j}} \leq \normz{X(p_{k_j})} / \min\{1,  \beta_0, \beta_1 \normz{V_{k_j}}^p \}$, 
	which is bounded and hence contradicts to the assumption that $\|V_{k_j}\| \to \infty$. 
	Therefore, $\|V_k\|$ is uniformly bounded and \eqref{eqn:est_decrease} holds for all $k \in \N$ and $0 \leq t < R / \sup_{k} \normz{V_k} =: \bar t$. 
	Moreover, when \eqref{eqn:strong-sufficient-descent} holds, we can choose $m_k$ such that 
	$\delta^{m_k} \leq \min\{ \bar t, 2 \min\{ 1, \varepsilon_0 \} (1 - \mu) / L \}$. Since the term on the right-hand side is independent of $k$ and positive, $\{ m_k \}$ is uniformly bounded.

	(ii) Since $\Omega$ is compact and $\{ p_k \} \subset \Omega$, the set $\cP$ is non-empty.
	Let $p_* \in \cP$, then there exists a subsequence $\{ k_j \}$ such that $p_{k_j} \to p_*$ as $j \to \infty$. Now, we prove that $X(p_*) = 0$.
	Assume on the contrary that $X(p_*) \neq 0$, then we claim that $\inf_{j} \normz{V_{k_j}} > 0$.
	Otherwise, when $\inf_j \normz{V_{k_j}} = 0$, in view of \eqref{eqn:newton-stopping}, 
	we have either $\normz{X(p_{k_j})} = \normz{V_{k_j}}$ or $\normz{X(p_{k_j})} \leq \eta_{k_j} + \normz{(H_{k_j} + \omega_{k_j} I)V_{k_j}}$.
	The upper-semicontinuity of $\cK$ and $\partial X$ shows that they are locally bounded and the compactness of $\Omega$ gives that $\cK \cup \partial X$ is uniformly bounded in $\Omega$.
	Therefore, $\sup_j \normz{H_{k_j}} < \infty$. 
	Since $\eta_{k_j} \to 0$, $\omega_{k_j} = \normz{X(p_{k_j})}^{\bar \nu}$ is bounded and $\inf_j \normz{V_{k_j}} = 0$, we conclude that $\liminf_{j \to \infty} \normz{X(p_{k_j})} = 0$, which contradicts to $X(p_*) \neq 0$. 
	Thus, $\normz{V_{k_j}}$ is bounded away from zero.

	From \eqref{linesearch} and the fact that $\{ \varphi(p_k) \}$ is non-increasing, we know
	\[  \begin{aligned}
		\varphi(p_{k_{j+1}}) - \varphi(p_{k_j}) &\leq \varphi(p_{k_j+1}) - \varphi(p_{k_j})
		\leq \mu \delta^{m_{k_j}} \innerz{X(p_{k_j})}{V_{k_j}} \\
		& \leq -\mu \delta^{m_{k_j}}\min\{ 1, \beta_0,\beta_1\|V_{k_j}\|^p \} \|V_{k_j}\|^2< 0.
	\end{aligned} \]
	Taking $j \to \infty$, the term on the left vanishes. 
	Since $\normz{V_{k_j}}$ is bounded away from zero, 
	 $\delta^{m_{k_j}} \to 0$ as $j \to \infty$, and hence $m_{k_j} \to \infty$. 
	 Thus, for sufficiently large $j$, we know $t_j := \delta^{m_{k_j}-1}<\bar t$.
	 As $m_k$ is the smallest integer such that \eqref{linesearch} holds, when replacing $m_k$ with $m_k - 1$, we have
	 \[
		  \varphi(q_{j}) > \varphi(p_{k_j}) + \mu t_j \innerz{X(p_{k_j})}{V_{k_j}},
	 \]
	 where $q_{j} := R_{p_{k_j}} (t_j V_{k_j})$.
	Together with \eqref{eqn:est_decrease}, it holds
	   \[ \mu t_j \innerz{X(p_{k_j})}{V_{k_j}}
		 < \varphi(q_j) - \varphi(p_{k_j})
		 \leq t_j \innerz{X(p_{k_j})}{V_{k_j}} + \frac{Lt_j^2}{2} \normz{V_{k_j}}^2, \]
		 and thus, 
		 \[ 0 \leq (1-\mu)\min\{ 1, \beta_0,\beta_1\|V_{k_j}\|^p\}\|V_{k_j}\|^2\leq (1-\mu) \innerz{-X(p_{k_j})}{V_{k_j}} < \frac{Lt_j}{2} \normz{V_{k_j}}^2. \]
	As $\{V_{k_j}\}$ is bounded, by passing to a subsequence if necessary, we assume that $\lim_{j \to \infty} V_{k_j} = V_* \neq 0$.
	Dividing the above inequality by $\|V_{k_j}\|^2$ and letting $j \to \infty$, 
	we obtain $(1-\mu)\min\{ \beta_0,\beta_1\|V_*\|^p\}=0$, which leads to a contradiction since $\mu \in (0, 1/2)$.
	Thus, we conclude that $X(p_*) = 0$.
	From the choice of $V_{k_j}$ in LS-I, we can also conclude that $\lim_{j \to \infty} V_{k_j} = 0$.

	(iii) If the whole sequence $\{ \normz{X(p_k)} \}$ contains an accumulation point other than $0$, 
	then there exists a subsequence $\{ \normz{X(p_{k_j})} \}$ such that $\normz{X(p_{k_j})} \to \alpha > 0$ as $j \to \infty$.
	By passing to a subsequence if necessary, we could assume $p_{k_j} \to p_*$ as $j \to \infty$. Thus, $p_* \in \cP$ and the above discussion gives that $X(p_*) = 0$, which leads to a contradiction.
	Similarly, it can be shown that the only accumulation point of $\{ \normz{V_k} \}$ is $0$. \hfill $\Box$
\end{proof}

\subsection{The Analysis of Local Superlinear Convergence} \label{sec:local-convergence}
To obtain the locally superlinear convergence rate of Algorithm~\ref{alg:semismooth}, we first introduce some basic results of Riemannian manifolds and technical lemmas. 
\begin{lemma}
	\label{lem:peturbation}
	Suppose $X$ is a locally Lipschitz vector field on $\cM$, $p_* \in \cM$, $\cK$ is defined in Definition~\ref{def:semismooth-general} and all elements in $\cK(p_*)$ are nonsingular and there exists $\lambda>0$ such that 
    $\lambda \geq \max\left\{ \|H^{-1}\|: H \in \cK(p_*) \right\}$.
    Then, for every $\varepsilon > 0$ and $\varepsilon \lambda < 1$, there exists a neighborhood $U$ of $p_*$ such that all elements in $\cK(p)$ are nonsingular for $p \in U$ and
    \[ \norm{H^{-1}}\leq \frac{\lambda}{1 -\varepsilon \lambda}, \ \forall\, p \in U,\ H \in \cK(p).
    \]
\end{lemma}
\begin{proof}
The proof is the same as the proof in \cite[Lemma 4.2]{de2018newton} in which $\cK=\partial X$ and only the upper-semicontinuity of $\partial X$ is used. \hfill $\Box$
\end{proof}

Below is a second-order Taylor theorem on manifolds, a generalization of Theorem 2.3 in \cite{hiriart1984generalized}. Its proof can be found in Appendix~\ref{app:proof-second-taylor}.
\begin{lemma}
	\label{lem:second-taylor}
	Suppose $\varphi: \cM \to \R$ is a continuously differentiable function with Lipschitz gradient, and  $p, q \in \cM$.
	Let $\gamma: [0, 1] \to \cM$ be a geodesic joining $p$ and $q$. Then, there exist $\xi \in (0, 1)$ and $M_\xi \in \partial\grad \varphi(\gamma(\xi))$ such that
    \[ \varphi(q) - \varphi(p) = \inner{\grad \varphi(p)}{\dot \gamma(0)} + \frac{1}{2}\inner{P^{\xi \to 0}_\gamma M_\xi P^{0 \to \xi}_\gamma\dot\gamma(0)}{\dot \gamma(0)}. \]
\end{lemma}

A direct consequence of Lemma~\ref{lem:peturbation} and \ref{lem:second-taylor} provides the following theorem that characterizes the second-order optimality conditions of \eqref{sub:obj_new}.
\begin{theorem}
	Suppose $\varphi: \cM \to \R$ is continuously differentiable with locally Lipschitz gradient, then the following properties hold 
	\begin{itemize}
		\item \emph{(Second-order necessary condition).}
		Suppose $x_* \in \cM$ is a local minimum of $\varphi$, then for any $v \in T_{x_*}\cM$, there is $H_v \in \partial\grad \varphi(x_*)$ such that $\inner{H_vv}{v} \geq 0$.
		\item \emph{(Second-order sufficient condition).}
		Suppose $x_* \in \cM$ such that $\grad \varphi(x_*) = 0$ and all elements in $\partial\grad\varphi(x_*)$ are positive definite, then $x_*$ is a strict local minimum of $\varphi$.
	\end{itemize}
	\label{thm:second-order-optimality}
\end{theorem}
\begin{remark}
	Using Lemma~\ref{lem:second-taylor}, we can see that the second-order sufficient condition implies the 
	\emph{strongly geodesic convexity} of $\varphi$ near $p_*$ as defined in \cite{zhang2016first}.
\end{remark}

The next two lemmas give the local analysis around the critical point of $X$. Their proofs are deferred to Appendices~\ref{app:proof-B-differentiability} and \ref{app:proof-mani-cvg}.
\begin{lemma}
	\label{lem:B-differentiability}
	Let $X$ be a locally Lipschitz vector field on $\cM$ and $p\in\cM$. If $X$ is directionally differentiable at $p$ and $X(p)=0$, then
	$\|P_{\exp_p v, p}X(\exp_p v)  - \nabla X(p; v)\| = o(\norm{v})$ as $T_p\cM  \ni v \to 0$.
\end{lemma}

\begin{lemma}
	\label{lem:mani-cvg}
	Suppose $\{ p_k \} \subset \cM$ and $V_k \in T_{p_k} \cM$ satisfy $\lim_{k \to \infty} p_k = p \in \cM$, $\lim\limits_{k \to \infty} \| V_k \| = 0$ and $d(R_{p_k} V_k, p) = o(d (p_k,p))$. 
	Then, $\lim\limits_{k\to\infty} \|V_k\| / d(p_k,p) = 1$ and $d(R_{p_k} V_k,p) = o( \|V_k \| )$.
\end{lemma}

\begin{theorem}
	\label{thm:semismooth-LS-I-rate}
	Under the same assumptions as in Theorem~\ref{thm:semismooth-LS-I-stationary}, 
    let $\cK$ be the set-valued map used in Algorithm~\ref{alg:semismooth}, and let $p_*$ be an accumulation point of $\{p_k\}$. If $X$ is semismooth at $p_*$ with order $\nu$ with respect to $\cK \cup \partial X$, 
    and all elements of $\cK(p_*) \cup \partial X(p_*)$ are positive definite,
	then we have $p_k\to p_*$ as $k\to\infty$ and for sufficiently large $k$, it holds
    \[ d(p_{k+1}, p_*) \leq O\big ( d(p_k, p_*)^{1 + \min\{ \nu, \bar \nu \}} \big ), \]
    where $\bar\nu\in(0,1]$ is the parameter defined in Algorithm~\ref{alg:semismooth}.
\end{theorem}
\begin{proof}
	Since $\cK(p_*) \cup \partial X(p_*)$ is positive definite, by Lemma~\ref{lem:peturbation} and the upper-semicontinuity of $\cK$ and $\partial X$, we can find a neighborhood $U \ni p_*$ and constants $\omega, M > 0$ such that 
	$M \|V\|^2 \geq \inner{HV}{V} \geq 2\omega \|V\|^2$ for every $p \in U, H \in \cK(p) \cup \partial X(p), V \in T_p\cM$. %
	By Lemma~\ref{lem:cosine-law} and the semismooth condition of $X$, there exist $C_0 > 0$ and $r_0 \in (0, 1/2]$ such that $B_{r_0}(p_*) \subset U$, the unique shortest geodesic joining points in $B_{r_0}(p_*)$ exists, and 
	\begin{equation}
	    \norm{X(q)+H_q\exp^{-1}_qp_*} \leq C_0 d(p_*,q)^{1+\nu},~\forall q\in B_{r_0}(p_*),
	\end{equation}
	and $X$ is $L$-Lipschitz in $B_{r_0}(p_*)$.
	From Lemma~\ref{lem:retraction-approx} and the compactness of $\Omega$ (the set defined in Theorem~\ref{thm:semismooth-LS-I-stationary}), we can further assume that \eqref{eqn:retraction-approx} holds for $q \in \Omega$ and $\| v \| < r_0$,
	and denote the constant in \eqref{eqn:retraction-approx} as $C_1 > 0$.

	From Theorem~\ref{thm:semismooth-LS-I-stationary}, both $\|V_k\|$ and $\|X(p_k)\|$ converge to zero as $k \to \infty$.
    Given arbitrary $0<r<r_0$, there exists some $K_0>0$ such that  $\omega_k = \|X(p_k) \|^{\bar \nu} < 1/2$, $\|V_{k}\| < r/(2C_1)$, $\beta_1 \|V_k\|^p \leq \omega$  and $(2M + 1)^{1 + \bar \nu} \|V_k\|^{\bar\nu} \leq \omega$ for $k \geq K_0$. 
	From Lemma~\ref{lem:second-taylor}, we know $\varphi(q) - \varphi(p_*) \geq \omega d(q, p_*)^2$ whenever $q \in B_r(p_*)$.
	Since $p_*$ is an accumulation point, we can find $K_1\geq K_0>0$ such that $\varphi(p_{K_1}) - \varphi(p_*) < \omega r^2 /4$ and $p_{K_1} \in B_r(p_*)$.
	Note that from \eqref{eqn:retraction-approx}, $d(p_{K_1}, p_{K_1 + 1}) \leq C_1\delta^{m_{K_1}}\|V_{K_1}\| < r/2$.
	Then, $d(p_{K_1 + 1}, p_*) \leq d(p_{K_1}, p_*) + d(p_{K_1}, p_{K_1 + 1}) < r$, i.e., $p_{K_1 + 1} \in B_r(p_*)$.
	Since $\varphi(p_k)$ is non-increasing, we still have $\varphi(p_{K_1 + 1}) - \varphi(p_*) < \omega r^2 /4$. By induction, we know $\{ p_k \}_{k \geq K_1} \subset B_r(p_*)$ which implies that $p_k\to p_*$ as $k\to\infty$.
	Below we assume that $k \geq K_1$.

	By the positive definiteness of $\cK \cup \partial X$ near $p_*$ and the self-adjoint assumption of $\cK$ (Assumption~\ref{assumption:symmetric-op}), the CG method in Algorithm~\ref{alg:semismooth} is able to find a direction $V_{k}$ satisfying \eqref{eqn:newton-stopping}.
	Thus, we know
    $\norm{X(p_{k})} \leq \norm{H_{k} + \omega_{k} I}\norm{V_{k}} / (1 - \norm{X(p_{k})}^{\bar \nu}) \leq (2M + 1)\norm{V_{k}}$.
	Then, we obtain
	\begin{equation}
		\label{eqn:V_k-control}
		\begin{aligned}
        \inner{-X(p_{k})}{V_{k}} 
        &= \inner{(H_{k} + \omega_{k} I)V_{k}}{V_{k}} - \inner{(H_{k} + \omega_{k} I)V_{k} + X(p_{k})}{V_{k}} \\
        & \geq 2\omega\norm{V_{k}}^2  - \tilde \eta_{k} \norm{V_{k}}
         \geq 2\omega\norm{V_{k}}^2  - (2M + 1)^{1 + \bar \nu} \norm{V_{k}}^{2 + \bar \nu} \geq \omega \|V_{k}\|^2.
		\end{aligned} 
	\end{equation}
    Thus, the condition \eqref{eqn:sufficient-descent} holds as $\omega \geq \beta_1\|V_k\|^p$.
	Note that $\| (H_{k} + \omega_{k} I)^{-1} \| \leq (2\omega)^{-1}$, and $\|X(p_{k})\| \leq L d(p_{k}, p_*)$. Define $C := 2\max\{C_0(2\omega)^{-1} , L^{1 + \bar \nu} + (2\omega)^{-1}L\}$, by the definition of $\omega_k, \tilde \eta_k$, we have
	\[ \begin{aligned}
			\peq  \normz{V_{k} - \exp_{p_{k}}^{-1} p_*} & \leq \normz{ (H_{k} + \omega_{k} I)^{-1}X(p_{k}) +  \exp_{p_{k}}^{-1} p_*} 
			  +(2\omega)^{-1} \tilde \eta_{k}  \\
			  & \leq (2\omega)^{-1} \normz{X(p_{k}) +  (H_{k}+\omega_{k} I)\exp_{p_{k}}^{-1} p_*} + (2\omega)^{-1} \tilde \eta_{k}  \\
			  & \leq (2\omega)^{-1} 
			  \left ( \normz{X({p_{k}}) +  H_{k}\exp_{p_{k}}^{-1} p_*} 
			  + \omega_{k} \normz{ \exp_{p_{k}}^{-1} p_* } \right )
			  + (2\omega)^{-1} \tilde \eta_{k}  \\
			  & \leq (2\omega)^{-1} 
			  \left ( C_0 d({p_{k}}, p_*)^{1 + \nu}  
			  + \omega_{k} d({p_{k}}, p_*) \right )
			  + (2\omega)^{-1} \tilde \eta_{k}  
			  \leq C d({p_{k}}, p_*)^{1 + \min\left\{ \nu, \bar\nu \right\}}.
	\end{aligned} \]
	
	Using \eqref{eqn:V_k-control} we find $\|V_k\| \leq \|X(p_k)\|/\omega \leq Ld(p_k, p_*) / \omega$.
	From Lemma~\ref{lem:cosine-law}, we know 
	\begin{equation} \label{eqn:locally-superlinear}\begin{aligned}
		d(R_{p_k} V_k, p_*)^2 
		\overset{\eqref{eqn:retraction-approx}}&{=} O\big (d(\exp_{p_k} V_k, p_*)^2 \big ) \\
		&\leq O \big ( \normz{V_k - \exp_{p_k}^{-1} p_*}^2 + K \|V_k\|^2 d(p_k, p_*)^2 \big) 
		\leq O\big(d({p_{k}}, p_*)^{2 + 2\min\left\{ \nu, \bar\nu \right\}}\big),
	\end{aligned} \end{equation}
	where $K$ is the constant defined in Lemma~\ref{lem:cosine-law}. 
	To complete the proof, we need to show that for $\mu \in (0, 1/2)$ and sufficiently large $k$, the linesearch condition 
	\begin{equation}
		\label{eqn:line-search-cond}
		\varphi(R_{p_k}V_k) \leq \varphi(p_k) + \mu \inner{X(p_k)}{V_k},
    \end{equation}
    holds with $m_k=0$. 
	Define $U_k := \exp^{-1}_{p_*} p_k$, $W_k := \exp^{-1}_{p_*} R_{p_k}V_k$.
    Applying Lemma~\ref{lem:second-taylor} to $p_*$, $p_k$ and $R_{p_k}V_k$, we have
    \[ \begin{aligned}
        \varphi(p_k)
        = \varphi(p_*) + \frac{1}{2}\innerz{\tilde R_kU_k}{U_k}, \text{ and }
        \varphi(R_{p_k}V_k) 
        = \varphi(p_*) + \frac{1}{2}\innerz{\tilde M_kW_k}{W_k},
	\end{aligned} \]
	where $\tilde R_k = P_{p_rp_*}R_kP_{p_*p_r}$, 
	$\tilde M_k = P_{p_mp_*}M_kP_{p_*p_m}$, 
	$R_k$ is a Clarke generalized covariant derivative at $p_r := \exp_{p_*}(\theta_kU_k)$ for some $\theta_k \in (0, 1)$, 
    and $M_k$ is a Clarke generalized covariant derivative at $p_m := \exp_{p_*}(\xi_kW_k)$ for some $\xi_k \in (0, 1)$.
    Subtracting the above two equations, we know
    \[ \begin{aligned}
        \varphi(R_{p_k}V_k)  - \varphi(p_k)
        &= -\frac{1}{2}\innerz{\tilde R_kU_k}{U_k} + O\left ( \normz{W_k}^2 \right ).
	\end{aligned} \]
	Thus, we have
    \begin{equation}\label{eqn:to_be_estimate_decrease}
        \begin{aligned}
        \peq \varphi(R_{p_k}V_k)  - \varphi(p_k) - & \frac{1}{2} \inner{X(p_k)}{V_k}  = -\frac{1}{2}\innerz{\tilde R_kU_k}{U_k} - \frac{1}{2}\inner{X(p_k)}{V_k} + O\left ( \normz{W_k}^2 \right ) \\
        &= -\frac{1}{2}\innerz{V_k}{X(p_k) - P_{p_*p_k}\tilde R_kU_k} - \frac{1}{2}\innerz{P_{p_kp_*}V_k + U_k}{\tilde R_k U_k}  + O\left ( \normz{W_k}^2 \right ).
	\end{aligned} 
    \end{equation} 
	Since $P_{p_rp_*}\exp^{-1}_{p_r}p_* = -\exp^{-1}_{p_*}p_r = -\theta_kU_k$ and $X(p_*) = 0$,
	then by the semismoothness
	$P_{p_rp_*} X(p_r) - \theta_k \tilde R_kU_k = P_{p_rp_*}(X(p_r) + R_k \exp^{-1}_{p_r}p_*) - X(p_*) =  o( \normz{\theta_k U_k} )$.
	From Lemma~\ref{lem:B-differentiability}, we know
	$P_{p_rp_*}X(p_r) = \theta_k \nabla X(p_*; U_k) + o(\normz{\theta_kU_k})$
	and $P_{p_kp_*}X(p_k) = \nabla X(p_*; U_k) + o(\normz{U_k})$, which implies
	\begin{equation}\label{eqn:transport_X_p_k}
	    P_{p_kp_*}X(p_k) = \tilde R_k U_k + o(\normz{U_k}). 
	\end{equation}
    Moreover, applying $\lim\limits_{k \to \infty} p_k = p_*$ and \eqref{eqn:locally-superlinear} to Lemma~\ref{lem:mani-cvg}, we have
	\begin{equation}\label{eqn:compare_U_V}
	    \frac{\norm{V_k}}{\norm{U_k}} = \frac{\norm{V_k}}{d(p_k,p_*)}\to 1 \text{ and } \|W_k\| = d(R_{p_k}V_k,p_*) = o(\norm{V_k})\text{ as }k\to\infty.
	\end{equation}
	From Lemma~\ref{lem:cosine-law}, it holds
	\begin{equation}
		\begin{aligned}
		 \| V_k - \exp^{-1}_{p_k}p_* \|^2 
		 &\leq d(\exp_{p_k} V_k, p_*)^2 + K \| V_k \|^2 \|\exp^{-1}_{p_k}p_*\|^2     \\
		 \overset{\eqref{eqn:retraction-approx}}&{\leq} \|W_k\|^2 + o( \|V_k\|^2)
		 \overset{\eqref{eqn:compare_U_V}}{=}  o( \|V_k\|^2).
		\end{aligned}
	\end{equation}
	Then, the above three displays imply that
	\begin{equation}\label{eqn:V_P_inequality}
	\begin{aligned}
	      |\innerz{V_k}{X(p_k) - P_{p_*p_k}\tilde R_kU_k}| &\leq \norm{V_k}\normz{X(p_k)-P_{p_*p_k}\tilde R_k U_k}
		  \leq o(\norm{V_k}\norm{U_k}) = o\left(\normz{V_k}^2\right),\\
	     |\innerz{P_{p_kp_*}V_k + U_k}{\tilde R_kU_k}| 
    &\leq \normz{V_k - \exp^{-1}_{p_k}p_*}\normz{\tilde R_k U_k}  
	\leq o\left ( \normz{V_k}^2 \right ).
	\end{aligned}
	\end{equation}
	Combining \eqref{eqn:V_k-control}, \eqref{eqn:to_be_estimate_decrease} and \eqref{eqn:V_P_inequality}, we have
    \[ \begin{aligned}
        \peq \varphi(R_{p_k}V_k)  - \varphi(p_k) 
        &= \frac{1}{2} \inner{X(p_k)}{V_k} + o\left( \normz{V_k}^2 \right )  \\
		&\leq \mu\inner{X(p_k)}{V_k}-\omega \left ( \frac{1}{2} - \mu \right) \normz{V_k}^2 + o\left( \normz{V_k}^2 \right ). 
	\end{aligned} \]
	Thus, \eqref{eqn:line-search-cond} holds for sufficiently large $k$ and $m_k = 0$ in LS-I. \hfill $\Box$
\end{proof}

\subsection{The Semismooth Condition on Submanifolds} \label{sec:semismoothness-from-euc}
In this section, we assume that the manifold $\cM$ is embedded in an ambient space $\bar \cM$ and the vector field $X$ on $\cM$ is the restriction of a vector field $\bar X$ defined on $\bar \cM$.
Our goal is to answer a natural question: can the Lipschitzness, the directional differentiability and the semismoothness of $X$ be inherited from those of~$\bar X$?
Throughout this section, we make the following assumptions.
\begin{assumption}
	\label{assumption:submanifolds}
	The $n$-dimensional Riemannian manifold $\cM$ is a compact embedded submanifold of the $d$-dimensional Riemannian manifold $\bar \cM$, i.e., $\cM \subset \bar \cM$ and the inclusion map $\iota: \cM \hookrightarrow \bar \cM$ is a smooth injection, and the differential of $\iota$ is injective, $\cM$ and $\iota(\cM)$ are homeomorphism, and the Riemannian metric of $\cM$ is inherited from that of $\bar \cM$~(see, e.g., \cite[Chapter 8]{lee2018introduction}).
\end{assumption}
\begin{assumption}
	\label{assumption:extension-of-X}
	The vector field $X$ is the restriction of $\bar X$ on $\cM$, where $\bar X$ is a vector field on $\bar \cM$.
\end{assumption}
It is noted that the compactness of $\cM$ is not essential as the semismoothness involves only local properties. For simplicity, we impose the assumption that $\cM$ is compact.
To answer the aforementioned question, we need to introduce some concepts about submanifolds.
Notations in Table~\ref{Table:Notations} are defined for $\cM$, and we add a line over them to represent corresponding notations on $\bar \cM$, e.g., $\overline \exp_p$, $\bar P_\gamma^{0 \to t}$ and $\bar \nabla_X Y$ are the exponential map, the parallel transport and the Levi-Civita connection on $\bar \cM$, respectively.
We use $d_\cM$ and $d_{\bar \cM}$ to denote the distance on $\cM$ and $\bar \cM$, respectively.

According to \cite[p.~226]{lee2018introduction}, for every $p \in \cM$, the tangent space
$T_p \bar \cM$ can be decomposed into $T_p \cM \oplus (T_p\cM)^\perp$,
i.e., $v \in T_p \bar \cM$ can be uniquely written as $v_\top + v_\perp$, where $v_\top \in T_p \cM$ and $v_\perp \in (T_p\cM)^\perp$.
Let $Y, Z$ be smooth vector fields on $\cM$, and $\bar Y, \bar Z$ be smooth vector fields on $\bar \cM$ such that $Y = \bar Y|_\cM$ and $Z = \bar Z|_\cM$, then the \emph{second fundamental form of $\cM$} is defined as $\II(Y, Z) := \left( \bar \nabla_{\bar Y} \bar Z \right)_\perp$.
Proposition~8.1 in~\cite{lee2018introduction} shows that $\II(Y, Z)$ is independent of the extensions $\bar Y, \bar Z$; $\II(Y, Z)$ is $C^\infty(\cM)$-bilinear; and the value of $\II(Y, Z)$ at $p \in \cM$ depends only on $Y(p), Z(p)$. Thus, we could safely write $\II(v, w)$ for $v, w \in T_p \cM$, and $\II(Y, Z)$ is still valid even for non-smooth vector fields $Y, Z$.
Moreover, the Gauss formula~\cite[Theorem 8.2, Corollary 8.3]{lee2018introduction} relates Levi-Civita connections on $\cM$ and $\bar \cM$ as follows:
\begin{equation}
	\label{eqn:gauss-formula}
	\bar \nabla_Y Z = \nabla_YZ + \II(Y, Z) \quad \text{ and } \quad
	\bar \nabla_{\dot\gamma} Z = \nabla_{\dot\gamma} Z + \II(\dot\gamma, Z),
\end{equation}
where $\gamma: (-1, 1) \to \cM$ is a smooth curve.
Note that $\II(Y, Z)$ is orthogonal to the tangent space of $\cM$, the above equations imply that $\nabla_YZ$ is the projection of $\bar \nabla_YZ$ onto $T\cM$, i.e., 
\begin{equation}
	\label{eqn:gauss-formula-projection}
	\nabla_YZ = (\bar \nabla_Y Z)_\top.
\end{equation}

Before presenting the main theorem of this section, we need to introduce the directional differentiability in the Hadamard sense whose Euclidean counterpart can be found in \cite[Definition 2.45]{bonnans2000perturbation}.
\begin{definition}
	\label{def:Hadamard-directionally-differentiable}
	A vector field $X: \cM \to T\cM$ is \emph{directionally differentiable at $p \in \cM$ in the Hadamard sense} if it is directionally differentiable at $p$ and 
	\begin{equation}
		\label{eqn:Hadamard-directionally-differentiable}
		\nabla X(p; v)= \lim_{\substack{t \downarrow 0 \\ T_p \cM \ni v^\prime \to v}} \frac{1}{t} \big [ P_{\exp_p(tv^\prime),p}X(\exp_p(tv^\prime)) - X(p) \big ].
	\end{equation}
\end{definition}
It is clear that the above definition is a stronger version than the directional differentiability in Definition~\ref{def:directional-differentiable}.
In Euclidean spaces, the Lipschitzness and the directional differentiability can imply the directional differentiability in the Hadamard sense~\cite[Proposition 2.49]{bonnans2000perturbation}. However, it is unclear whether this result can be generalized to Riemannian manifolds.

Below is our main theorem whose proof is deferred to Appendix~\ref{app:proof-semismoothness-from-euc}.

\begin{theorem}
	\label{thm:semismoothness-from-euc}
		Let $\bar \cK: \bar \cM \rightrightarrows \cL(T\bar \cM)$ be an upper semicontinuous set-valued map such that 
		$\bar \cK(p)$ is a non-empty compact set for every $p \in \bar \cM$,
		and let $\cK: \cM \rightrightarrows \cL(T\cM)$ be a set-valued map satisfying the conditions in Definition~\ref{def:semismooth-general}.
		Suppose $X$ is continuous on $\cM$, then the following statements hold.
	\begin{enumerate}[(i)]
	    \item If $\bar X$ is locally Lipschitz at $p \in \cM$, then $X$ is locally Lipschitz at $p$ (see Definition~\ref{def:lipschitz}).
	    \item If $\bar X$ is directionally differentiable at $p \in \cM$ in the Hadamard sense, then $X$ is directionally differentiable at $p$ in the Hadamard sense, and the directional derivative is 
		\begin{equation}
			\label{eqn:directional-derivative-on-submanifolds}
			\nabla X(p; v) = \bar \nabla X(p; v) - \II(v, X(p)) = (\nabla \bar X(p; v))_\top. 
		\end{equation} 
		\item 
		Suppose there exists a neighborhood $V_p \subset \cM$ of $p$ such that one of the following assumptions holds:
		 \begin{enumerate}[(a)]
			 \item $X(p) = 0$, and 
			 for every $q \in V_p, H_q \in \cK(q), v \in T_q \cM$, there exists $\bar H_q \in \bar \cK(q)$ such that $(\bar H_qv)_\top = H_q v$.
			 \item $X$ is locally Lipschitz at $p$, and for every $q \in V_p, H_q \in \cK(q), v \in T_q \cM$, there exists $\bar H_q \in \bar \cK(q)$
			 such that $(\bar H_q v)_\top = H_q v$ and $(\bar H_q v)_\perp = \II(v, X(q))$.
		 \end{enumerate}
		 Fix $\mu \in [0, 1]$ and $C > 0$, if there exists $\delta > 0$ such that for every $q \in \cM$ with $d_{\bar \cM}(p, q) < \delta$ and every $\bar H_q \in \bar \cK(q)$, the inequality \eqref{eqn:semismooth-general} holds, i.e., 
		 \begin{equation} 
			\label{eqn:theorem-semismooth-ambient}
			\| \bar X(p) - \bar P_{qp} [ \bar X(q) + \bar H_q \overline \exp_q^{-1}p] \| \leq C d_{\bar \cM}(p, q)^{1 + \mu}. 
		 \end{equation} 
		 Then, there exist $\hat C > 0, \tilde \delta> 0$ such that for every $q \in \cM$ with $d_\cM(p, q) < \tilde \delta$ and every $H_q \in \cK(q)$, the following inequality holds.
		\begin{equation}
			\label{eqn:theorem-semismooth-submanifolds}
		  \| X(p) - P_{qp} [ X(q) +  H_q \exp_q^{-1}p] \| \leq \hat Cd_{\cM}(p, q)^{1 + \mu}. 
		\end{equation}
		Moreover, when $\mu \in [0, 1)$, the constant $\hat C$ can be chosen as $2C$.

		 \item 
		 Suppose either the assumption (iii.a) or (iii.b) is fulfilled, and $\bar X$ is Lipschitz and directionally differentiable in the Hadamard sense in a neighborhood $\bar U_p \subset \bar \cM$ of $p \in \cM$.
		If $\bar X$ is semismooth at $p$ with respect to $\bar \cK$, then $X$ is also semismooth at $p$ with respect to $\cK$.
		 Moreover, if $\bar X$ is semismooth at $p$ with order $\mu \in (0, 1]$ with respect to $\bar \cK$, then $X$ is semismooth at $p$ with order $\mu$ with respect to $\cK$.
	\end{enumerate}
\end{theorem}

Next, we show that the Clarke generalized covariant derivative satisfies the assumption (iii.b), and illustrate how to verify the semismoothness condition in Theorem~\ref{thm:semismooth-LS-I-rate}. %

The following proposition relates the Clarke generalized covariant derivatives on $\cM$ and its ambient space $\bar \cM$.
As a special case when $\cM := \R^n \subset \R^{n+k} =: \bar \cM$, the second fundamental form $\II(u, v) \equiv 0$ and \eqref{eqn:clarke-inclusion} can be obtained from the corollary in \cite[p.~75]{clarke1990optimization}. 
The proof of this proposition is deferred to Appendix~\ref{app:proof-clarke-inclusion}.
\begin{proposition}
	\label{prop:clarke-inclusion}
	Let $p \in \cM$ and $v \in T_p \cM$, if $\bar X$ is locally Lipschitz at $p$, then
	\begin{equation}
	\label{eqn:clarke-inclusion}
	 \partial X(p)[v] \subseteq \bar \partial \bar X(p)[v] - \II(v, X(p)),
	\end{equation}
	where $\partial X(p)[v] := \{ Hv : H \in \partial X(p) \}$ and $\bar \partial \bar X(p)[v] = \{\bar H v:\bar H\in\partial \bar X(p)\}$.
\end{proposition}

\begin{corollary}
	\label{cor:semismoothness-from-euc}
	Suppose that $\bar X$ is Lipschitz and directionally differentiable in the Hadamard sense in a neighborhood $\bar U_p \subset \bar \cM$ of $p \in \cM$.
	If $\bar X$ is semismooth at $p$ with respect to $\bar \partial \bar X(p)$, then $X$ is also semismooth at $p$ with respect to $\partial X(p)$. 
	Moreover, if $\bar X$ is semismooth at $p$ with order $\mu \in (0, 1]$ with respect to $\bar \partial \bar X(p)$, 
	then $X$ is also semismooth at $p$ with order $\mu \in (0, 1]$ with respect to $\partial X(p)$.
\end{corollary}
\begin{proof}
	Proposition~\ref{prop:clarke-inclusion} shows that the assumption (iii.b) in Theorem~\ref{thm:semismoothness-from-euc} is satisfied, then the conclusion follows from Theorem~\ref{thm:semismoothness-from-euc}~(iv).
	\hfill $\Box$
\end{proof}
\begin{remark}
	When $X = X_1 + X_2 + X_3$ as in Assumption~\ref{assumption:vector-field-decomposition}, where $X_2, X_3$ are locally Lipschitz and $X_1$ is smooth. 
	Suppose $X_2, X_3$ can be extended to $\bar X_2, \bar X_3$ in a neighborhood of $\cM$ in $\bar \cM$.
	If we could apply Corollary~\ref{cor:semismoothness-from-euc} to show that $X_2, X_3$ are semismooth with respect to the Clarke generalized covariant derivatives, then Proposition~\ref{prop:semismooth-sum} would imply that $X_2 + X_3$ is semismooth with respect to $\partial X_2 + \partial X_3$.
\end{remark}
\begin{example} \label{example:semismoothness-of-CM}
	In the CM problem~\eqref{eqn:cm-problem}, the function $L_k$ used in Algorithm~\ref{alg:semismooth} (see \eqref{eqn:Lk-grad-cm} for its gradient) can be naturally extended to the Euclidean space $\R^{n \times r}$.
	We denote this extension by $\bar L_k$.
	It is the sum of the smooth function $\tr(X^\trans H X)$ and the Moreau-Yosida regularization of the $\| \cdot \|_1$-norm whose Euclidean gradient is a piecewise linear map.
	It is well-known that a piecewise linear map (in Euclidean spaces) is strongly semismooth~(see, e.g., \cite[Proposition 7.4.7]{facchinei2007finite}), 
	and hence $\nabla \bar L_k$ is also strongly semismooth (with respect to the Clarke generalized derivative).
	Since $X(Q) := \grad L_k(Q) = \Proj_Q \bar \nabla \bar L_k$ for $Q \in \St{n, r}$, and $\Proj_Q(V) = V - Q (Q^\trans V + V^\trans Q) / 2$~(see, e.g., \cite{hu2019brief}), then $X$ can also be extended to a vector-valued map $\bar X$ in $\R^{n\times r}$.
	Note that $\bar X$ is the composition of a smooth map and a strongly semismooth map, so $\bar X$ is also strongly semismooth.
	By Corollary~\ref{cor:semismoothness-from-euc}, we know that $\grad L_k = X$ is strongly semismooth with respect to $\partial \grad L_k(p_*)$.
	Indeed, the SPCA problem~\eqref{eqn:spca-problem} and the CM problem have the same non-smooth term, so $\grad L_k$ for \eqref{eqn:spca-problem} is also strongly semismooth.
	The constrained SPCA problem~\eqref{eqn:cons-spca-problem} has an additional non-smooth term, i.e., the indicator function of a convex set. The Euclidean gradient of the Moreau-Yosida regularization of this term is also strongly semismooth, and therefore $\grad L_k$ for \eqref{eqn:cons-spca-problem} is strongly semismooth.
\end{example}

\subsection{Calculations of Clarke Generalized Covariant Derivatives} \label{sec:calc-clarke}

In this section, we discuss the approach of calculating the Clarke generalized covariant derivative of a locally Lipschitz vector field $X$, which is difficult in general manifolds.
Here, we follow Assumption~\ref{assumption:submanifolds} and \ref{assumption:extension-of-X}, and further assume that $\bar \cM$ is the Euclidean space $\R^d$.

If $X$ is differentiable at $p \in \cM$,  it is known from \eqref{eqn:gauss-formula-projection} (see also \cite{hu2019brief}) that for every $v \in T_p\cM$,
\begin{equation}
\label{eqn:hess-euc}
\nabla X(p)[v] = \Proj_{p}(\bar \nabla \bar X(p)[v]),
\end{equation}
where $\Proj_p$ is the projection onto $T_p\cM$.
If $X$ is not differentiable at $p\in\cM$, Proposition~\ref{prop:clarke-inclusion} shows that 
	$\partial X(p)[v] \subseteq \Proj_p(\bar \partial \bar X(p)[v])$.
 As illustrated by the following example, the inclusion relationship can be strict, and thus we need further discussion for finding an element in $\partial X(p)$.
\begin{example}
	\label{example:clarke-s1}
Consider the manifold $\mathbb S^1 = \{ (x, y) \in \R^2 : x^2 + y^2 = 1 \} \subset \R^2$ and the vector-valued function $Y: \R^2 \to \R^2$ defined by
$Y(x, y) = \big( 2| x - 1/ \sqrt 2 |,\ 4| y - 1/\sqrt 2 | \big)^\trans$.
Let $p_0 = (1/\sqrt 2, 1/\sqrt 2)^\trans \in \mathbb S^1$. Note that the projection onto the tangent space $T_p \mathbb S^1$ is $I - pp^\trans$, we can define the vector field $X: \mathbb S^1 \to T\mathbb S^1$ by $X(p) = (I - pp^\trans)Y(p) \in T_p \mathbb S^1$ for $p \in \mathbb S^1$. 
In this case, $\bar X:\R^2\to\R^2$ is also $\bar X(p) = (I-pp^\trans)Y(p)$ for $p \in \R^2$.
If $\bar X$ is differentiable  at $p\in\mathbb S^1$, we know from \eqref{eqn:hess-euc} that
\begin{equation}
	\label{eqn:S1-euc-proj}
	\nabla X(p) = (I - pp^\trans) \bar \nabla \bar X(p) = (I - pp^\trans) (\bar \nabla Y(p) - p^\trans Y(p) I).
\end{equation} 
Moreover, by a direct calculation, it is known that 
\begin{equation}
	\label{eqn:S1-euc-proj-p0}
\Proj_{p_0}(\bar \partial\bar X(p_0)) =  (I - p_0p_0^\trans) \bar \partial \bar X(p_0)  = \mathrm{co}\left \{  
	\pm \begin{pmatrix}
		1 & -2 \\ -1 & 2
	\end{pmatrix},\
	\pm \begin{pmatrix}
		1 & 2 \\ -1 & -2
	\end{pmatrix}
\right \},
\end{equation}
where ``$\mathrm{co}$'' is the convex hull. Since $T_{p_0} \mathbb S^1 = \{ v \in \R^2 : v^\trans p_0 = 0 \} = \{ (t, -t) \in \R^2 : t \in \R \}$, 
then a linear operator on $T_{p_0}\mathbb S^1$ can be uniquely determined by a real number $a \in \R$, i.e., $(t, -t) \mapsto a(t, -t)$. As $\Proj_{p_0}(\partial\bar X(p_0))\subset \cL(T_{p_0}\mathbb S^1)$, we know the set $\Proj_{p_0}(\partial\bar X(p_0))$ is equivalent to $\mathrm{co}\{ \pm 3, \pm 1 \} = [-3, 3]$.

Next, we consider $\partial X(p_0)$.
Let $p_k \in \mathbb S^1$ be a sequence converging to $p_0$. 
We may assume $p_k = (\cos \theta_k, \sin \theta_k)$, where $\theta_k \to \pi /4$ and $\theta_k \neq \pi /4$. 
The derivative of $X(p_k)$ can be given by \eqref{eqn:S1-euc-proj}.
The second term in \eqref{eqn:S1-euc-proj} converges to $0$ since $Y(p_0) = 0$.
Observe that
$\bar \nabla Y(p_k) = \mathrm{diag}(2, -4)$ when $\theta_k > \pi / 4$ and $\bar \nabla Y(p_k) = \mathrm{diag}(-2, 4)$ when $\theta_k < \pi /4$.
Then, we know the Clarke generalized covariant derivative of $X$ at $p_0$ is
\[ \begin{aligned}
	 \partial X(p_0) 
= \mathrm{co} \left \{ 
	\pm \begin{pmatrix}
		1 & 2 \\ -1 & -2
	\end{pmatrix} \right \},
\end{aligned}\]
which is equivalent to $\mathrm{co} \{ \pm 1 \} = [-1, 1]$. It is clear that 
$[-1,1]\subset[-3,3]$, and thus $\partial X(p_0) \subset \Proj_p(\bar \partial \bar X(p))$.
This strict inclusion is mainly because we can only find the two tangent directions on $\mathbb S^1$ converging to $p_0$, while the other two normal directions are available only in $\R^2$. 
\end{example}
 We make the next assumption on the vector field $X$, which covers the problems in our experiments.
\begin{assumption}\label{ass:vector_field}
	For any  $p = (p_1, \dots, p_d) \in \cM$, the vector field $X(p) := F(p, f_1(p_1),\dots, f_d(p_d))\allowbreak \in T_p\cM$ is locally Lipschitz, where
	\begin{enumerate}[(i)]
		\item $F: \R^d \times \R^{n_1} \times \dots \times \R^{n_d} \to T_p \cM$ is continuously differentiable;
		\item $f_j: \R \to \R^{n_{j}}$, $j=1,\ldots,d$ and the non-differentiable points of $f_j$ are isolated, i.e., for any point $q\in\R$, there exists some $\delta>0$ such that $f_j$ is continuously differentiable on $(q-\delta,q+\delta)\backslash\{q\}$;
		\item The left and right derivatives of $f_j$ exist.
	\end{enumerate} 
\end{assumption}
Based on the above assumptions, the next lemma finds an element in $\partial X(q)$ by choosing a proper path that converges to the point $q \in \cM$.
\begin{lemma}
	\label{lemma:calc-clarke}
	Fix $q \in \cM$.
	Suppose that Assumption~\ref{ass:vector_field} holds, and $\{q^{(n)}\}\subset\cM$ is a sequence converging to $q$ such that for each $j\in[d]$, 
	it holds that either $q_j^{(n)} > q_j$ for all $n \in \N$, or $q_j^{(n)}< q_j$ for all $n \in \N$.
	Then, we have 
	$\Proj_q \circ dF|_q \circ (\mathrm{id}_{\R^d}, r_1, r_2, \dots, r_d) \in \partial X(q)$,
	where $dF|_q$ is the differential of $F: \R^{d + \sum_{i=1}^d n_i} \to \R^d$ at $(q, f_1(q_1), \dots, f_d(q_d))$, and $r_j: T_q \cM \to \R^{n_j}$ is a linear operator such that
	\begin{equation}
	r_j(v)  = \begin{cases}
	v_j \lim\limits_{t \downarrow q_j} f^\prime_j(t), & \text{ if } q^{(n)}_j > q_j,\\
	v_j \lim\limits_{t \uparrow q_j} f^\prime_j(t), & \text { if } q^{(n)}_j < q_j.
	\end{cases}
	\end{equation}
\end{lemma}
\begin{proof}
	This is a direct consequence of Definition~\ref{def:clarke-generalized-covariant} and \eqref{eqn:hess-euc}.
	\hfill $\Box$
\end{proof}

It is noted that the existence of the sequence $\{q^{(n)}\}$ depends on the manifold. The next theorem gives a construction of such a sequence on the Stiefel manifold, which relies on the following assumption.

\begin{assumption}\label{ass:projection_to_tangent}
    Let $Q \in \St{n, r}$ and $V  \in T_Q\St{n, r}$. Then, for all $i \in [n], j \in [r]$, it holds that $Q_{ij}\in\{\pm1\}$ whenever $V_{ij}=0$.
\end{assumption}

\begin{theorem}
	\label{thm:clarke-stiefel}
	Let $Q \in \St{n, r}$, $V\in T_Q\St{n, r}$, and $X(P) = F( P, f_{11}(P_{11}), f_{12}(P_{12}), \dots, f_{nr}(P_{nr}))$ be a locally Lipschitz vector field on $\St{n,r}$ such that Assumption~\ref{ass:vector_field} and \ref{ass:projection_to_tangent} hold.
	Then, 
	    $\Proj_Q \circ dF|_Q \circ (\mathrm{id}_{\R^{n\times r}}, H_{11}, \dots, H_{nr}) \in \partial X(Q)$,
	where $H_{ij}: T_Q\St{n, r} \to \R^{n_{ij}}$ is a linear map satisfying $H_{ij}(W) = W_{ij} D_{ij}$ and 
	\[ D_{ij} =  \begin{cases}
	\lim_{t \downarrow Q_{ij}} f^\prime_{ij}(t), & V_{ij} > 0 \text{, or } V_{ij} = 0 \text{ but } Q_{ij} = -1,  \\
	\lim_{t \uparrow Q_{ij}} f^\prime_{ij}(t), & V_{ij} < 0 \text{, or } V_{ij} = 0 \text{ but } Q_{ij} = 1.
	\end{cases} \]
\end{theorem}
\begin{proof}
	When $V_{ij} > 0$, we consider the sequence $Q^{(n)} = \exp_Q(t_nV)$ with $t_n \downarrow 0$.
	Since $\left .\frac{d}{dt} \right|_{t = 0} \exp_Q(tV) = V$, then $Q^{(n)}_{ij} > Q_{ij}$ for sufficiently large $n$. 
	In the case where $V_{ij} = 0$ and $Q_{ij} = -1$, 
	since $P^\trans P = I_r$ for every $P \in \St{n, r}$, then there exists a sequence $Q^{(n)}$ converging to $Q$ such that $Q^{(n)}_{ij} > -1 = Q_{ij}$ for all $n \in \N$.
	The other two cases are similar, and hence we can use Lemma~\ref{lemma:calc-clarke} to find the derivative.
	\hfill $\Box$
\end{proof}

The following proposition shows that for a fixed $Q \in \St{n, r}$, Assumption~\ref{ass:projection_to_tangent} holds for almost every $V \in T_Q \St{n, r}$.
\begin{proposition}
	Let $Q \in \St{n, r}$ and $\cI = \{ (i, j) : V_{ij} = 0$ for all $V \in T_Q\St{n, r} \}$, then the following properties hold:
	\begin{enumerate}[(i)]
		\item If $(i, j) \in \cI$, then $Q_{ij} \in \{ \pm 1 \}$;
		\item If $(i, j) \notin \cI$, then the set $\cZ_{ij} := \{ Z \in \R^{n\times r} : (\Proj_QZ)_{ij} = 0 \}$ has zero Lebesgue measure.
	\end{enumerate}
\end{proposition}
\begin{proof}
	First, the second property holds since $\cZ_{ij}$ is a linear subspace in $\R^{n \times r}$ with $\dim \cZ_{ij} < nr$.

	Next, we prove that $Q_{ij}\in\{\pm1\}$ for all $(i,j)\in\cI$.
	If $(i_0, j_0) \in \cI$, then the matrix $N \in (T_Q\St{n, r})^\perp$, where $N_{ij} = 1$ if $i = i_0, j = j_0$ and $N_{ij} = 0$ otherwise.
	From \cite[Example 3.6.2]{absil2009optimization}, the normal space can be written as $(T_Q\St{n, r})^\perp = \{ QS: S\in\R^{r\times r},S^\trans = S \}$.
	Without the loss of generality, we can assume $i = j = 1$ and there exists $S = S^\trans$ such that $QS = N$, and then
	the equations $QS = N, Q^\trans Q = I_r$ can be rewritten using block matrices:
	\[ 
	\begin{pmatrix}
		a & x^\trans \\
		y & A
	\end{pmatrix} 
	\begin{pmatrix}
		b & z^\trans \\
		z & B
	\end{pmatrix} 
	=
	\begin{pmatrix}
		1 & 0 \\
		0 & 0
	\end{pmatrix} ,
	\quad 
	\begin{pmatrix}
		a & y^\trans \\
		x & A^\trans
	\end{pmatrix} 
	\begin{pmatrix}
		a & x^\trans \\
		y & A
	\end{pmatrix} 
	=
	\begin{pmatrix}
		1 & 0 \\
		0 & I_{r-1}
	\end{pmatrix} ,
	\]
	where $a, b \in \R$, $y, z \in \R^{r - 1}$, $x \in \R^{n-1}$, $A \in \R^{(n-1)\times (r-1)}$ and $B^\trans = B \in \R^{(r-1)\times (r - 1)}$.

	From $a(ab + x^\trans z) = a$ and $y^\trans (yb + Az) = 0$, note that $a^2 + y^\trans y = 1$ and $ax^\trans + y^\trans A = 0$, we know $a = b$.
	Similarly, $B = 0$ holds by combining equations $x(az^\trans + x^\trans B) = 0$, $A^\trans(yz^\trans + AB) = 0$, $xa + A^\trans y = 0$ and $xx^\trans + A^\trans A = I_{r - 1}$.
	Note $yz^\trans + AB = 0$ and $B = 0$, then either $y = 0$ or $z = 0$ holds.
	Therefore, $1 = a^2 + y^\trans y = a^2$ when $y = 0$; and $1 = ab + x^\trans z = a^2$ when $z = 0$.
	Thus, $Q_{11} = a = \pm 1$. \hfill $\Box$
\end{proof}

When $Z \in \R^{n\times r}$ is randomly sampled from a probability measure that is absolutely continuous with respect to the Lebesgue measure in $\R^{n\times r}$, the tangent vector $V := \Proj_Q Z$ fulfills Assumption~\ref{ass:projection_to_tangent} almost surely.
Therefore, we derive the following algorithm, which finds an element of the Clarke generalized covariant derivative almost surely.

\begin{alg}
	\label{alg:compute-clarke}
{\bf Input:} $Q\in\St{n,r}$, {\bf Output:} $H\in \partial X(Q)$.
\begin{enumerate}[label=(\roman*)]
    \item Sample $Z \in \R^{n \times r}$ from the standard Gaussian distribution.
    \item Calculate the projection $V = \Proj_QZ$.
    \item If there exists $(i, j)$ such that $V_{ij} = 0$ and $Q_{ij} \notin \{ \pm 1\}$, re-run (i)-(ii).
    \item Calculate $H$ by Theorem~\ref{thm:clarke-stiefel}.
\end{enumerate}
\end{alg}

\section{Numerical Experiments}
\label{sec:exp}
In this section, we evaluate our algorithm on three problems mentioned before: compressed modes (CM)~\cite{ozolina2013compressed}, sparse PCA and the constrained sparse PCA~\cite{lu2012spca}. 
In CM and SPCA, we compare our algorithm with SOC~\cite{lai2014splitting}, ManPG~\cite[ManPG-Ada (Algorithm 2)]{chen2020proximal}, accelerated ManPG (AManPG)~\cite{huang2019extending}, and accelerated Riemannian proximal gradient (ARPG)~\cite{huang2019riemannian}. 
In the constrained SPCA, we compare our algorithm with ALSPCA~\cite{lu2012spca}. 
Codes of SOC and ManPG are provided in~\cite{chen2020proximal}, codes of AManPG and ARPG are provided by~\cite{huang2019riemannian}, and the code of ALSPCA is provided by~\cite{lu2012spca}. 
The ManOPT package is used in our implementation~\cite{manopt}.
All codes are implemented in MATLAB and evaluated on Intel i9-9900K CPU. 
Reported results are averaged over $20$ runs with different random initial points.

\subsection{Compressed Modes} \label{sec:exp-cm}
\begin{figure}[t]
	\centering
	\begin{subfigure}{0.328\linewidth}
		\centering
		\includegraphics[width=\linewidth]{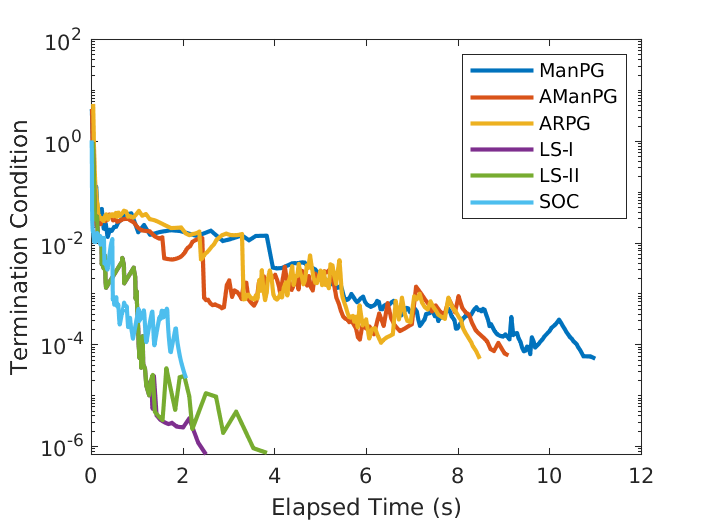}
	\end{subfigure} \hfill
	\begin{subfigure}{0.328\linewidth}
		\centering
		\includegraphics[width=\linewidth]{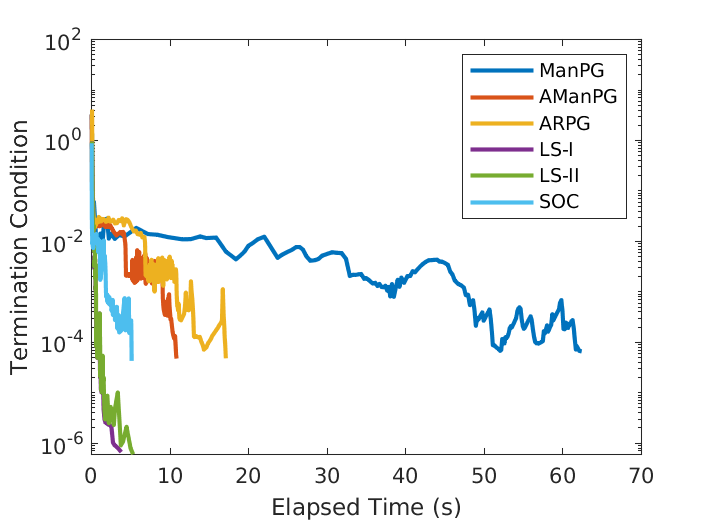}
	\end{subfigure} \hfill
	\begin{subfigure}{0.328\linewidth}
		\centering
		\includegraphics[width=\linewidth]{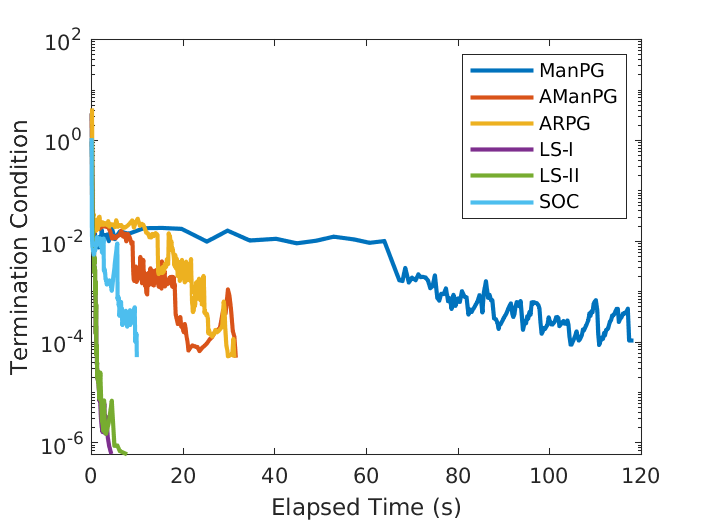}
	\end{subfigure}\\
	\begin{subfigure}{0.328\linewidth}
		\centering
		\includegraphics[width=\linewidth]{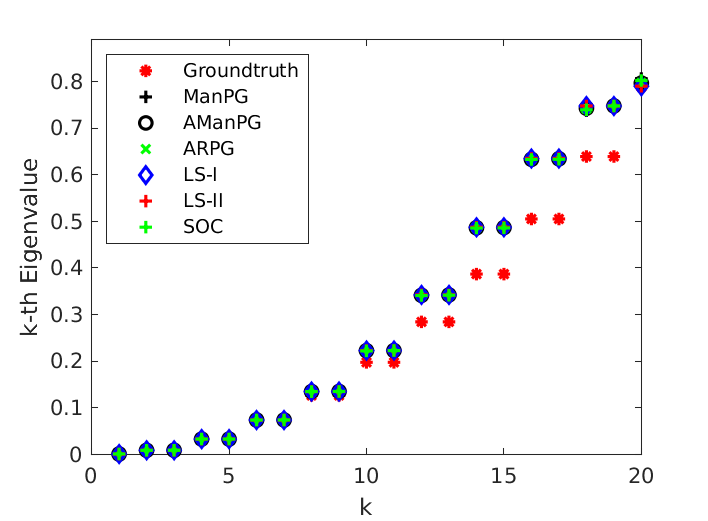}
	\end{subfigure} \hfill
	\begin{subfigure}{0.328\linewidth}
		\centering
		\includegraphics[width=\linewidth]{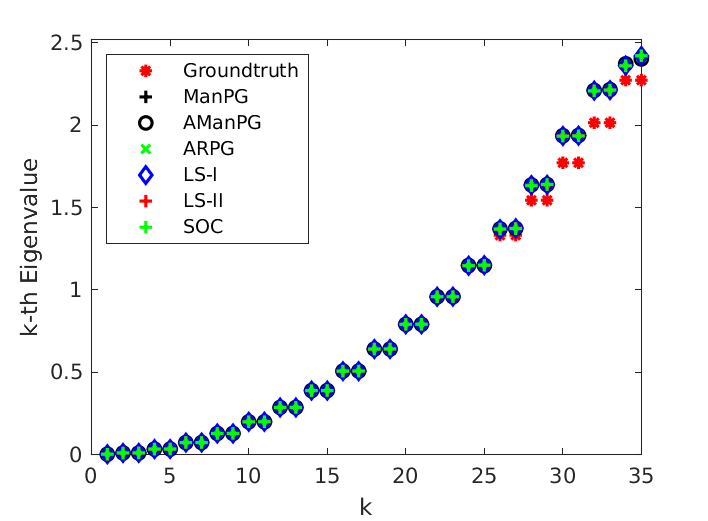}
	\end{subfigure} \hfill
	\begin{subfigure}{0.328\linewidth}
		\centering
		\includegraphics[width=\linewidth]{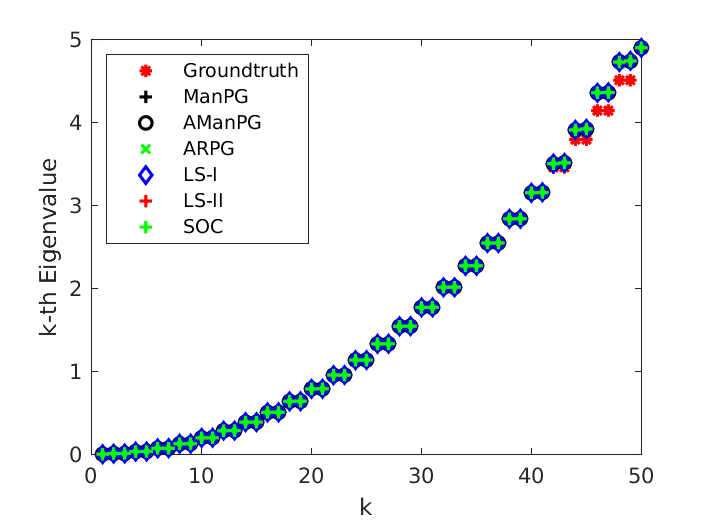}
	\end{subfigure}
	\caption{Comparisons on CM with $(n, \mu) = (500, 0.05)$ and $r = 20, 35, 50$. The top row plots the termination condition, which is the sum of the left-hand parts in \eqref{eqn:soc-term-1} and \eqref{eqn:soc-term-2} for SOC, and similar for other algorithms. The bottom row plots the eigenvalues of $Q^\trans HQ$, where $Q$ is the solution of the corresponding method. ``LS-I'' and ``LS-II'' refer to our algorithms using the exponential map.}
	\label{fig:cm}
\end{figure}

\begin{figure}[t]
	\centering
	\begin{subfigure}{0.328\linewidth}
		\centering
		\includegraphics[width=\linewidth]{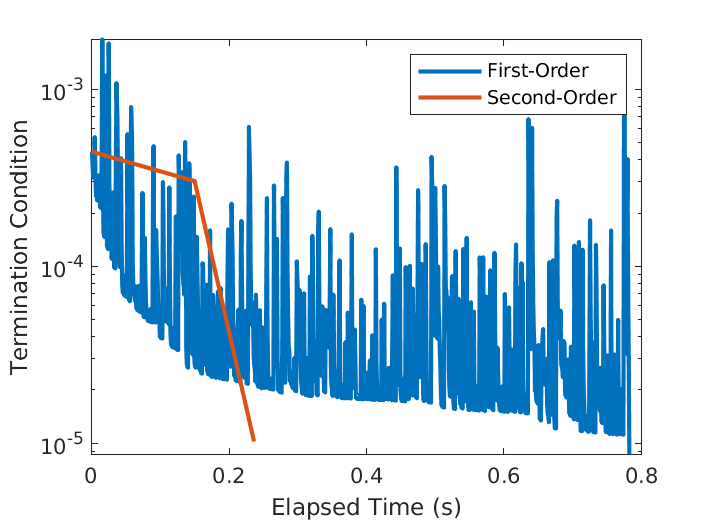}
	\end{subfigure} \hfill
	\begin{subfigure}{0.328\linewidth}
		\centering
		\includegraphics[width=\linewidth]{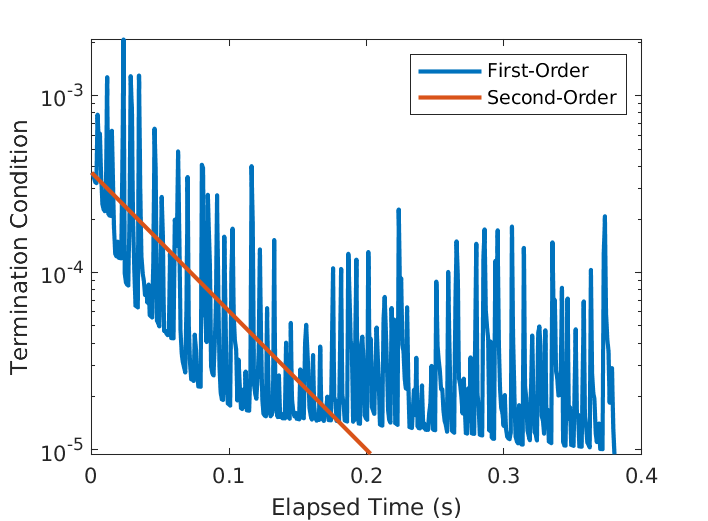}
	\end{subfigure} \hfill
	\begin{subfigure}{0.328\linewidth}
		\centering
		\includegraphics[width=\linewidth]{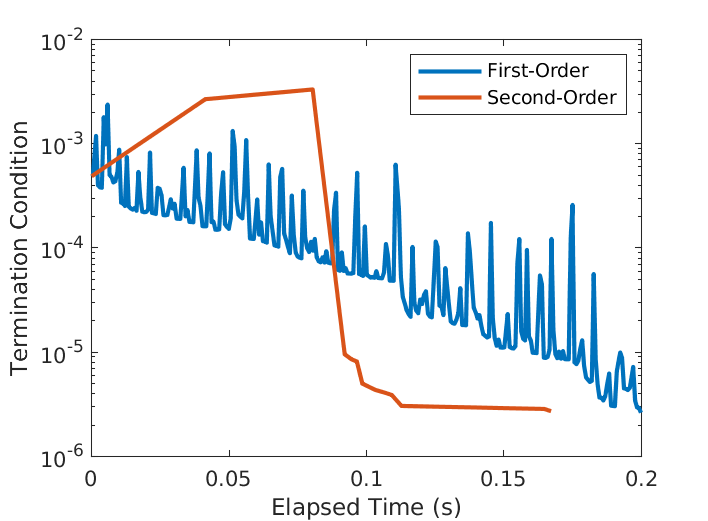}
	\end{subfigure}\\
	\begin{subfigure}{0.328\linewidth}
		\centering
		\includegraphics[width=\linewidth]{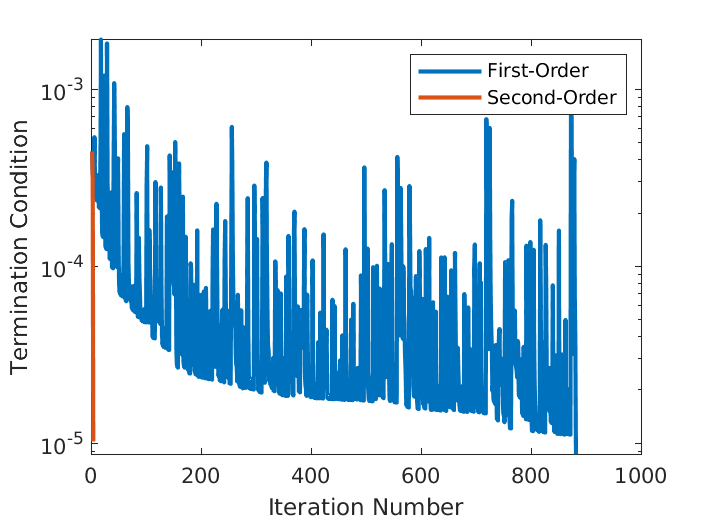}
	\end{subfigure} \hfill
	\begin{subfigure}{0.328\linewidth}
		\centering
		\includegraphics[width=\linewidth]{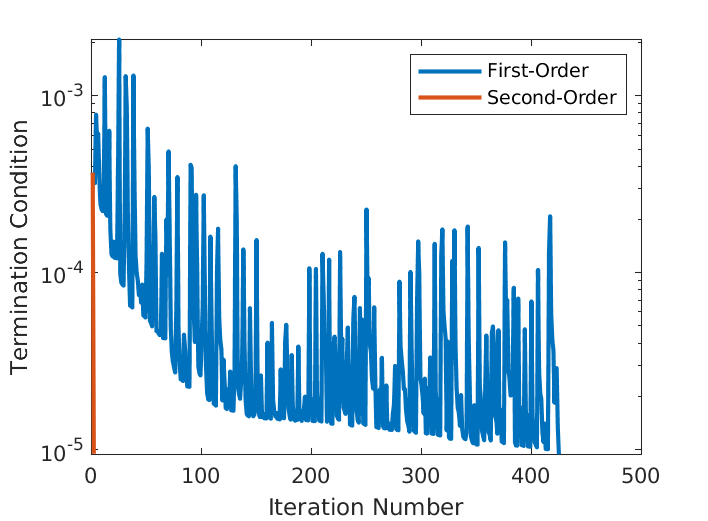}
	\end{subfigure} \hfill
	\begin{subfigure}{0.328\linewidth}
		\centering
		\includegraphics[width=\linewidth]{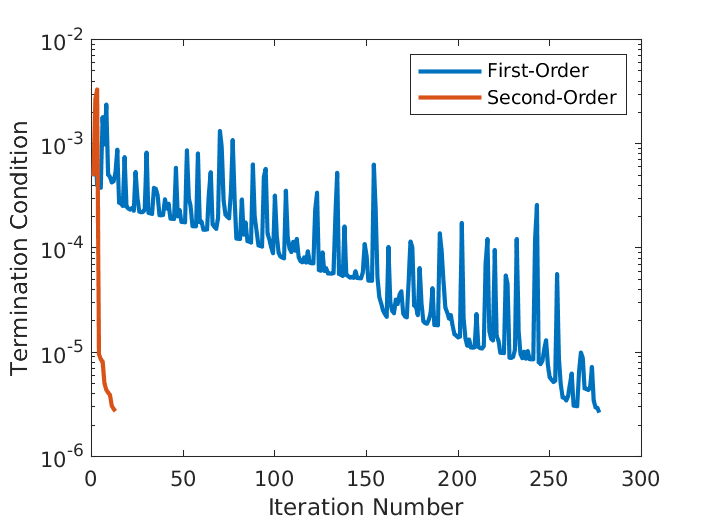}
	\end{subfigure}
	\caption{Comparisons on the first-order method~\cite{wen2013feasible} and the second-order method in solving the CM subproblem \eqref{eqn:Lk}. Columns represent the behavior of these methods in solving different subproblems in Algorithm~\ref{alg:alm-outer}.}
	\label{fig:cm-first-second}
\end{figure}

\begin{figure}[t]
	\centering
	\begin{subfigure}{0.328\linewidth}
		\centering
		\includegraphics[width=\linewidth]{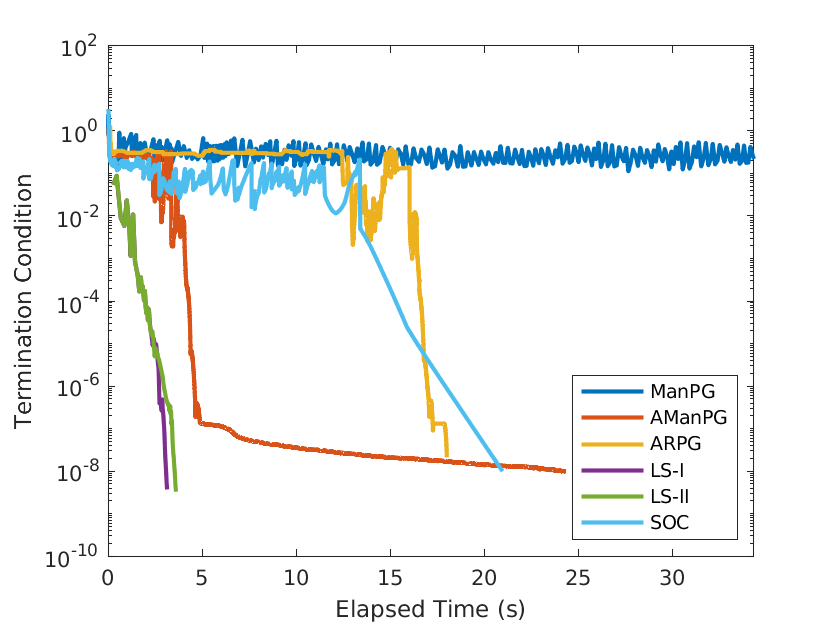}
	\end{subfigure} 
	\begin{subfigure}{0.328\linewidth}
		\centering
		\includegraphics[width=\linewidth]{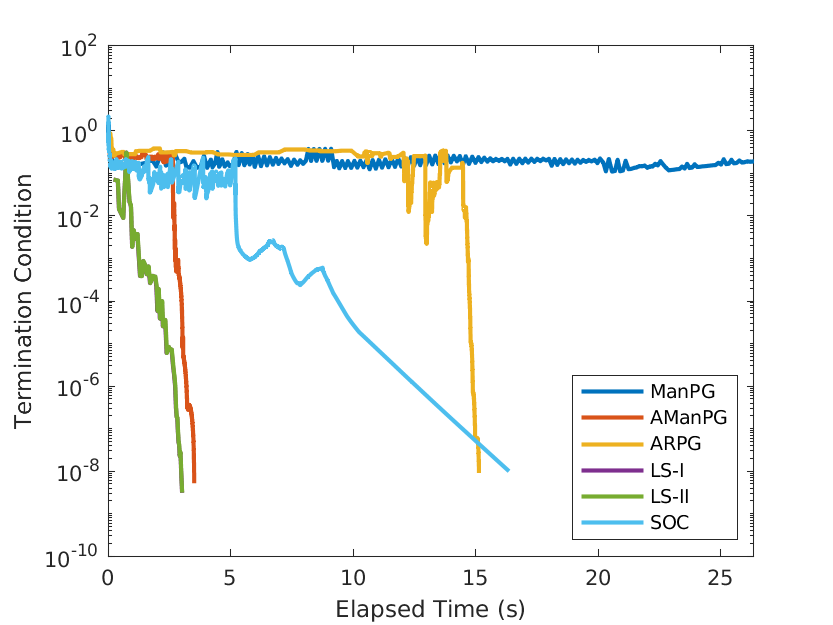}
	\end{subfigure} 
	\begin{subfigure}{0.328\linewidth}
		\centering
		\includegraphics[width=\linewidth]{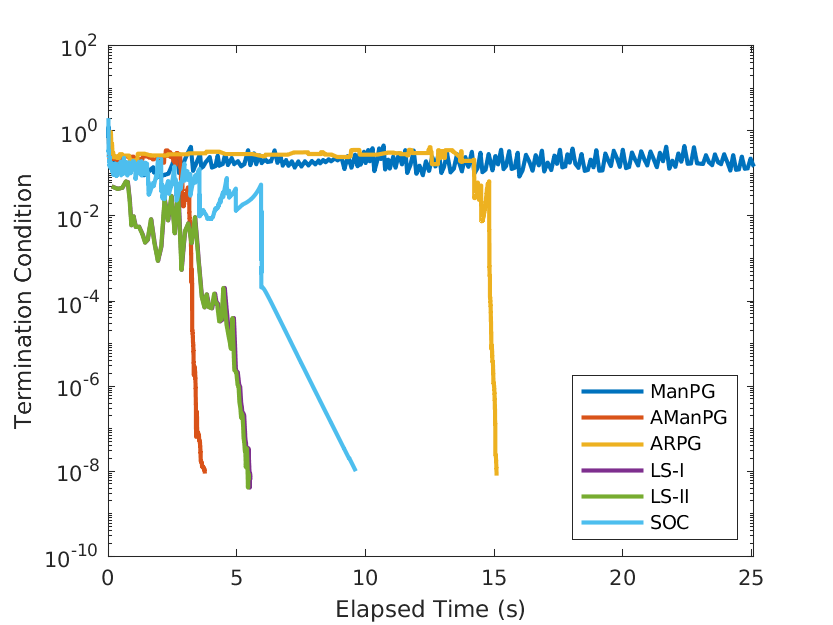}
	\end{subfigure} 
	\begin{subfigure}{0.328\linewidth}
		\centering
		\includegraphics[width=\linewidth]{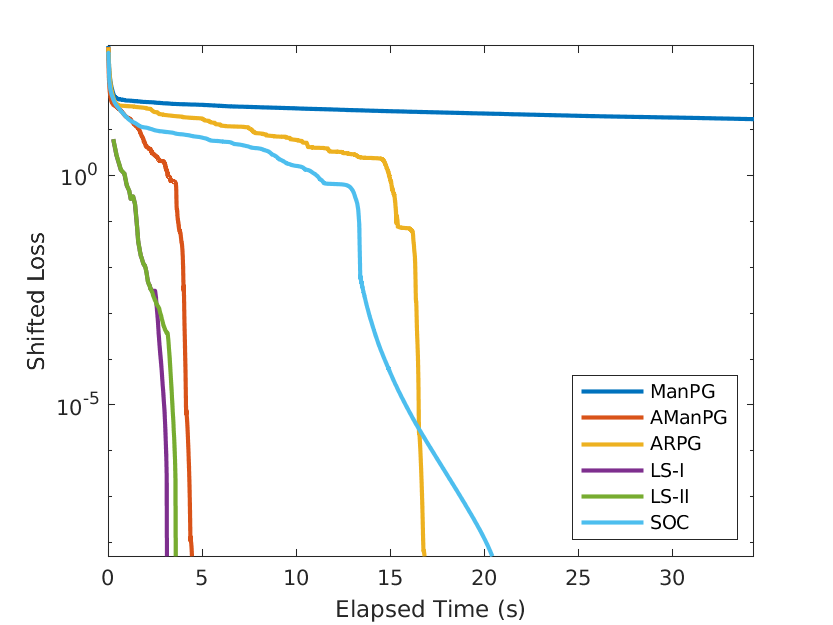}
	\end{subfigure} 
	\begin{subfigure}{0.328\linewidth}
		\centering
		\includegraphics[width=\linewidth]{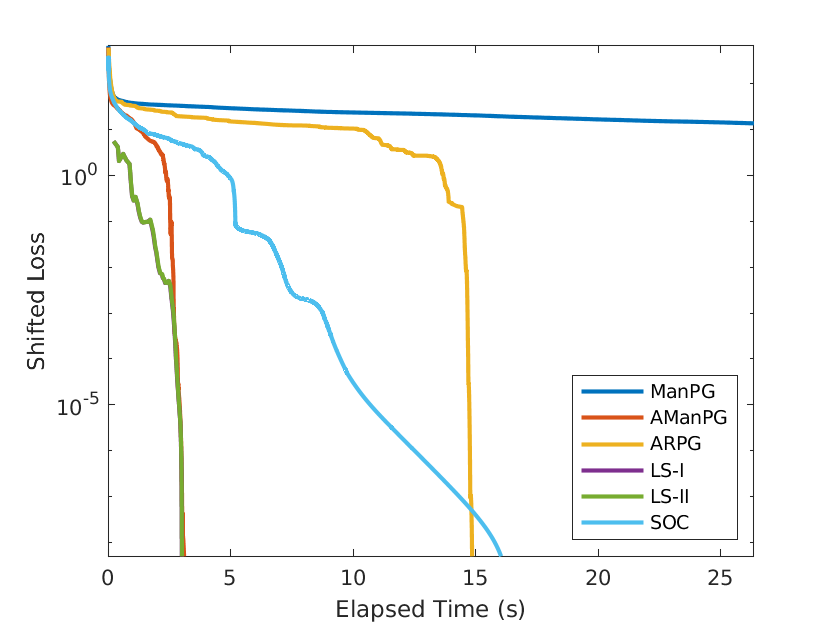}
	\end{subfigure} 
	\begin{subfigure}{0.328\linewidth}
		\centering
		\includegraphics[width=\linewidth]{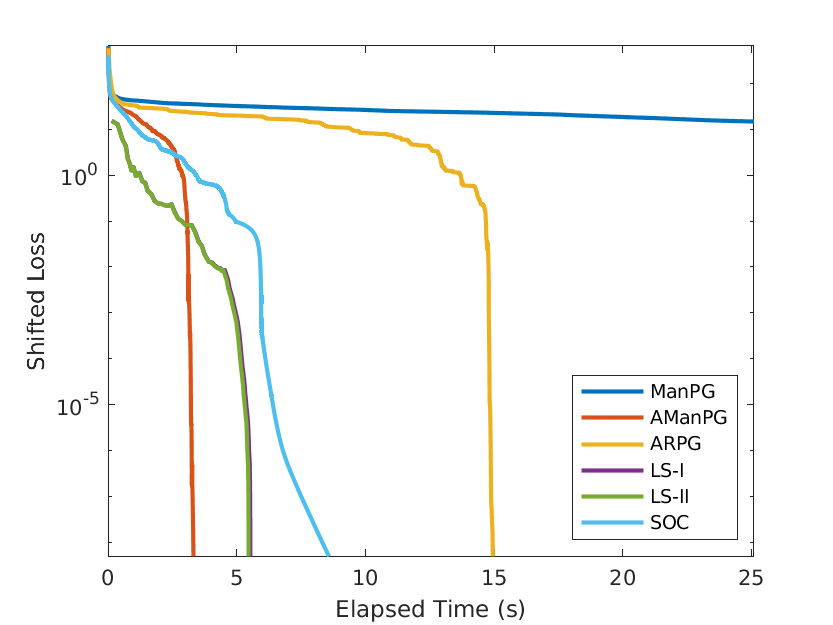}
	\end{subfigure} 
	\caption{Comparisons on SPCA with $(n, r, \mu) = (500, 20, 1.00)$ and three different $A^\trans A$. The shifted loss is the loss subtracted by the minimal loss among all iterations, i.e., $F(x_k) - \min_j F(x_j)$. 
	``LS-I'' and ``LS-II'' refer to our algorithm using the retraction based on the QR decomposition.}
	\label{fig:spca-comparison}
\end{figure}

We consider the CM problem \eqref{eqn:cm-problem} and follow the setting of \cite{chen2020proximal} to solve the Schr\"odinger equation of 1D free-electron model with periodic boundary condition:
\[ -\frac{1}{2}\Delta \phi(x) = \lambda \phi(x), \quad x \in [0, 50]. \]

We discretize the domain $[0, 50]$ into $n$ nodes, and let $H$ be the discretized version of $-\frac{1}{2}\Delta$. The CM problem needs to solve \eqref{eqn:cm-problem}.
The first-order optimality condition is 
\begin{equation}\label{eqn:opt-cond-cm}
    0 \in 2\Proj_Q(HQ) + \mu \Proj_Q(\partial \norm{Q}_1).
\end{equation}
It is worth mentioning that this condition is difficult to check in general because of the existence of the projection.
Recall that ManPG solves the following subproblem:
\[ V_* = \argmin_{V \in T_Q \St{n, r}} \left \{ \inner{\grad_Q \tr(Q^\trans HQ)}{V} + \frac{1}{2t}\norm{V}_F^2 + \norm{Q + V}_1 \right \}. \]
The solution of the above problem satisfies
\[ -V_*/ t \in 2\Proj_Q(HQ) + \mu \Proj_Q\left (\partial \norm{Q + V_*}_1 \right ). \]
As suggested by \cite{chen2020proximal, huang2019extending}, ManPG and AManPG use $t^{-1} \norm{V_*}_\infty / (\|Q\|_F + 1) \leq 5 \times 10^{-5}$ as the termination condition, which can be regarded as an approximation of the optimality condition \eqref{eqn:opt-cond-cm} of the CM problem.

For SOC and our algorithm, the termination rules are their KKT conditions. %
Specifically, SOC rewrites \eqref{eqn:cm-problem} into the following problem:
\begin{equation}
	\label{eqn:soc-problem}
	\begin{aligned}
		\min_{Q, P, R \in \R^{n\times r}}\  & \left \{ \tr(P^\trans HP) + \mu\norm{R}_1 \right \}, \\
		\mathrm{s.t.}\ & Q = P, R = P, Q \in \St{n, r}. 
	\end{aligned}
\end{equation}

The Lagrangian of \eqref{eqn:soc-problem} is $L_{S}(Q, P, R, \Gamma, \Lambda) = \tr(P^\trans HP) + \mu\norm{R}_1 + \Gamma^\trans (Q - P) + \Lambda^\trans (R - P)$, where $P, R \in \R^{n\times r}$ and $Q \in \St{n, r}$.
We terminate SOC when both the following conditions are satisfied
\begin{align}
	\frac{\norm{Q - P}_\infty}{\max\{ \norm{Q}_F, \norm{P}_F \} + 1}
	+ \frac{\norm{R - P}_\infty}{\max\{\norm{R}_F, \norm{P}_F\} + 1} 
	&\leq 5 \times 10^{-7}, \label{eqn:soc-term-1} \\
	\frac{\norm{\grad_Q L_S}_\infty}{\norm{Q}_F + 1}
	+ \frac{\norm{\nabla_P L_S}_\infty}{\norm{P}_F + 1}
	+ \frac{\min_{G \in \partial_R L_S}\norm{G}_\infty}{\norm{R}_F + 1}
	&\leq 5 \times 10^{-5}. \label{eqn:soc-term-2}
\end{align}

Similarly, our algorithm rewrites the problem into
\begin{equation}
	\label{eqn:our-problem}
	\begin{aligned}
		\min_{Q, R \in \R^{n\times r}}\  & \left \{ \tr(Q^\trans HQ) + \mu\norm{R}_1 \right \}, \\
		\mathrm{s.t.}\ & Q = R, Q \in \St{n, r}. 
	\end{aligned}
\end{equation}

The Lagrangian is $L_{N}(Q, R, \Lambda) = \tr(Q^\trans HQ) + \mu\norm{R}_1 + \Lambda^\trans (Q - R)$, where $R \in \R^{n\times r}$ and $Q \in \St{n, r}$.
The termination conditions are both \eqref{eqn:term-ours-feasibility} and \eqref{eqn:term-ours-stationary}:
\begin{align}
	\frac{\norm{Q - R}_\infty}{\max\{\norm{Q}_F, \norm{R}_F\} + 1}
	&\leq 5 \times 10^{-7}, 
	\label{eqn:term-ours-feasibility}
	\\
	\frac{\norm{\grad_Q L_N}_\infty}{\norm{Q}_F + 1}
	+ \frac{\min_{G \in \partial_R L_N}\norm{G}_\infty}{\norm{R}_F + 1}
	&\leq 5 \times 10^{-5}.
	\label{eqn:term-ours-stationary}
\end{align}

Below we illustrate how to apply our algorithm in this problem. In the subproblem of our algorithm, we need to find an approximated stationary point of
$L_k(Q) = \tr(Q^\trans HQ) + G_{\sigma}(Q + \Lambda / \sigma)$,
where $G_{\sigma}$ is the Moreau-Yosida regularization of $\lambda \norm{\cdot}_1$. Let $\bar L_k: \R^{n\times r} \to \R$ be the extension of $L_k$ into the Euclidean space, i.e., $\bar L_k|_{\St{n, r}} = L_k$. The Euclidean gradient and Hessian of $\bar L_k$ at a differentiable point $Q$ can be easily computed as
\begin{align}
	\label{eqn:Lk-grad-cm}
	\nabla \bar L_k(Q) &= 2HQ + (\sigma Q + \Lambda) - \sigma F, \\
	\nabla^2 \bar L_k(Q)[Z] &= 2HZ + \sigma Z\odot E,
\end{align} 
where $\odot$ is the Hadamard product and $E, F \in \R^{n\times r}$ are matrices such that $E_{ij} = \bone_{\{ |Q_{ij} + \Lambda_{ij}/\sigma|\leq \lambda / \sigma \}}$ and $F_{ij} = \prox_{\lambda\norm{\cdot}_1/\sigma}(Q_{ij} + \Lambda_{ij}/\sigma)$. %
When $\bar L_k$ is differentiable at $Q$, the Riemannian gradient and Hessian can be written as~\cite{absil2009optimization}
 \begin{align}
	\grad L_k(Q) &= \Proj_Q \nabla \bar L_k(Q), \\
	\Hess L_k(Q)[Z] &= \Proj_Q (\nabla^2 \bar L_k(Q)[Z] - Z\,\mathrm{sym}(Q^\trans\nabla\bar L_k(Q))),
\end{align} 
where $\mathrm{sym}\, Q := (Q + Q^\trans) / 2$.
We note that $\grad L_k$ fulfills Assumption~\ref{ass:vector_field} with $f_{ij}(Q_{ij})=F_{ij}$, and thus Algorithm~\ref{alg:compute-clarke} can be applied to find an element of $\partial \grad L_k$ when $\bar L_k$ is non-differentiable at $Q$.

The parameters of ManPG and SOC are the same as those in \cite{chen2020proximal}. 
The codes and parameters of AManPG and ARPG are adopted from the SPCA implementation of Huang et al. \cite{huang2019riemannian}. 
All the four methods are terminated when their termination conditions are satisfied or the number of iterations exceeds $30000$. 

We consider the exponential map and two retractions from \cite[Example 4.1.3]{absil2009optimization} in Algorithm~\ref{alg:semismooth}.
One of these retractions is based on the polar decomposition:
\[ R^{\text{polar}}_X(V) := (X + V)(I + V^\trans V)^{-\frac{1}{2}}, \]
and the other one is based on the QR decomposition:
\[ R^{\text{QR}}_X(V) := \text{qf}(X + V), \]
where $\text{qf}(X)$ is the $Q$ component in the QR decomposition of $X$.

In our algorithm, let $\Delta_k, \Gamma_k$ be the left-hand sides of \eqref{eqn:term-ours-feasibility} and \eqref{eqn:term-ours-stationary} evaluated at the $k$-th step, respectively.
We set $\tau = 0.97$, $\rho = 1.25$, $\alpha = 1.01$, $\sigma_0 = 1$, and $\varepsilon_k = \min \{ 0.95^k, 5 \Delta_k \}$, in which the $0.95^k$ component ensures that $\varepsilon_k \to 0$ and the $5\Delta_k$ component suggests that the violation of the stationary condition \eqref{eqn:term-ours-stationary} would not be large when the violation of the feasibility condition \eqref{eqn:term-ours-feasibility} is small.
We also update $\sigma_k$ using \eqref{eqn:alm-update-penalty} when $\Delta_k > 2.5\Gamma_k$, i.e., the converse of the aforementioned case.

To solve our subproblem,
we use the first-order method to find a good initial point, and start the second-order method when $\norm{\grad L_k} < 5 \times 10^{-4} =: \Delta_G$, where $L_k$ is defined in \eqref{eqn:Lk}. 
The maximum numbers of iterations in CG and the first-order method are both $1000$.
In Algorithm~\ref{alg:semismooth}, the linesearch parameter is $\mu = 0.1$. We set $\omega_k = \min \{ 0.7^k, 200 \| X(p_k) \| \}$ and the minimal stepsize of the linesearch is $10^{-4}$. 
When CG finds a negative direction $p_k$, we set $\omega_k = -2\mathrm{Hess}\, L_k(Q_k) [p_k, p_k] / \norm{p_k}^2$ and restart the iteration.
Parameters in the condition \eqref{eqn:sufficient-descent} are $p = 0.05, \beta_0 = 1, \beta_1 = 10^{-3}$.
The threshold $\Delta_G$ will be modified during the iterations when one of the following cases occurs.
\begin{itemize}
    \item When the linesearch does not find an $m_k$ such that $\delta^{m_k} > 10^{-4}$, we replace $\Delta_G$ with $0.95\Delta_G$.
    \item When the number of iterations of Algorithm~\ref{alg:semismooth} exceeds $10$, we replace $\Delta_G$ with $0.9\Delta_G$.
    \item When $\| V_k \| > 10^4$, i.e., $H_k + \omega_k I$ is nearly singular, we replace $\Delta_G$ with $0.8\Delta_G$.
\end{itemize}
These modifications of $\Delta_G$ may be helpful when Algorithm~\ref{alg:semismooth} starts too early. 
We also note that from Table~\ref{tab:cm-newton-it} the second-order method generally starts (i.e., the condition $\| \grad L_k \| < \Delta_G$ holds) during nearly the last $1/4$ iterations of Algorithm~\ref{alg:alm-outer}.

We report the results in Table~\ref{tab:cm} and Figure~\ref{fig:cm}, from which
we see that all methods find solutions with similar objective function values and comparable sparsity levels.
It is noted that our method is generally faster than (or comparable to) other methods. 
The running time for different retractions in our algorithm are close, 
except that when $r = 10$ and $15$, LS-I seems extremely slow because the second-order method starts too early such that most evaluations of second-order directions are wasted.
This issue could be addressed by a careful tuning of parameters.
From Figure~\ref{fig:cm}, we find ManPG has difficulty attaining the desired termination condition when $r$ is large.
The performance of the two linesearch methods are similar, except for the aforementioned cases in which  the second-order method does not start appropriately.

Next, we examine whether the second-order method is helpful in this problem. 
We terminate our method when both \eqref{eqn:term-ours-feasibility} and \eqref{eqn:term-ours-stationary} are less than $5 \times 10^{-8}$ and report the number of iterations and the computational cost of the second-order method used to solve subproblems.\footnote{We only consider our method in this setting since other algorithms mentioned in Table~\ref{tab:cm} have difficulty in reaching these stopping conditions. Indeed, these observations are also valid for the original termination conditions  with a lower significance.}
We also try to run the first-order method~\cite{wen2013feasible} from the same initial point in each subproblem and terminate it when it attains the same accuracy as the second-order method.
Results are illustrated in Figure~\ref{fig:cm-first-second} and reported in Table~\ref{tab:cm-high}.
From Table~\ref{tab:cm-high} and the bottom row of Figure~\ref{fig:cm-first-second}, we see that to achieve the same accuracy, the number of iterations of the second-order method is significantly smaller than the first-order method. 
Since the second-order method requires solving a linear equation in each iteration, the first-order method may be faster than the second-order method in terms of the computational cost as can be observed in Table~\ref{tab:cm-high}.
However, by comparing the computational cost of these two methods in LS-I, we see that the second-order method can be faster in some cases due to its fast convergence, which is also illustrated in the top row of Figure~\ref{fig:cm-first-second}.
The average iteration number of CG and the percentage that CG detects a negative curvature direction are also reported in Table~\ref{tab:cm-high}, 
from which  we observe that negative curvature directions are frequently detected when $r = 10$ and $15$. 
We suspect that this is possibly because the CM problem is more ``singular'' in these cases, which may also be supported by Table~\ref{tab:cm-eigen}, in which the minimum eigenvalue is found to be significantly smaller.

Finally, although the CM problem appears ``convex'', the positive definiteness of $\partial \grad L_k$ required by Theorem~\ref{thm:semismooth-LS-I-rate} is neither obvious nor easy to theoretically verify.
Indeed, finding all elements in $\partial \grad L_k$ is also difficult (see the discussions in Sec.~\ref{sec:calc-clarke}).
We partially verify this condition using numerical simulations and report the minimum eigenvalue of one element in $\partial \grad L_k$. The results in Table~\ref{tab:cm-eigen} suggest that the positive-definite condition may be fulfilled in this problem.

We also note that in the Euclidean setting the sparsity may be exploited to accelerate the conjugate gradient in the second-order method~\cite{li2018highly}. 
However, this is not straightforward in the Riemannian setting since the existence of the projection in the Riemannian Hessian may destroy the sparsity. 
Our current implementation of Algorithm~\ref{alg:semismooth} does not exploit the sparsity of the solution. It is believed that the second-order could be accelerated if we could exploit the sparsity to design a faster CG method.

\begin{table}[t]
	\centering
	\footnotesize
	\caption{Comparison on CM. $(n, r, \mu) = (1000, 20, 0.1)$ and one of them varies. {\sc ManPG} is the adaptive version (ManPG-Ada) in \cite{chen2020proximal}. 
	{\sc LSe, LSq, LSp} denote our algorithm with the exponential map, the retraction using QR decomposition the retraction using polar decomposition, respectively.} %
	\label{tab:cm}
	\begin{subtable}{\linewidth}
	\begin{center}
		
		\begin{tabular}{c|c|cccccccccc} \toprule
			\multicolumn{2}{c|}{} & 
			\sc ManPG & \sc AManPG & \sc ARPG & \sc SOC & \sc LSe-I & \sc LSe-II & \sc LSq-I & \sc LSq-II & \sc LSp-I & \sc LSp-II \\ 
			\multicolumn{2}{c|}{} & \multicolumn{10}{c}{Running Time (s)} \\ \midrule
			 \multirow{5}{*}{$n$}

 & 200 &
11.54 & 3.86 & 6.43 & 1.78 & 1.20 & 1.70 & 1.21 & 1.76 & \textbf{1.10} & 1.54
\\
 & 500 &
21.02 & 8.32 & 9.15 & 5.41 & 4.00 & 6.19 & \textbf{3.96} & 5.88 & 4.15 & 6.14
\\
 & 1000 &
66.30 & \textbf{8.60} & 11.30 & 14.99 & 12.06 & 11.34 & 14.07 & 11.11 & 11.93 & 9.49
\\
 & 1500 &
44.73 & 40.18 & 42.85 & 39.69 & \textbf{24.00} & 27.93 & 24.42 & 27.85 & 25.88 & 33.46
\\
 & 2000 &
42.48 & 46.86 & 51.47 & 46.33 & 33.50 & 28.42 & 31.77 & 29.05 & 27.91 & \textbf{26.22}
\\
\midrule \multirow{4}{*}{$r$}
 & 10 &
15.35 & 15.63 & 15.71 & 17.28 & 31.60 & 12.74 & 70.55 & 15.74 & 80.23 & \textbf{11.26}
\\
 & 15 &
35.35 & \textbf{9.55} & 11.17 & 15.47 & 49.14 & 12.07 & 53.59 & 11.74 & 51.58 & 12.01
\\
 & 25 &
87.64 & 22.73 & 23.93 & 27.11 & \textbf{17.39} & 20.99 & 17.79 & 20.61 & 18.41 & 20.53
\\
 & 30 &
83.13 & 27.41 & 29.35 & 18.10 & 11.49 & 15.49 & 11.40 & 15.13 & \textbf{11.02} & 14.42
\\
\midrule \multirow{4}{*}{$\mu$}
 & 0.05 &
102.63 & 14.86 & 13.98 & \textbf{6.34} & 7.19 & 8.04 & 7.17 & 7.97 & 7.03 & 7.60
\\
 & 0.15 &
65.08 & 17.09 & 21.48 & 32.51 & 14.22 & 16.03 & 14.13 & 15.69 & \textbf{13.98} & 15.59
\\
 & 0.20 &
51.46 & \textbf{12.96} & 24.12 & 33.71 & 17.31 & 20.66 & 19.94 & 19.92 & 21.92 & 22.65
\\
 & 0.25 &
42.74 & \textbf{13.41} & 25.43 & 34.80 & 18.08 & 20.84 & 61.23 & 53.44 & 35.93 & 44.30
\\

\midrule 
		\multicolumn{2}{c|}{} & \multicolumn{10}{c}{Loss Function: $\tr (X^\trans H X) + \mu \| X \|_1$} \\ \midrule
			 \multirow{5}{*}{$n$}

 & 200 &
14.10 & 14.10 & 14.10 & 14.10 & 14.10 & 14.10 & 14.10 & 14.10 & 14.10 & 14.10
\\
 & 500 &
18.60 & 18.60 & 18.60 & 18.60 & 18.60 & 18.60 & 18.60 & 18.60 & 18.60 & 18.60
\\
 & 1000 &
23.30 & 23.30 & 23.30 & 23.30 & 23.30 & 23.30 & 23.30 & 23.30 & 23.30 & 23.30
\\
 & 1500 &
26.90 & 26.80 & 26.80 & 26.80 & 26.80 & 26.80 & 26.80 & 26.80 & 26.80 & 26.80
\\
 & 2000 &
29.80 & 29.70 & 29.70 & 29.70 & 29.70 & 29.70 & 29.70 & 29.70 & 29.70 & 29.70
\\
\midrule \multirow{4}{*}{$r$}
 & 10 &
10.70 & 10.70 & 10.70 & 10.70 & 10.70 & 10.70 & 10.70 & 10.70 & 10.70 & 10.70
\\
 & 15 &
16.50 & 16.40 & 16.40 & 16.40 & 16.40 & 16.40 & 16.40 & 16.40 & 16.40 & 16.40
\\
 & 25 &
32.00 & 32.00 & 32.00 & 32.00 & 32.00 & 32.00 & 32.00 & 32.00 & 32.00 & 32.00
\\
 & 30 &
42.90 & 42.90 & 42.90 & 42.90 & 42.90 & 42.90 & 42.90 & 42.90 & 42.90 & 42.90
\\
\midrule \multirow{4}{*}{$\mu$}
 & 0.05 &
15.10 & 15.10 & 15.10 & 15.10 & 15.10 & 15.10 & 15.10 & 15.10 & 15.10 & 15.10
\\
 & 0.15 &
31.00 & 31.00 & 31.00 & 31.00 & 31.00 & 31.00 & 31.00 & 31.00 & 31.00 & 31.00
\\
 & 0.20 &
38.30 & 38.20 & 38.20 & 38.20 & 38.20 & 38.20 & 38.20 & 38.20 & 38.20 & 38.20
\\
 & 0.25 &
45.30 & 45.20 & 45.20 & 45.30 & 45.20 & 45.20 & 45.20 & 45.20 & 45.20 & 45.20
\\

			\bottomrule
        \end{tabular}%

	\end{center}
	\end{subtable}
\end{table}

\begin{table}[hbpt]
	\centering
	\footnotesize
	\caption{The average iteration number of Algorithm~\ref{alg:alm-outer} in the CM problem.
	The number in the bracket is the percentage that the second-order method starts.}
	\label{tab:cm-newton-it}
	\begin{center}
		
		\begin{tabular}{c|c|cccccc} \toprule
			\multicolumn{2}{c|}{} & 
			\sc LSe-I & \sc LSe-II & \sc LSq-I & \sc LSq-II & \sc LSp-I & \sc LSp-II \\  \midrule
			 \multirow{5}{*}{$n$}
 & 200 &
78 (23\%) & 78 (23\%) & 78 (23\%) & 79 (23\%) & 74 (25\%) & 74 (25\%)
\\
 & 500 &
75 (31\%) & 76 (32\%) & 75 (31\%) & 76 (32\%) & 74 (31\%) & 75 (32\%)
\\
 & 1000 &
46 (25\%) & 44 (22\%) & 45 (18\%) & 44 (21\%) & 44 (17\%) & 43 (20\%)
\\
 & 1500 &
57 (37\%) & 58 (38\%) & 57 (37\%) & 58 (38\%) & 57 (36\%) & 58 (38\%)
\\
 & 2000 &
56 (34\%) & 55 (34\%) & 56 (31\%) & 55 (34\%) & 55 (33\%) & 55 (35\%)
\\
\midrule \multirow{4}{*}{$r$}
 & 10 &
103 (38\%) & 75 (16\%) & 93 (16\%) & 75 (16\%) & 98 (13\%) & 80 (17\%)
\\
 & 15 &
67 (33\%) & 48 (7\%) & 70 (4\%) & 48 (6\%) & 71 (4\%) & 50 (7\%)
\\
 & 25 &
61 (36\%) & 61 (36\%) & 61 (36\%) & 61 (37\%) & 61 (34\%) & 61 (35\%)
\\
 & 30 &
56 (27\%) & 59 (30\%) & 56 (27\%) & 59 (31\%) & 56 (28\%) & 58 (30\%)
\\
\midrule \multirow{4}{*}{$\mu$}
 & 0.05 &
67 (34\%) & 67 (34\%) & 67 (34\%) & 67 (34\%) & 67 (33\%) & 68 (34\%)
\\
 & 0.15 &
54 (32\%) & 54 (32\%) & 54 (32\%) & 54 (32\%) & 54 (33\%) & 55 (33\%)
\\
 & 0.20 &
51 (23\%) & 51 (24\%) & 51 (19\%) & 50 (21\%) & 53 (22\%) & 51 (22\%)
\\
 & 0.25 &
49 (9\%) & 48 (8\%) & 59 (8\%) & 49 (5\%) & 54 (7\%) & 49 (5\%)
\\

			\bottomrule
        \end{tabular}%

	\end{center}
\end{table}

\begin{table}[hbpt]
	\centering
	\footnotesize
	\caption{Comparison on CM with higher accuracy. $(n, r, \mu) = (1000, 20, 0.1)$ and one of them varies. The termination threshold is $5 \times 10^{-8}$. 
	The column {\sc Second-Order} is the total number of iterations, CPU time, average CG iteration number of the second-order method, together with the percentage that CG detects the negative curvature direction (denoted by {\sc CG Rst}).
	The column {\sc First-Order} is the total number of iterations and CPU time of the first-order method~\cite{wen2013feasible} started from the same initial point as the second-order method and terminated when it attains the same optimality condition.} 
	\label{tab:cm-high}
	\begin{center}
		
\begin{tabular}{c|c|cc|cc|cc|cc|cc|cc} \toprule
    \multicolumn{2}{c|}{}
    & \multicolumn{8}{c|}{\sc Second-Order}
    & \multicolumn{4}{c}{\sc First-Order}
    \\ %
    \multicolumn{2}{c|}{}
    & \multicolumn{2}{c|}{\sc Iteration}
    & \multicolumn{2}{c|}{\sc Time (s)}
    & \multicolumn{2}{c|}{\sc Avg. CG Iter}
    & \multicolumn{2}{c|}{\sc CG Rst (\%)}
    & \multicolumn{2}{c|}{\sc Iteration}
    & \multicolumn{2}{c}{\sc Time (s)}
    \\
    \multicolumn{2}{c|}{}
    & \sc LS-I & \sc LS-II &
        \sc LS-I & \sc LS-II &
        \sc LS-I & \sc LS-II &
        \sc LS-I & \sc LS-II &
        \sc LS-I & \sc LS-II &
        \sc LS-I & \sc LS-II \\  \midrule
    \multicolumn{2}{c|}{} & \multicolumn{12}{c}{Exponential Map} \\ \midrule
        \multirow{5}{*}{$n$}
 & 200 &
288 & 369
& 3.06 & 5.12
& 66.13 & 82.87
& 0.00 & 0.26
& 28924 & 23605
& 9.22 & 7.52
\\
 & 500 &
298 & 358
& 5.73 & 10.44
& 60.53 & 88.65
& 0.00 & 0.00
& 39901 & 27498
& 18.83 & 13.17
\\
 & 1000 &
244 & 90
& 7.22 & 4.27
& 45.13 & 96.49
& 4.34 & 24.03
& 3504 & 4374
& 2.65 & 3.24
\\
 & 1500 &
251 & 363
& 26.61 & 65.24
& 156.09 & 256.81
& 1.66 & 3.97
& 68778 & 58475
& 70.11 & 60.56
\\
 & 2000 &
203 & 202
& 21.52 & 33.61
& 157.13 & 249.48
& 4.72 & 0.18
& 37675 & 21100
& 53.22 & 27.54
\\
\midrule \multirow{4}{*}{$r$}
 & 10 &
267 & 373
& 24.93 & 76.76
& 290.48 & 568.20
& 21.91 & 18.23
& 3955 & 23636
& 1.94 & 11.84
\\
 & 15 &
240 & 306
& 30.79 & 81.31
& 313.80 & 592.89
& 25.47 & 19.91
& 6277 & 11495
& 4.20 & 9.39
\\
 & 25 &
237 & 334
& 14.51 & 28.99
& 104.94 & 141.85
& 0.00 & 0.00
& 29139 & 25811
& 27.10 & 23.94
\\
 & 30 &
239 & 441
& 24.02 & 83.16
& 139.21 & 254.22
& 0.00 & 0.79
& 33426 & 20143
& 36.68 & 22.24
\\
\midrule \multirow{4}{*}{$\mu$}
 & 0.05 &
174 & 292
& 5.61 & 13.44
& 73.71 & 93.62
& 0.00 & 0.00
& 9940 & 7906
& 7.40 & 5.91
\\
 & 0.15 &
182 & 215
& 7.28 & 13.31
& 83.12 & 124.71
& 0.00 & 0.00
& 27398 & 19627
& 22.43 & 16.62
\\
 & 0.20 &
121 & 178
& 6.11 & 27.31
& 104.88 & 297.51
& 1.13 & 4.47
& 13972 & 16872
& 10.42 & 13.02
\\
 & 0.25 &
133 & 378
& 10.96 & 142.51
& 169.43 & 652.74
& 3.53 & 10.39
& 15413 & 23251
& 12.23 & 19.91
\\

\midrule 
    \multicolumn{2}{c|}{} & \multicolumn{12}{c}{Retraction using QR Decomposition} \\ \midrule
        \multirow{5}{*}{$n$}
 & 200 &
264 & 377
& 2.39 & 5.12
& 60.13 & 83.44
& 0.00 & 0.00
& 25699 & 25639
& 8.15 & 8.16
\\
 & 500 &
265 & 373
& 4.32 & 9.96
& 53.93 & 84.43
& 0.00 & 0.00
& 30973 & 29728
& 14.61 & 14.04
\\
 & 1000 &
110 & 86
& 5.12 & 3.63
& 89.69 & 85.90
& 20.93 & 18.24
& 2478 & 4965
& 1.85 & 3.69
\\
 & 1500 &
238 & 360
& 22.33 & 65.66
& 142.90 & 263.24
& 2.35 & 5.87
& 69516 & 59054
& 71.92 & 61.12
\\
 & 2000 &
136 & 197
& 14.89 & 32.68
& 182.04 & 249.38
& 5.51 & 0.31
& 23681 & 21536
& 29.78 & 28.15
\\
\midrule \multirow{4}{*}{$r$}
 & 10 &
603 & 335
& 87.72 & 65.44
& 415.50 & 548.50
& 34.42 & 14.05
& 44448 & 26438
& 22.12 & 13.01
\\
 & 15 &
393 & 331
& 50.38 & 70.59
& 286.62 & 481.59
& 44.88 & 23.94
& 24163 & 18022
& 16.03 & 11.89
\\
 & 25 &
225 & 327
& 12.83 & 26.98
& 103.27 & 137.76
& 0.55 & 0.00
& 29205 & 24441
& 26.70 & 22.44
\\
 & 30 &
177 & 439
& 16.66 & 77.70
& 145.66 & 238.25
& 0.00 & 0.00
& 27444 & 20960
& 29.86 & 23.22
\\
\midrule \multirow{4}{*}{$\mu$}
 & 0.05 &
165 & 278
& 5.25 & 13.06
& 74.65 & 97.48
& 0.00 & 0.00
& 9431 & 7887
& 7.00 & 5.85
\\
 & 0.15 &
136 & 217
& 5.64 & 12.71
& 98.06 & 120.97
& 0.00 & 0.00
& 21286 & 20372
& 17.11 & 17.39
\\
 & 0.20 &
99 & 161
& 6.75 & 16.89
& 151.63 & 211.26
& 8.47 & 0.33
& 18256 & 13665
& 13.92 & 10.27
\\
 & 0.25 &
210 & 354
& 33.98 & 89.84
& 313.85 & 450.47
& 20.33 & 0.36
& 79478 & 37584
& 61.58 & 29.45
\\

\midrule 
    \multicolumn{2}{c|}{} & \multicolumn{12}{c}{Retraction using Polar Decomposition} \\ \midrule

        \multirow{5}{*}{$n$}
 & 200 &
274 & 369
& 2.62 & 5.07
& 61.67 & 83.38
& 0.00 & 0.00
& 28362 & 25357
& 9.02 & 8.06
\\
 & 500 &
295 & 335
& 5.19 & 9.91
& 55.56 & 91.45
& 0.00 & 0.00
& 39868 & 27719
& 19.18 & 13.40
\\
 & 1000 &
113 & 71
& 4.84 & 3.78
& 84.84 & 106.35
& 18.23 & 18.42
& 2783 & 5549
& 2.09 & 4.18
\\
 & 1500 &
251 & 390
& 23.54 & 74.01
& 142.35 & 267.34
& 2.79 & 4.64
& 73268 & 62999
& 75.87 & 65.86
\\
 & 2000 &
137 & 201
& 14.82 & 32.11
& 184.55 & 240.82
& 5.04 & 0.18
& 23561 & 21548
& 29.89 & 29.70
\\
\midrule \multirow{4}{*}{$r$}
 & 10 &
491 & 264
& 69.46 & 50.06
& 408.00 & 536.71
& 32.46 & 14.51
& 45323 & 34803
& 22.18 & 17.09
\\
 & 15 &
394 & 274
& 50.06 & 56.21
& 284.55 & 462.63
& 44.71 & 23.48
& 25270 & 18112
& 16.94 & 12.11
\\
 & 25 &
217 & 328
& 12.65 & 28.22
& 106.78 & 143.70
& 0.55 & 0.00
& 28576 & 26171
& 25.92 & 23.86
\\
 & 30 &
160 & 432
& 15.85 & 76.47
& 156.98 & 239.43
& 0.00 & 0.00
& 24631 & 20308
& 26.58 & 22.41
\\
\midrule \multirow{4}{*}{$\mu$}
 & 0.05 &
160 & 281
& 5.17 & 13.37
& 77.30 & 98.57
& 0.00 & 0.00
& 9345 & 8016
& 6.88 & 5.97
\\
 & 0.15 &
136 & 214
& 5.73 & 12.49
& 96.72 & 122.45
& 0.00 & 0.00
& 21143 & 18790
& 17.08 & 14.89
\\
 & 0.20 &
93 & 151
& 5.74 & 14.92
& 143.59 & 199.05
& 7.25 & 0.28
& 16898 & 13309
& 12.59 & 10.08
\\
 & 0.25 &
194 & 281
& 26.74 & 72.03
& 269.42 & 465.45
& 20.43 & 0.41
& 75149 & 36278
& 57.65 & 28.19
\\

    \bottomrule
\end{tabular}%

	\end{center}
\end{table}
\begin{table}[t]
	\centering
	\footnotesize
	\caption{The minimum eigenvalue of the Hessian matrix of \eqref{eqn:Lk} in the CM problem~\eqref{eqn:cm-problem}. 
	$(n, r, \mu) = (1000, 20, 0.1)$ and one of them varies. 
	We report the results of $5$ different runs. 
	Let $\{ x_k \}$ be the sequence generated by Algorithm~\ref{alg:alm-outer} and $I$ be the index set at which Algorithm~\ref{alg:semismooth} is called, then each number of this table is $\min_{k \in I} \lambda_{\min} (H_k)$, where $\lambda_{\min}$ is the minimum eigenvalue and $H_k \in \partial \grad L_k(x_k)$ is from Algorithm~\ref{alg:compute-clarke}.  }
	\label{tab:cm-eigen}
	\begin{subtable}{\linewidth}
	\begin{center}
		
		\begin{tabular}{c|c|ccccc} \toprule
			\multicolumn{2}{c|}{} & 
			\multicolumn{5}{c}{\sc Minimum Eigenvalue ($\times 10$)} 
 \\ \midrule
			 \multirow{5}{*}{$n$}
			 & 200 &
3.04960 & 3.05820 & 3.05160 & 3.23330 & 1.75420
\\
 & 500 &
0.97663 & 1.08610 & 1.14290 & 1.17090 & 1.01000
\\
 & 1000 &
0.13422 & 2.06550 & 0.22452 & 1.88420 & 0.05863
\\
 & 1500 &
0.65285 & 0.61208 & 0.63683 & 0.86749 & 0.80422
\\
 & 2000 &
0.24362 & 0.24363 & 0.24154 & 0.24101 & 0.23518
\\
\midrule \multirow{4}{*}{$r$}
 & 10 &
0.00264 & 0.00570 & 0.00287 & 0.00901 & 0.00005
\\
 & 15 &
0.06096 & 0.05403 & 0.01370 & 0.00334 & 0.06086
\\
 & 25 &
0.60886 & 0.60154 & 0.36930 & 0.76388 & 0.06192
\\
 & 30 &
3.01550 & 6.18060 & 2.96390 & 6.18660 & 5.35900
\\
\midrule \multirow{4}{*}{$\mu$}
 & 0.05 &
0.58725 & 1.03500 & 1.01720 & 1.12870 & 1.15590
\\
 & 0.15 &
0.57333 & 0.73569 & 0.39304 & 0.57299 & 0.53624
\\
 & 0.20 &
0.34736 & 0.20896 & 0.45627 & 0.34404 & 0.03558
\\
 & 0.25 &
0.01539 & 0.16556 & 0.05790 & 0.01262 & 0.05763
\\

			\bottomrule
        \end{tabular}%

	\end{center}
	\end{subtable}
\end{table}

\subsection{Sparse PCA} \label{sec:spca}
In this section, we consider the SPCA problem \eqref{eqn:spca-problem}.
We compare our algorithm with AManPG, ARPG and SOC in high accuracy. The termination conditions are similar to those in CM and the threshold is set to $5 \times 10^{-8}$.
In our algorithm, we set $\tau = 0.99$, the maximum numbers of iterations in CG and the first-order method are both $300$. 
We set $\varepsilon_k = 0.9^k$ and the initial value of $\sigma$ is $3\lambda_{\mathrm{max}}(A^\trans A)$, where $\lambda_{\mathrm{max}}$ denotes the maximal eigenvalue. 
Other parameters of our algorithm are the same as those in CM.
Note that since ManPG generally cannot achieve our requirement on the accuracy, we omit it in this experiment.

\begin{table}[h]
	\centering
	\footnotesize
	\caption{Comparison on SPCA. $A\in\R^{50\times n}$, $(n, r, \mu) = (2000, 20, 1.0)$ and one of them varies. 
	{\sc LSe, LSq, LSp} denote our algorithm with the exponential map, the retraction using QR decomposition the retraction using polar decomposition, respectively.} 
	\label{tab:spca}
	\begin{center}
    \resizebox{\columnwidth}{!}{%
	
		\begin{tabular}{c|c|ccccccccc} \toprule
			\multicolumn{2}{c|}{} & 
			\sc AManPG & \sc ARPG & \sc SOC & \sc LSe-I & \sc LSe-II & \sc LSq-I & \sc LSq-II & \sc LSp-I & \sc LSp-II \\ 
			\multicolumn{2}{c|}{} & \multicolumn{8}{c}{Running Time (s)} \\ \midrule
			 \multirow{6}{*}{$n$}
 & 500 &
6.77 & 48.25 & 18.85 & 7.69 & 7.60 & 5.12 & 5.19 & 5.14 & \textbf{5.03}
\\
 & 1000 &
12.05 & 79.78 & 34.94 & 14.66 & 18.24 & 13.56 & \textbf{11.93} & 13.19 & 12.06
\\
 & 1500 &
34.72 & 87.92 & 57.74 & 29.49 & 28.77 & 24.94 & 25.12 & 27.76 & \textbf{23.55}
\\
 & 2000 &
34.64 & 112.64 & 94.48 & 33.30 & 42.67 & 47.19 & 34.15 & 45.30 & \textbf{30.55}
\\
 & 2500 &
57.50 & 116.51 & 124.66 & 45.11 & 42.63 & 46.11 & \textbf{37.25} & 52.52 & 37.63
\\
 & 3000 &
81.80 & 134.70 & 158.95 & 77.56 & 81.44 & 99.62 & \textbf{58.87} & 79.36 & 64.45
\\
\midrule \multirow{4}{*}{$r$}
 & 5 &
\textbf{7.37} & 25.00 & 55.96 & 20.86 & 19.32 & 17.73 & 17.86 & 18.37 & 16.04
\\
 & 10 &
18.52 & 44.10 & 75.23 & 22.49 & 32.83 & 25.89 & \textbf{18.12} & 29.05 & 19.12
\\
 & 15 &
38.14 & 67.22 & 86.39 & 35.84 & 31.50 & 38.00 & \textbf{18.94} & 30.52 & 19.12
\\
 & 25 &
62.88 & 146.31 & 105.45 & 39.49 & 35.56 & 27.73 & \textbf{27.45} & 33.83 & 28.75
\\
\midrule \multirow{4}{*}{$\mu$}
 & 0.25 &
\textbf{70.97} & 80.56 & 93.19 & 116.40 & 141.50 & 89.27 & 96.15 & 87.74 & 97.62
\\
 & 0.50 &
57.01 & 86.69 & 93.66 & 53.81 & 50.27 & 58.37 & 41.38 & 57.38 & \textbf{38.75}
\\
 & 0.75 &
51.82 & 98.02 & 94.44 & 37.36 & 39.85 & 43.89 & 32.15 & 39.06 & \textbf{30.23}
\\
 & 1.25 &
37.36 & 112.67 & 93.86 & 40.34 & 30.16 & 53.26 & \textbf{23.25} & 52.10 & 24.18
\\

 \midrule 
		\multicolumn{2}{c|}{} & \multicolumn{9}{c}{Loss Function: $-\tr (X^\trans A^\trans A X) + \mu \| X \|_1$} \\ \midrule
			 \multirow{6}{*}{$n$}
 & 500 &
-337.79 & -337.73 & -337.32 & \textbf{-337.82} & \textbf{-337.82} & \textbf{-337.82} & \textbf{-337.82} & \textbf{-337.82} & \textbf{-337.82}
\\
 & 1000 &
-749.79 & \textbf{-749.92} & -749.62 & -749.61 & -749.61 & -749.61 & -749.61 & -749.61 & -749.61
\\
 & 1500 &
-1207.65 & \textbf{-1207.83} & -1206.96 & -1207.57 & -1207.48 & -1207.57 & -1207.48 & -1207.57 & -1207.48
\\
 & 2000 &
-1621.52 & -1621.85 & -1621.35 & \textbf{-1622.69} & \textbf{-1622.69} & \textbf{-1622.69} & \textbf{-1622.69} & \textbf{-1622.69} & \textbf{-1622.69}
\\
 & 2500 &
-2077.05 & \textbf{-2077.30} & -2076.24 & -2077.13 & -2077.13 & -2077.13 & -2077.13 & -2077.13 & -2077.13
\\
 & 3000 &
-2542.05 & -2541.80 & -2541.15 & \textbf{-2542.36} & \textbf{-2542.36} & \textbf{-2542.36} & \textbf{-2542.36} & \textbf{-2542.36} & \textbf{-2542.36}
\\
\midrule \multirow{4}{*}{$r$}
 & 5 &
-1497.56 & -1497.50 & -1497.52 & -1497.66 & \textbf{-1497.67} & \textbf{-1497.67} & \textbf{-1497.67} & \textbf{-1497.67} & \textbf{-1497.67}
\\
 & 10 &
-1640.15 & -1640.17 & -1639.74 & \textbf{-1640.63} & \textbf{-1640.63} & \textbf{-1640.63} & \textbf{-1640.63} & \textbf{-1640.63} & \textbf{-1640.63}
\\
 & 15 &
-1641.56 & \textbf{-1641.60} & -1640.95 & -1640.68 & -1640.79 & -1640.79 & -1640.79 & -1640.79 & -1640.79
\\
 & 25 &
-1636.00 & -1636.46 & \textbf{-1636.54} & -1636.48 & -1636.48 & -1636.48 & -1636.48 & -1636.48 & -1636.48
\\
\midrule \multirow{4}{*}{$\mu$}
 & 0.25 &
-1874.13 & -1874.15 & -1871.42 & \textbf{-1874.23} & \textbf{-1874.23} & -1874.22 & -1874.22 & -1874.22 & -1874.22
\\
 & 0.50 &
\textbf{-1780.83} & -1780.81 & -1778.62 & -1780.71 & -1780.69 & -1780.69 & -1780.69 & -1780.69 & -1780.69
\\
 & 0.75 &
-1698.08 & -1697.79 & -1695.84 & \textbf{-1698.17} & \textbf{-1698.17} & \textbf{-1698.17} & \textbf{-1698.17} & \textbf{-1698.17} & \textbf{-1698.17}
\\
 & 1.25 &
-1552.05 & -1552.28 & -1552.04 & \textbf{-1553.42} & -1553.41 & -1553.41 & -1553.41 & -1553.41 & -1553.41
\\

			\bottomrule
        \end{tabular}%

	}
	\end{center}
\end{table}

The data matrix $A \in \R^{50\times n}$ is generated as follows: First, we randomly generate $A$ from the standard Gaussian distribution. 
Then, we modify the singular values  of $A$ to $\{ w_i^4 + 10^{-5} \}_{i=1}^{50}$ to make $A$ ill-conditioned, where $\{ w_i \}$ is sampled from the standard Gaussian distribution. Finally, columns of $A$ are normalized to have zero mean and unit length. 
Results are reported in Table~\ref{tab:spca} and Figure~\ref{fig:spca-comparison}. 
We find that the losses of the solutions obtained by these methods are comparable, and our algorithm (LS-II) is faster than other methods when one of $\mu, r, n$ is large.

\subsection{Constrained Sparse PCA}
\begin{table}[t]
	\centering
	\footnotesize
	\caption{Comparison on constrained SPCA. $(n, r, \mu) = (2000, 20, 1.0)$ and one of them varies. {\sc Equality} and {\sc Inequality} are the logarithm of the violation of equality/manifold and inequality constraints, respectively.}
	\label{tab:constrained-spca}
	\begin{center}
    \resizebox{\columnwidth}{!}{%
   		\begin{tabular}{c|c|cc|cc|cc|cc|cc} \toprule
			\multicolumn{2}{c|}{}
			& \multicolumn{2}{c|}{\sc CPU (s)} 
			& \multicolumn{2}{c|}{\sc Sparsity (\%)} 
			& \multicolumn{2}{c|}{\sc CPAV (\%)} 
			& \multicolumn{2}{c|}{\sc Equality} 
			& \multicolumn{2}{c}{\sc Inequality} 
			\\
			\multicolumn{2}{c|}{}
			 & \sc Ours & \sc ALSPCA & 
			 \sc Ours & \sc ALSPCA & 
			 \sc Ours & \sc ALSPCA & 
			 \sc Ours & \sc ALSPCA & 
			 \sc Ours & \sc ALSPCA  \\
			 \midrule
			 \multirow{5}{*}{$r$}
& 5 &
\textbf{9.01} & 12.17 &
\textbf{37.83} & 34.43 &
11.82 & \textbf{11.93} &
\textbf{-15.31} & -9.19 &
-8.21 & \textbf{-8.46}\\
& 10 &
\textbf{14.67} & 20.80 &
\textbf{38.16} & 33.74 &
22.77 & \textbf{23.19} &
\textbf{-15.32} & -8.74 &
\textbf{-8.29} & -8.19\\
& 15 &
\textbf{21.10} & 31.29 &
\textbf{38.85} & 38.62 &
\textbf{33.03} & 32.97 &
\textbf{-15.27} & -8.73 &
\textbf{-8.33} & -8.19\\
& 25 &
\textbf{32.51} & 33.67 &
\textbf{39.23} & 38.68 &
51.59 & \textbf{52.09} &
\textbf{-15.14} & -9.06 &
\textbf{-8.36} & -7.98\\
& 30 &
\textbf{39.79} & 39.84 &
39.65 & \textbf{40.51} &
60.09 & \textbf{60.31} &
\textbf{-15.12} & -9.03 &
\textbf{-8.23} & -7.91\\
\midrule
			\multirow{5}{*}{$n$}
& 500 &
\textbf{10.14} & 17.13 &
\textbf{72.62} & 68.67 &
\textbf{35.71} & 35.61 &
\textbf{-15.23} & -8.93 &
\textbf{-8.32} & -8.26\\
& 1000 &
\textbf{15.64} & 22.97 &
\textbf{56.07} & 49.80 &
40.57 & \textbf{41.78} &
\textbf{-15.20} & -8.99 &
\textbf{-8.23} & -7.96\\
& 2000 &
\textbf{26.23} & 28.91 &
\textbf{38.74} & 38.01 &
42.65 & \textbf{42.94} &
\textbf{-15.12} & -8.76 &
-8.11 & \textbf{-8.22}\\
& 4000 &
52.48 & \textbf{44.21} &
\textbf{23.30} & 21.65 &
43.22 & \textbf{44.06} &
\textbf{-15.08} & -9.13 &
\textbf{-8.62} & -7.89\\
& 6000 &
71.72 & \textbf{60.98} &
\textbf{17.68} & 16.99 &
43.13 & \textbf{43.74} &
\textbf{-15.22} & -8.89 &
\textbf{-8.79} & -7.94\\

			\bottomrule
		\end{tabular}
	}
	\end{center}
\end{table}

\begin{figure}[t]
	\centering
	\begin{subfigure}{0.328\linewidth}
		\centering
		\includegraphics[width=\linewidth]{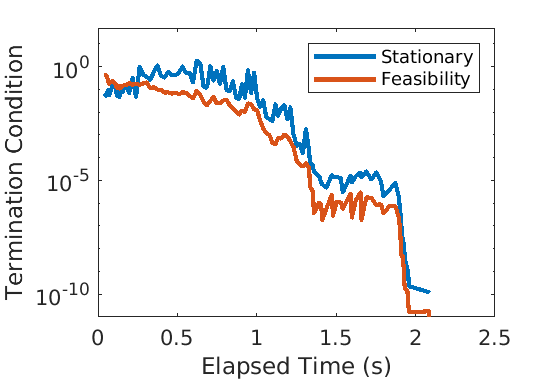}
	\end{subfigure} \hfill
	\begin{subfigure}{0.328\linewidth}
		\centering
		\includegraphics[width=\linewidth]{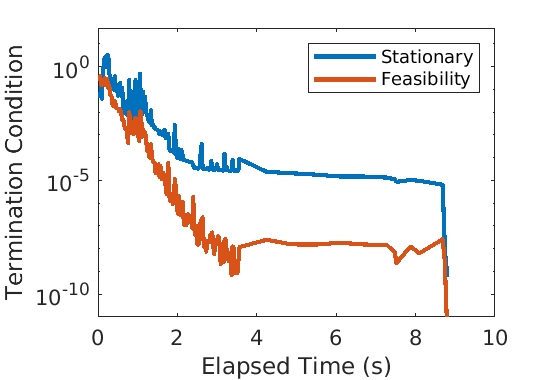}
	\end{subfigure} \hfill
	\begin{subfigure}{0.328\linewidth}
		\centering
		\includegraphics[width=\linewidth]{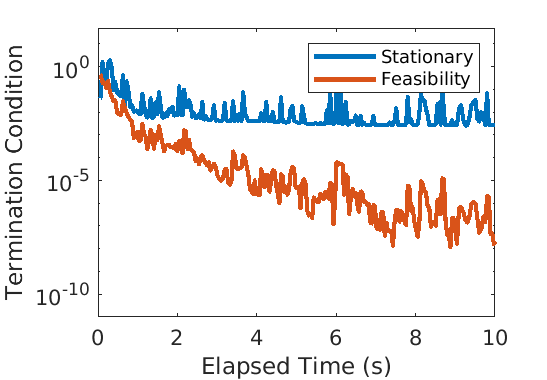}
	\end{subfigure}
	\caption{Results of Algorithm~\ref{alg:alm-outer} and \ref{alg:semismooth} on constrained SPCA with $(n, \mu) = (600, 1.0)$ and $r = 3, 5, 10$, respectively. We use the first-order method to find an initial point for Algorithm~\ref{alg:semismooth} and start the second-order method when $\| \grad L_k \| \leq 5 \times 10^{-4}$. ``Feasibility'' refers to the condition like \eqref{eqn:term-ours-feasibility} and ``Stationary'' refers to the condition like \eqref{eqn:term-ours-stationary}.
	The second-order method is able to start when $r = 3, 5$ and fails to start when $r = 10$.}
	\label{fig:ortho-spca}
\end{figure}

Since the SPCA problem in Sec.~\ref{sec:spca} has no guarantee on finding eigenvectors, 
the constrained SPCA problem \eqref{eqn:cons-spca-problem} is considered in~\cite{lu2012spca}.
In this section, we compare our algorithm with the method in~\cite{lu2012spca}, which will be referred to as ALSPCA. In our implementation, we apply the first-order feasible method on Stiefel manifolds~\cite{wen2013feasible} to solve the subproblem when the iterates do not meet the starting conditions of the semismooth Newton method. Following~\cite{lu2012spca}, the termination condition of the subproblem is based on the feasibility conditions~\eqref{eqn:term-ours-feasibility}. 
In our algorithm, we set $\tau = 0.25$, $\rho = 10$, $\alpha = 1.01$, the maximum number of iterations of the first-order method is $2000$. We set $\varepsilon_k = 0.1^k$ and the initial value of $\sigma$ is $1$. The parameters of ALSPCA are the same as those in \cite{lu2012spca}. The data matrix $A \in \R^{50 \times n}$ is generated from the standard Gaussian distribution and each column of $A$ is normalized to have zero mean and unit length.

In our experiments, we observe that our algorithm and ALSPCA find very different solutions for the same penalty parameter $\mu$. To better compare their performance, we set $\mu = 1$ for our algorithm, and tune $\mu$ to achieve a comparable sparsity for ALSPCA. We compare the quality of solutions in terms of the cumulative percentage of adjusted variance (CPAV) proposed by \cite{lu2012spca}. It is defined as follows:
\[ \text{CPAV}(V) = \frac{1}{\tr(A^\trans A)} \Big ( \tr(V^\trans A^\trans AV) - \sqrt{\sum_{i \neq j} (V_i^\trans A^\trans AV_j)^2 } \Big ). \]
This quantity measures how well the data is explained by the found principal components.
We set $\Delta_{ij} = 10^{-8}$ and report the results in Table~\ref{tab:constrained-spca}. We find solutions obtained from these algorithms are comparable in terms of CPAV, and we also find there is a trade-off between the sparsity and the explainability (CPAV). The speed of our algorithm is comparable with ALSPCA for large $n, r$, and is faster for small $n, r$. Moreover, as shown in the first two columns in Figure~\ref{fig:ortho-spca}, the semismooth Newton method can start when $r=3$ and $r=5$. When $r=10$, the behavior of our algorithm is similar to the first order methods as it is difficult to fulfill the stationary conditions like \eqref{eqn:term-ours-stationary}. One possible reason is that \eqref{eqn:cons-spca-problem} has $r(r-1)$ inequality constraints such that the non-degeneracy condition required by Theorem~\ref{thm:semismooth-LS-I-rate} is likely violated for the case when $r$ is large. 
A similar phenomenon has also been observed in SDP~\cite{zhao2010newton,yang2015sdpnal}. Even in this case, we find that our algorithm is still faster than the ALSPCA method.

\section{Conclusion}
\label{sec:conclusions}
The paper proposes an augmented Lagrangian method for solving a class of nonsmooth optimization problems on manifolds with proved convergence. Using the Moreau-Yosida regularization, the augmented Lagrangian subproblem can be efficiently solved by the globalized semismooth Newton method, which exploits the second order geometry structure of manifolds. The local superlinear convergence rate is established under certain non-degenerate conditions that are similar to the case in the Euclidean space. Our numerical experiments on various applications show the advantages of the proposed method over the existing methods. The work done in  this paper on ALM for solving the nonsmooth optimization problems on manifolds is by no means complete. There are many unanswered questions on both theory and algorithm design. For example, the  convergence results of ALM for solving the nonsmooth manifold optimization problem are established under the constraint qualifications such as CPLD and  LICQ. A systematical study on the  convergence analysis of ALM under weaker assumptions is certainly of paramount necessity for solving the nonsmooth manifold optimization problems. Another direction is to design more efficient algorithms for solving the ALM subproblems in particular for the more challenging problems such as the constrained sparse PCA. It is our firm belief that a better usage of the inherent second order geometry structure of manifolds rather than projecting them into the tangent spaces will lead to more efficient optimization methods for solving nonsmooth manifold optimization problems.

\section*{Acknowledgement}
The authors would like to thank the two anonymous referees and the associate editor for their valuable comments and constructive suggestions,
which have greatly improved the quality and the presentation of this paper.

\begin{appendices}

\section{The Deferred Proofs in Sections~\ref{sec:global-convergence} and \ref{sec:local-convergence}}
\subsection{Proof of Lemma~\ref{lem:retraction-approx}}
\label{app:proof-retraction-approx}
\begin{proof}
	From the compactness of $U$, there exist $C, K,  r_0 > 0$ such that Lemma~\ref{lem:cosine-law} holds for every $q \in U$.
	By shrinking the constant $r_0$ if necessary, we could also assume that for every $q \in U$, $R_q v$ is defined for every $v \in T_q\cM$ with $\|v \| < r_0$.
	For a fixed point $q \in \cM$, let $\{ e_i \}_{i \in [n]}$ be an orthonormal basis of $T_q \cM$. We define the isomorphism $\psi_q: \R^n \to T_q \cM$ as $\psi_q(w_1, \dots, w_n) = \sum_{i=1}^n w_i e_i$,
	and the chart $\varphi_q: B_{r_0}(q) \to \R^n$ as 
	$\varphi_q\left ( \exp_q \psi_q(w)  \right ) = w$, where $\|w\|_{\R^n} < r_0$,\footnote{Indeed, this is called the \emph{normal coordinate} (see, e.g., \cite[p.~86]{carmo1992riemannian} and \cite[p.~131]{lee2018introduction}).}
	which implies that $(\psi_q \circ \varphi_q )(p) = \exp_q^{-1} p$ for every $p \in B_q(r_0)$.

	For $\|v \| < r_0$, we define the curve $\gamma(t) := R_q (tv)$, then 
	\begin{equation}
		\label{eqn:first-order-retraction-approx}
		d(R_q v, q) \leq \ell(\gamma|_{[0, 1]})
		= \int_0^1 \| \dot \gamma(t) \| \dd t
		= \int_0^1 \| \dd R_q|_{tv} [v] \|  \dd t
		\leq \tilde C_q \| v \|,
	\end{equation}
	where $\tilde C_q := \sup_{\|v\| < r_0} \| \dd R_q|_v \|$. Since $R_q v$ is $C^2$ with respect to $q$ and $v$ and $U$ is compact, we know $\tilde C := \sup_{q \in U} \tilde C_q < \infty$.
	Thus, $\exp^{-1}_q R_q v$ is defined for $q \in U$ and $\|v \| < r := r_0 / \tilde C$.

	Let $\hat B_{r} := \{ w \in \R^n : \|w \|_{\R^n} < r \}$, and $\hat R_q, \hat E_q: \hat B_{r} \to \R^n$ be the maps $\varphi_q \circ R_q \circ \psi_q|_{\hat B_r}$, $\varphi_q \circ \exp_q \circ \psi_q|_{\hat B_r}$, respectively.
	Since the differentials of $\exp_q$ and $R_q$ are the same at $0$, then the differentials of $\hat E_q$ and $\hat R_q$ also coincide at $0$.
	Let $\hat D_q := \hat E_q - \hat R_q$, 
	then Taylor's theorem states that for any $w \in \hat B_r$, there exists $t_w \in (0, 1)$ such that
	$\| \hat D_q(w) \|_{\R^n}  = \frac{1}{2} \| \nabla^2 \hat D_q(t_w w)[w, w] \|_{\R^n} \leq \hat C_q \| w \|^2_{\R^n}$, where $\hat  C_q := \frac{1}{2}\sup_{ \| w \|_{\R^n} < r} \| \nabla^2 \hat D_q(w) \|$.

	Since $\psi_q$ is a linear isomorphism, then
	$(\psi_q \circ \hat D_q \circ \psi_q^{-1})(v) 
	= (\psi_q \circ (\varphi_q \circ \exp_q - \varphi_q \circ R_q))(v) 
	= (\psi_q \circ \varphi_q \circ \exp_q)(v) - (\psi_q \circ \varphi_q \circ R_q)(v) 
	= v - \exp^{-1}_q R_q v$.
	Therefore, by Lemma~\ref{lem:cosine-law}, we have 
	\[ d(R_q v, \exp_q v)^2 
	\leq  \| v - \exp^{-1}_q R_q v \|^2 + K  \| \exp^{-1}_q R_q v \|^2 \| v \|^2
	\leq (\hat C_q^2+ \tilde C^2 K ) \| v \|^4
	\] 
	for every $v \in T_q \cM$ with $\| v \| < r$, where the last inequality follows from \eqref{eqn:first-order-retraction-approx} and $\| \exp^{-1}_q R_q v \| = d(R_q v, q)$.
	Finally, since $R_q v$ and $\exp_q v$ are both $C^2$ with respect to $q$ and $v$, $\hat C_q$ is uniformly bounded on the compact set $U$. 
	Thus, the required constant $C$ exists and is finite. \hfill $\Box$
\end{proof}

\subsection{Proof of Lemma~\ref{lem:second-taylor}}
\label{app:proof-second-taylor}
\begin{proof}
	Define $\hat \varphi = \varphi \circ \gamma$. We know that $\hat \varphi$ is continuously differentiable with Lipschitz gradient in~$[0, 1]$.
	By Theorem 2.3 in \cite{hiriart1984generalized}, there exist $\xi \in (0, 1), \hat M_\xi \in \partial^2 \hat\varphi(\xi)$ such that $\hat \varphi(1) - \hat \varphi(0) = \hat\varphi^\prime(0) + \frac{\hat M_\xi}{2}$. Note that due to $\hat\varphi^\prime(t) = \inner{\grad \varphi(\gamma(t))}{\dot \gamma(t)}$ 
	and $\nabla_{\dot \gamma}\dot\gamma = 0$, we find $\hat \varphi^{\prime\prime}(t) = \inner{\nabla_{\dot\gamma(t)} \grad \varphi(\gamma(t))}{\dot\gamma(t)}$,  whenever~$\hat \varphi^\prime$ is differentiable at $t$.
	From the linearity of the parallel transport, it suffices to consider the case $\hat M_\xi \in \partial^2_B \hat \varphi(\xi)$, i.e.,
	there exists $\{ t_k \} \to \xi$ such that $\hat \varphi^{\prime\prime}$ exists at $t_k$ and $\hat \varphi^{\prime\prime}(t_k) \to \hat M_\xi$.
	Since by the isometry property of the parallel transport, $\hat \varphi^{\prime\prime}(t_k) = \langle  \nabla_{ \dot \gamma(t_k)} \grad \varphi(\gamma(t_k)) , \dot \gamma(t_k) \rangle = \langle P_{\gamma}^{t_k \mapsto \xi} \nabla_{P_\gamma^{\xi \mapsto t_k} \dot \gamma(\xi)} \grad \varphi(\gamma(t_k)) , \dot \gamma(\xi) \rangle$, then
	there is $M_\xi \in \partial \grad \varphi(\gamma(\xi))$ such that $\hat M_\xi = \inner{M_\xi \dot \gamma(\xi)}{\dot \gamma(\xi)}$.
	Note $\dot\gamma(\xi) = P_\gamma^{0\to\xi}\dot\gamma(0)$, then the conclusion holds. \hfill $\Box$
\end{proof} 

\subsection{Proof of Lemma~\ref{lem:B-differentiability}}
\label{app:proof-B-differentiability}
\begin{proof}
	Suppose the conclusion does not hold, then there exist $\delta>0$ and $v_k \in T_p\cM \to 0$ such that 
	\[ \normz{P_{\exp_p v_k, p}X(\exp_p v_k) - \nabla X(p; v_k)} \geq \delta \norm{v_k}. \]
	By Lemma~\ref{lem:cosine-law} and the locally Lipschitz condition on $X$, there exists $r_0>0$ such that for every $p_1,p_2\in B_{r_0}(p)$, there exists a unique shortest geodesic joining $p_1$ and $p_2$ and $X$ is $L$-Lipschitz in $B_{r_0}(p)$. 
	Let $0<r<r_0$, by taking the subsequence of $\{v_k\}$, we assume that $\bar v_k := \frac{rv_k}{\norm{v_k}} \to v \in T_p\cM$ with $\norm{v} = r$. 
	We define $t_k = \norm{v_k}/r$, $q_k = \exp_p v_k$, $r_k = \exp_p \bar v_k$ and $s_k = \exp_p t_k v$, then it holds that $\{q_k,r_k,s_k\}\subset B_{r_0}(p)$, and hence we have
	\begin{equation}\label{estimate_412}
	    \begin{aligned}
	    \|P_{q_k,p}X(q_k) - \nabla X(p;v_k)\| &\leq \|P_{q_k,p} X(q_k)-P_{s_k,p} X(s_k)\| + \|P_{s_k,p} X(s_k) - \nabla X(p;v_k)\|\\
	    &\leq \underbrace{\|P_{p,s_k}P_{q_k,p} X(q_k) - P_{q_k,s_k} X(q_k))\|}_{\rm (A)} + \underbrace{\|P_{q_k,s_k} X(q_k) -  X(s_k)\|}_{\rm (B)}\\
	    &\peq + \underbrace{\|P_{s_k,p} X(s_k) - t_k \nabla X(p;v)\|}_{\rm (C)} + \underbrace{\|\nabla X(p;v_k) - t_k\nabla X(p;v)\|}_{\rm (D)}.
	\end{aligned}
	\end{equation}
	 From Lemma 10 in~\cite{huang2015riemannian}, there exists $C>0$ such that
	\begin{equation}\label{estimate:I}
	     {\rm (A)} \leq \|P_{s_k,q_k}P_{p,s_k}P_{q_k,p} - \id\|\|X(q_k)\| \leq C\max(d(p,q_k), d(q_k,s_k)) \|X(q_k)\| = o(t_k), 
	\end{equation}
	since $\|X(q_k)\|\leq L \|v_k\|\rightarrow 0$ as $v_k\to 0$. Moreover, since $X$ is locally Lipschitz and directional differentiable at $p$, using Lemma~\ref{lem:cosine-law}, as $k\to\infty$, we have
	\begin{equation}\label{estimate:II_III}
	   {\rm (B)} \leq L d(q_k,s_k) = o(t_k), \quad {\rm (C)}  = \|P_{s_k,p} X(s_k) - X(p) - t_k\nabla X(p;v)\| = o(t_k),
	\end{equation}
	where the first estimate follows from Definition~\ref{def:lipschitz} and $\ell(\gamma) = d(q_k, s_k)$.
	Let $q_k^t = \exp_p t\bar v_k$ and $q_v^t = \exp_p tv$, from Lemma 10 in~\cite{huang2015riemannian} and Lemma~\ref{lem:cosine-law}, we have when $t \to 0$,
	\begin{equation*}
	\begin{aligned}
	    \|P_{q_k^t,p}X(q_k^t)-P_{q_v^t,p}X(q_v^t)\|
	    & \leq  \|P_{p,q_v^t}P_{q_k^t,p}X(q_k^t)- P_{q_k^t,q_v^t}X(q_k^t)\| + \|P_{q_k^t,q_v^t}X(q_k^t)-X(q_v^t)\| \\
	    & \leq  \|P_{q_v^t,q_k^t}P_{p,q_v^t}P_{q_k^t,p}-\id\|\|X(q_k^t)\| + L d(q_k^t,q_v^t) = O(t^2) + O(t)\|\bar v_k - v\|,
	\end{aligned}
	\end{equation*}
	where the constants in the big-O notation depend on $p$, $v$ and $\cM$.
	An intermediate result of the above estimate is 
	\begin{equation}
		\label{eqn:est-direction}
	    \|\nabla X(p;\bar v_k) - \nabla X(p;v)\|  \leq O(\|\bar v_k - v\|) = o(1).
	\end{equation}
	Thus, the last term in~\eqref{estimate_412} is 
	\begin{equation}\label{estimation:IV}
	    {\rm (D)} = t_k\|\nabla X(p;\bar v_k) - \nabla X(p;v)\| = o(t_k)
	\end{equation}
	Combining \eqref{estimate:I}, \eqref{estimate:II_III} and \eqref{estimation:IV}, we find that 
	    $\|P_{q_k,p}X(q_k) - \nabla X(p;v_k)\| = o(t_k) = o(\|v_k\|)$, 
which contradicts with our assumption. \hfill $\Box$
\end{proof}

\subsection{Proof of Lemma~\ref{lem:mani-cvg}}
\label{app:proof-mani-cvg}
\begin{proof}
	Define $q_{k + 1} := R_{p_k} V_k$. %
	Let $B_r(p), K$ be the quantities in Lemma~\ref{lem:cosine-law}. 
	Without loss of generality, we may assume $p_k \in B_r(p)$  and $\| v\| < r$ for every $k$, and shrink the ball $B_r(p)$ such that \eqref{eqn:retraction-approx} holds in $B_r(p)$.
	Using Lemma~\ref{lem:cosine-law}, we know the following inequality holds for every geodesic triangle in $B_r(p)$ whose edges are $a, b, c$:
	\[ a^2
	\leq b^2 + c^2 - 2bc \cos A + Kb^2c^2 \leq (b + c)^2 + Kb^2c^2,
	 \]
    where $A$ is the angle between $b$ and $c$. 
    Let $a = d(p_k, q_{k+1})$, $b = d(p_k, p)$, $c = d(q_{k+1}, p)$, then we have
    \begin{equation}
        \label{eqn:ratio-limsup_1} \limsup_{k \to \infty} \frac{d(p_k,q_{k+1})^2}{d(p_k,p)^2} 
    \leq K \limsup_{k \to \infty} d(q_{k+1}, p)^2 + \limsup_{k \to \infty} \left ( 1 + \frac{d(q_{k+1},p)}{d(p_k,p)} \right )^2 
    = 1.
    \end{equation}
    
	Define $a = d(p_{k}, p)$, $b = d(p_k, q_{k+1})$, $c = d(q_{k + 1}, p)$, note that $d(q_{k+1},p) = o(d(p_k, p))$, then for any $\varepsilon > 0$ there exists $k_\varepsilon > 0$ such that for any $k \geq k_\varepsilon$, it holds that 
	$d(q_{k + 1}, p)^2 
        \leq \varepsilon d(p_k, p)^2
        \leq 2\varepsilon d(q_{k + 1}, p)^2  + 2\varepsilon d(p_{k}, q_{k+1})^2 + \varepsilon K d(q_{k+1}, p)^2 d(p_k, q_{k+1})^2$.
	From \eqref{eqn:ratio-limsup_1} and the fact that $p_k \to p$ as $k \to \infty$, there exists $K_0 > 0$ such that $K d(p_k, q_{k+1})^2 \leq 2K d(p_k, p)^2 \leq 2$ for $k \geq K_0$.
	Thus, for all $\varepsilon \in (0, 1/8)$ and $k \geq \max\{ k_\varepsilon, K_0 \}$, it holds
    \[ \begin{aligned}
        d(q_{k + 1}, p)^2 
        \leq \frac{2\varepsilon}{1 - 2\varepsilon - \varepsilon K d(p_k, q_{k+1})^2} d(p_k, q_{k+1})^2 \leq 4 \varepsilon d(p_k, q_{k+1})^2, 
    \end{aligned} \]
	i.e., $d(q_{k+1}, p)^2 = o( d(p_k, q_{k+1})^2)$.
	From \eqref{eqn:retraction-approx}, we have 
	\[d(p_k, q_{k+1}) \leq d(p_k, \exp_{p_k} V_k) + d(\exp_{p_k} V_k, q_{k+1}) \leq \| V_k \| + C \| V_k \|^2,\] 
	and hence it holds $d(q_{k+1},p) = o( d(p_k, q_{k+1})) = o( \| V_k \|)$.

	On the other hand, 
    \[ \limsup_{k \to \infty} \frac{d(p_k,p)^2}{d(p_k,q_{k+1})^2} 
    \leq K \limsup_{k \to \infty} d(q_{k+1}, p)^2 + \limsup_{k \to \infty} \left ( 1 + \frac{d(q_{k+1},p)}{d(p_k,q_{k+1})} \right )^2 
    = 1.
    \]
    Combining with \eqref{eqn:ratio-limsup_1}, we have $\lim_{k \to \infty}\limits \frac{d(p_k, p)}{d(p_k, q_{k+1})} = 1$. \hfill $\Box$
\end{proof}

\section{The Deferred Proofs in Section~\ref{sec:semismoothness-from-euc}}

\subsection{Proof of Theorem~\ref{thm:semismoothness-from-euc}}
\label{app:proof-semismoothness-from-euc}
We prove Theorem~\ref{thm:semismoothness-from-euc} in this section.
First, we give a simple result that will be frequently used in our proof:
For $p \in \bar \cM$, $\xi \in T_p \bar \cM$, and a curve $\gamma: [0, 1] \to \bar \cM$ with $\gamma(0) = p$, 
the vector field $Z(t) := \bar P_\gamma^{0 \to t} \xi$ is parallel to $\gamma$, and hence by the definition of parallel transport we have
\begin{equation}
	\label{eqn:parallel-transport-gen}
	\bar \nabla_{\dot \gamma} \bar P_\gamma^{0 \to t} \xi = 0.
\end{equation}
Similarly, for $p \in \cM, \zeta \in T_p \cM$ and $\gamma: [0, 1] \to \cM$, it holds that $\nabla_{\dot \gamma} P^{0 \to t}_\gamma \zeta = 0$.

The following lemma compares the parallel transports of the manifold $\cM$ and its ambient space $\bar \cM$.
\begin{lemma}
	\label{lem:compare-transport}
	There exists $C > 0$ such that for every $p \in \cM$, and every smooth curve $\gamma: [-1, 1] \to \cM$ with $\gamma(0) = p$, and every $v \in T_p \cM$, $t \in [0, 1]$, it holds that
	\begin{equation}
	\label{eqn:compare-transport}
		\| P_\gamma^{0 \to t} v - \bar P_{\gamma}^{0 \to t} v \| \leq C \ell(\gamma|_{[0,t]}) \| v \|.
	\end{equation}
	Moreover, there exists $C^\prime > 0$ such that for every $p \in \cM$, $w \in T_p \cM$ with $\| w\| \leq 2$, and every $t \in [0, 1]$, the following inequality holds for $\gamma(t) := \exp_p(tw)$.
	\begin{equation}
	\label{eqn:compare-transport-sec}
		\| P_\gamma^{0 \to t} v - \bar P_{\gamma}^{0 \to t} v - t \II(\dot\gamma(t), P^{0 \to t}_\gamma v) \| \leq C^\prime t^2.
	\end{equation}
\end{lemma}
\begin{proof}
	Let  $Y: [0, 1] \to T\cM$ be a smooth vector field along $\gamma$.
	For any fixed $t \in [0, 1]$ and $\xi \in T_{\gamma(t)} \bar \cM$, it holds that
	\[  \begin{aligned}
		\inner{ \frac{\dd}{\dd s} \bar P^{s \to t}_\gamma Y(s)}{\xi} 
		&= \frac{\dd}{\dd s} \inner{ \bar P^{s \to t}_\gamma Y(s)}{\xi} 
		= \frac{\dd}{\dd s} \inner{ Y(s)}{\bar P^{t \to s}_\gamma \xi}  \\
		&=  \inner{ \bar \nabla_{\dot\gamma} Y(s)}{\bar P^{t \to s}_\gamma \xi} 
		+  \inner{ Y(s)}{\bar \nabla_{\dot\gamma} \bar P^{t \to s}_\gamma \xi} 
		\overset{\eqref{eqn:parallel-transport-gen}}{=}  \inner{\bar P^{s \to t}_\gamma  \bar \nabla_{\dot\gamma} Y(s)}{\xi},
	\end{aligned}
	\]
	Therefore, $\frac{\dd}{\dd s} \bar P_\gamma^{s \to t}Y(s) = \bar P_{\gamma}^{s \to t} \bar \nabla_{\dot \gamma} Y(s)$.

	Setting $Y(t) = P_\gamma^{0 \to t} v$ for $v \in T_p \cM$, we know
	\[ \bar \nabla_{\dot\gamma} Y \overset{\eqref{eqn:gauss-formula}}{=} \nabla_{\dot\gamma} Y + \II(\dot\gamma, Y) 
	\overset{\eqref{eqn:parallel-transport-gen}}{=} \II(\dot\gamma, Y), \]
	and thus,
	\begin{equation}
		\label{eqn:diff-parallel-transport}
		Y(t) -  \bar P_{\gamma}^{0 \to t}Y(0) 
		=  \int_0^t \frac{\dd}{\dd s} \bar P_{\gamma}^{s \to t}Y(s) \dd s 
		=  \int_0^t \bar P_{\gamma}^{s \to t}\II(\dot\gamma(s), Y(s)) \dd s.
	\end{equation}
	Since $\II$ smoothly depends on the Riemannian metric of $\bar \cM$ and $\cM$ is compact, 
	the constant $C := \sup\{ \|\II(u, v) \| : p \in \cM, u, v \in T_p\cM, \|u\| = \|v\| = 1 \}$ is finite.
	As the parallel transport is an isometry, it holds that $\|Y(t) \| = \|Y(0)\| = \|v\|$ and
	\[ \begin{aligned}
		\| Y(t) -  \bar P_{\gamma}^{0 \to t}Y(0)  \|
		\leq  \int_0^t C\|v \|\| \dot \gamma(s) \|  \dd s
		=  C\|v \| \ell(\gamma|_{[0,t]}).
	\end{aligned} \]

	Moreover, we have
	\[ \begin{aligned}
		t \II(\dot\gamma(t), Y(t)) - [Y(t) - \bar P_{\gamma}^{0 \to t} Y(0)]
		\overset{\eqref{eqn:diff-parallel-transport}}&{=}  \int_0^t \II(\dot\gamma(t), Y(t)) -  \bar P_{\gamma}^{s \to t}\II(\dot\gamma(s), Y(s))  \dd s \\
		& \hspace{-12em}
		=  \int_0^t \int_s^t \frac{\dd}{\dd u} \bar P_{\gamma}^{u \to t}\II(\dot\gamma(u), Y(u)) \dd u \dd s
		=  \int_0^t \int_s^t \bar P_{\gamma}^{u \to t} \bar \nabla_{\dot\gamma} \II(\dot\gamma(u), Y(u)) \dd u \dd s.
	\end{aligned} \]
	When $\gamma(t) := \exp_p (tw)$ for $w \in T_p \cM$ with $\| w \| \leq 2$, the above integrand is continuous with respect to $v, w, t, u, p$,
	and thus, the constant
	$C^\prime := 2\sup\{ \| \bar\nabla_{\dot\gamma} \II(\dot\gamma, P_\gamma^{0 \to t} v) \| : p \in \cM \text{ and } w, v \in T_p\cM, \|w \| \leq 2,  \|v\| \leq 2, t \in [0, 1] \}$ is finite.
	Therefore, \eqref{eqn:compare-transport-sec} holds. \hfill $\Box$
\end{proof}

The following lemma controls the error of parallelly transporting a vector along the shortest geodesic on $\cM$ and then moving back along the shortest geodesic on $\bar \cM$.
\begin{lemma}
	\label{lem:cycle-transport}
	Fix $p \in \cM$, then
	there exists $C > 0$ such that for every $v \in T_p \cM$ with $\| v \| \leq 2$, $t \in [0, 1]$ and $w, \xi \in T_p \bar \cM$, the following inequality holds
	\begin{equation}
		\label{eqn:cycle-parallel-transport}
		|\inner{\bar P_{\gamma(t),p} \bar P_{\gamma}^{0 \to t} w  }{ \xi } - \inner{w}{\xi}|
		\leq C t^2 \| w \| \| \xi \|,
	\end{equation}
	where $\gamma(t) := \exp_p (tv)$. Consequently, 
	\begin{equation}
		\label{eqn:cycle-parallel-transport-map}
		\| \bar P_{p,\gamma(t)} - \bar P^{0 \to t}_\gamma \| \leq Ct^2.
	\end{equation}
\end{lemma}
\begin{proof}
	Let $w, \xi \in T_p \bar \cM$ and $v \in T_p \cM$ with $\| v\| \leq 2$ and $\gamma(t) := \exp_p(tv)$, then
	\begin{align}
		 \frac{\dd}{\dd t}  \inner{  \bar P_{\gamma(t),p} \bar P_{\gamma}^{0 \to t} w }{ \xi }  \nonumber
		&= \frac{\dd}{\dd t}  \inner{  \bar P_{\gamma}^{0 \to t} w }{ \bar P_{p, \gamma(t)} \xi } \\
		&= \inner{ \bar  \nabla_{\dot \gamma}\bar P_{\gamma}^{0 \to t} w }{ \bar P_{p, \gamma(t)} \xi }  \nonumber
		 + \inner{  \bar P_{\gamma}^{0 \to t} w }{ \bar \nabla_{\dot\gamma}\bar P_{p, \gamma(t)} \xi } \\
		 \label{eqn:cycle-first-order}
		\overset{\eqref{eqn:parallel-transport-gen}}&{=} \inner{  \bar P_{\gamma}^{0 \to t} w }{ \bar \nabla_{\dot\gamma}\bar P_{p, \gamma(t)} \xi }, \\
			\frac{\dd^2}{\dd t^2} \inner{  \bar P_{\gamma(t),p} \bar P_{\gamma}^{0 \to t} w }{ \xi }
		&= \inner{ \bar\nabla_{\dot\gamma} \bar P_{\gamma}^{0 \to t} w }{ \bar \nabla_{\dot\gamma}\bar P_{p, \gamma(t)} \xi } \nonumber
		 + \inner{ \bar P_{\gamma}^{0 \to t} w }{ \bar\nabla_{\dot\gamma} \bar \nabla_{\dot\gamma}\bar P_{p, \gamma(t)} \xi } \\
		\overset{\eqref{eqn:parallel-transport-gen}}&{=} \inner{ \bar P_{\gamma}^{0 \to t} w }{ \bar\nabla_{\dot\gamma} \bar \nabla_{\dot\gamma}\bar P_{p, \gamma(t)} \xi }.
		 \label{eqn:cycle-second-order}
	\end{align}
	Let $Y(q) := \bar P_{pq} \xi$ for $q \in \bar \cM$, then $Y(\gamma(t)) = \bar P_{p, \gamma(t)} \xi$.
	From \cite[Proposition 2.2(c)]{carmo1992riemannian}, we know $\bar \nabla_{\dot \gamma} Y = \bar \nabla_{\dot \gamma} \bar P_{p, \gamma(t)} \xi$.
	By the definition of $Y$, $\bar \nabla_v Y(p) = 0$ for every $v \in T_p \bar \cM$.
	Therefore, the right-hand side of \eqref{eqn:cycle-first-order} vanishes at $t = 0$.
	Since $\bar \nabla_{\dot\gamma}\bar \nabla_{\dot\gamma} \bar P_{p, \gamma(t)} \xi$ is continuous with respect to $t, v$ and is linear with respect to $\xi$, then
	\[
		C := \sup\left \{ 
			 \| \bar \nabla_{\dot\gamma}\bar \nabla_{\dot\gamma} \bar P_{p, \gamma(t)} \xi \|  :   
			 t \in [0, 1], \| v \| \leq 2, \| \xi \| \leq 1
		 \right \} < \infty.
		\]
	Applying Taylor's theorem at $t = 0$, we know for every $t \in [0, 1]$ and $w, \xi \in T_p \bar \cM$,
	\[
		|\inner{\bar P_{\gamma(t),p} \bar P_{\gamma}^{0 \to t} w  }{ \xi } - \inner{w}{\xi}|
		\leq C t^2 \| w \| \| \xi \|.
	\]
	Since $C$ is independent of $v$, then the above holds for every $\| v \| \leq 2$.
	Finally, \eqref{eqn:cycle-parallel-transport-map} follows from
	\[
		\| \bar P_{p,\gamma(t)} - \bar P^{0 \to t}_\gamma \| = \| \bar P_{\gamma(t),p} \bar P_\gamma^{0 \to t} - \mathrm{Id} \| \leq Ct^2.
	\]
	This completes the proof. \hfill $\Box$
\end{proof}

The lemma below bounds the error of the inverse exponential maps of $\cM$ and $\bar \cM$, which is proved by extending $\exp_p$ to a retraction on $\bar \cM$ using the Fermi coordinate~\cite[p.~135]{lee2018introduction} and applying \cite[Lemma 3]{zhu2020riemannian} to bound the difference of two inverse retractions.
\begin{lemma}
	\label{lem:diff-submanifold-retraction}
	Fix $p_0 \in \cM$, there exist $C, r > 0$ such that the following inequality holds for every $p, q \in B_{r}(p_0)$.
	\begin{equation}
		\label{eqn:diff-submanifold-retraction}
		\| \exp^{-1}_p q - \overline \exp^{-1}_p q \|  \leq C d_{\bar \cM}(p, q)^2.
	\end{equation}
\end{lemma}
\begin{proof}
	From Lemma~\ref{lem:cosine-law}, there exists $r > 0$ such that $\exp_p$ is a diffeomorphism from $\{ v \in T_p \cM : \| v \| < 4r \}$ to $B_{4r}(p)$ for every $p \in B_{4r}(p_0)$.
	According to Theorem 5.25 in \cite{lee2018introduction}, there exists $r^\prime > 0$ such that $\overline \exp$ is a diffeomorphism from
	 $V := \{ (p, v) \in N\cM : p \in B_{2r}(p_0), \| v \| < r^\prime \}$ to its image under $V$, where $N\cM := \bigcup_{p \in \cM} (T_p \cM)^{\perp}$ is the normal bundle.
	 Let $E_1, \dots, E_{d}$ be vector fields on $B_{4r}(p_0)$ such that 
	 $\{ E_1(p), \dots, E_{n}(p) \}$ and $\{ E_{n+1}(p), \dots, E_{d}(p) \}$ are orthonormal bases of $T_p\cM$ and $(T_p\cM)^\perp$, respectively.
	 For $p \in B_r(p_0)$, $v := \sum_{i=1}^d w_i E_i(p) \in T_p \bar \cM$ with $\| v_\top \| < r$ and $\| v_\perp \| < r^\prime$, we define $q(p, v) := \exp_p (v_\top)$ and 
	 \[ E(p, v) := \overline \exp_{q(p, v)} \left ( \sum_{i=n+1}^d w_i E_i(q(p, v)) \right ).  \]
	 The above discussion shows that $E$ is well-defined, and $E(p, v) = \exp_p v$ if $v \in T_p \cM$ and $\| v\| < r$, and $E(p, v) = \overline \exp_p v$ if $v \in (T_p \cM)^\perp$ and $\|v\| < r^\prime$.
	 Thus, $E$ can be regarded as a retraction on $\bar \cM$ defined near $p_0$, and \eqref{eqn:diff-submanifold-retraction} follows from \cite[Lemma 3]{zhu2020riemannian}. \hfill $\Box$
\end{proof}

\begin{proof}[Proof of Theorem~\ref{thm:semismoothness-from-euc}]

	(i) First, we verify the local Lipschitz property of $X$. 
	Since $\bar X$ is locally Lipschitz at $p \in \cM$, then there exist $L_p > 0$ and a neighborhood $\bar U_p \subset \bar \cM$ at $p$ such that $\bar X$ is $L_p$-Lipschitz in $\bar U_p$.
	By Lemma~\ref{lem:compare-transport}, we know for every $p, q \in \cM \cap \bar U_p$ and every geodesic $\gamma$ joining $p$ and $q$, it holds that
	\[ \begin{aligned}
		\| P_{\gamma}^{0 \to 1} X(p) - X(q) \|
		& \leq 
		\| P_{\gamma}^{0 \to 1} X(p) - \bar P_{\gamma}^{0 \to 1}X(p) \|
		+ \| \bar P_{\gamma}^{0 \to 1}X(p) - X(q) \| \\
		& \leq \left (C \sup_{p \in \cM }\| X(p) \| + L_p \right ) \ell(\gamma),
	\end{aligned} \]
	where $C$ is the constant in Lemma~\ref{lem:compare-transport}.
	Since $\cM$ is compact and $X$ is continuous, $M := \sup_{p \in \cM} \| X(p) \| < \infty$, then $X$ is Lipschitz on $\cM \cap \bar U_p$, i.e., $X$ is locally Lipschitz at $p$.

	(ii)
	Next, we show that $X$ is directionally differentiable at $p$ in the Hadamard sense.
	Fix $v \in T_p \cM$ with $\| v \| = 1$ and 
	let $\gamma(t) := \exp_p(tv^\prime)$ for $t \in (-1, 1)$ and $v^\prime \in T_p \cM$ with $\| v^\prime \| \leq 2$, and decompose $P_\gamma^{t \to 0} X(\exp_p(tv^\prime)) - X(p) $ into
	\begin{equation}
		\label{eqn:directional-decomp}
		\begin{aligned}
		\underbrace{ (P_\gamma^{t \to 0}  - \bar P_\gamma^{t \to 0})X(\gamma(t)) - t S(t) }_{\text{(A)}} 
		+ \underbrace{ (\bar P_\gamma^{t \to 0} - \bar P_{\gamma(t),p}) X(\gamma(t)) }_{\text{(B)}} 
		+ \underbrace{  \bar P_{\gamma(t),p} X(\gamma(t)) - X(p) + tS(t) }_{\text{(C)}},
		\end{aligned}
	\end{equation}
	where $S(t) := \II(-\dot\gamma(0), P_\gamma^{t \to 0} X(\gamma(t)))$.

	From Lemma~\ref{lem:compare-transport} and \ref{lem:cycle-transport}, there exists $C > 0$ (independent of $v^\prime$)  such that
	\[ \begin{aligned} 
		| (\mathrm A) |
		\overset{\eqref{eqn:compare-transport-sec}}&{\leq} Ct^2
		\quad \text{ and } \quad
		|(\mathrm B)| \leq   \|  \bar P_\gamma^{t \to 0} - \bar P_{\gamma(t), p} \| \| X(\gamma(t)) \|
		\overset{\eqref{eqn:cycle-parallel-transport-map}}{\leq}
		CM t^2.
	\end{aligned} \]
	When $t \to 0^{+}$ and $v^\prime \to v$, we know 
	$|(\mathrm A)| + |(\mathrm B)| = O(t^2)$ and  
	$v^\prime_t := t^{-1}\overline \exp^{-1}_p \gamma(t) \to v$, 
	and hence
	the continuity of $S(t)$ and the Hadamard differentiability~\eqref{eqn:Hadamard-directionally-differentiable} imply that 
	\[ \begin{aligned}
		&\peq \lim_{\substack{t \downarrow 0 \\ v^\prime \to v}}
		\frac{1}{t} \left [ P_{\exp_p(tv^\prime),p} X(\exp_p(tv^\prime)) - X(p) \right ]
		= 
		\lim_{\substack{t \downarrow 0 \\ v^\prime \to v}}
		\frac{1}{t}(\mathrm C)  \\
		& %
		=  \lim_{t \downarrow 0} \frac{1}{t} \left [ P_{\overline \exp_p(tv_t^\prime),p} X(\overline \exp_p(tv_t^\prime)) - X(p) \right ]
		+ \lim_{\substack{t \downarrow 0 \\ v^\prime \to v}} S(t)
		\overset{\eqref{eqn:Hadamard-directionally-differentiable}}{=}
		 \bar \nabla X(p; v) - \II(v, X(p)).
	\end{aligned} \]
	Note that $\II(v, X(p)) \in (T_p\cM)^\perp$, we find $\nabla X(p; v) = \bar \nabla X(p; v) - \II(v, X(p)) = (\bar \nabla X(p; v))_\top$ for every $\| v\| = 1$.
	Since $\nabla X(p; tv) = t\nabla X(p; v)$ for all $t > 0$, we know $X$ is directionally differentiable at $p$ in the Hadamard sense.

	(iii.b) Finally, we need to verify the inequality \eqref{eqn:theorem-semismooth-submanifolds}. 
	Suppose that the assumption (iii.b) holds.
	Since the parallel transport is an isometry, then from \eqref{eqn:theorem-semismooth-ambient}, 
	there exist $C > 0, \delta \in (0, 1)$ such that for every $q \in \cM$ with $d_{\bar \cM}(p, q) < \delta$ and $\bar H_q \in \bar \cK(q)$, it holds that
	\begin{equation}
		\label{eqn:proof-semismooth-ambient}
		\| \bar P_{pq} X(p) -  X(q) - \bar H_q \overline \exp_q^{-1} p \| \leq C d_{\bar \cM}(p, q)^{1 + \mu} \leq C d_\cM(p, q)^{1 + \mu}.
	\end{equation}
	Without the loss of generality, we could assume that $B_\delta(p) \subset V_p$, where $V_p$ is the neighborhood of $p$ in the assumption (iii.b).
	For every $q \in \cM$ and $H_q \in \cK(q)$ such that $d_\cM(q, p) < \delta$, 
	the assumption (iii.b) states that there exists $\bar H_q \in \bar \cK(q)$ satisfying 
	$(\bar H_q \exp^{-1}_qp)_\perp = \II(\exp^{-1}_qp, X(q))$ and $(\bar H_q \exp^{-1}_qp)_\top = H_q \exp^{-1}_qp$.
	Let $\alpha = d_\cM(p, q)$, $\xi := \alpha^{-1} \exp^{-1}_p q$ and $\gamma(t) := \exp_p(  t \xi)$, it holds that
	\[ \begin{aligned}
		\| P_{pq} X(p) - X(q) - H_q \exp^{-1}_q p \|
		&\leq 
		\underbrace{ \| P_{pq} X(p) - \bar P_{\gamma}^{0 \to \alpha} X(p) + \II(\exp^{-1}_q p, X(q)) \|  }_{\text{(A)}}
		\\
		&\hspace{-12em} 
		+ \underbrace{\| ( \bar P_{\gamma}^{0 \to \alpha} - \bar P_{pq} )X(p) \| }_{\text{(B)}}
		+ \underbrace{\|  \bar P_{pq}X(p) - X(q) - \bar H_q \overline\exp^{-1}_qp \| }_{\text{(C)}}
		+ \underbrace{\| \bar H_q (\overline\exp^{-1}_qp -  \exp^{-1}_q p) \| }_{\text{(D)}}.
	\end{aligned} \]

	Let $r, \tilde C > 0$ be the constants such that Lemma~\ref{lem:compare-transport}, \ref{lem:cycle-transport} and \ref{lem:diff-submanifold-retraction} hold with $p_0 = p$, 
	and $\tilde L_p > 0$ be the Lipschitz constant of $X$ at $p$, and by the compactness of $\cM$ and the continuity of $\II(v, w)$, we further assume that $\sup\{ \| \II(v, w) \| : p \in \cM, v, w \in T_p\cM, \|v\|=\|w\|=1 \} < \tilde C < \infty$.
	Note that $\dot\gamma(\alpha) = -\alpha^{-1} \exp^{-1}_q p$ and
		$\II(\exp^{-1}_qp, X(q)) = -\alpha \II(\dot\gamma(\alpha), X(q))$,
		then 
	\[ \begin{aligned}
		|(\mathrm A)|
		&\leq 
		\| P_{pq} X(p) - \bar P_{\gamma}^{0 \to \alpha} X(p) - \alpha \II(\dot\gamma(\alpha), P_\gamma^{0 \to \alpha}X(p)) \|
		+ \alpha \| \II(\dot\gamma(\alpha), P_\gamma^{0 \to \alpha}X(p) - X(q)) \| \\
		\overset{\eqref{eqn:compare-transport-sec}}&\leq
		\tilde C \alpha^2
		+ \tilde C \alpha \| \dot\gamma(\alpha) \| \| P_\gamma^{0 \to \alpha}X(p) - X(q) \|
		\leq \tilde C(1 + \tilde L_p) \alpha^2,
	\end{aligned} \]
	where the last inequality follows from the Lipschitzness of $X$ at $p$.
	From \eqref{eqn:cycle-parallel-transport-map} we find that 
	$|(\mathrm B)| \leq \tilde CM\alpha^2$.
	Since $\cK$ is locally bounded, there exists $\tilde M > 0$ such that all elements in $\bigcup_{d(p, q) \leq \delta} \bar \cK(q)$ are bounded by $\tilde M$,
	and thus \eqref{eqn:diff-submanifold-retraction} yields $|(\mathrm D)| \leq \tilde C \tilde M \alpha^2$.
	Combining with \eqref{eqn:proof-semismooth-ambient} and letting $C_0 := \tilde C(M + \tilde M + \tilde L_p + 1)$, it holds that
	\[ \begin{aligned}
		\| P_{pq} X(p) - X(q) - H_q \exp^{-1}_q p \| \leq C d_\cM(p, q)^{1 + \mu} +  C_0 d_\cM(p, q)^2.
	\end{aligned} \]
	Thus, the inequality \eqref{eqn:theorem-semismooth-submanifolds} holds for $q \in \cM$ with $d_\cM(p, q) < \tilde \delta := \min\{ 1, r, \delta \}$ and $\hat C := C + C_0$.
	Moreover, when $\mu \in [0, 1)$, the same inequality holds for $q \in \cM$ with $d_\cM(p, q) < \tilde \delta := \min\{ r, \delta, (C/C_0)^{\frac{1}{1 - \mu}} \}$  and $\hat C := 2C$.

	(iii.a) When $X(p) = 0$, we know
	\[ \begin{aligned}
		\| X(q) + H_q \exp^{-1}_q p \|
		& =
		\| (\bar X(q) + \bar H_q \exp^{-1}_q p)_\top \| 
		\leq \| \bar X(q) + \bar H_q \exp^{-1}_q p\| \\
		&\leq  \| \bar X(q) + \bar H_q \overline \exp^{-1}_q p \|
		+ \| \bar H_q (\overline\exp^{-1}_qp -  \exp^{-1}_q p) \|.
	\end{aligned} \]
	The above two terms can be controlled by \eqref{eqn:theorem-semismooth-ambient} and \eqref{eqn:diff-submanifold-retraction}, respectively, and therefore
	the conclusion follows from a similar discussion as in (iii.b). 

	(iv) 
	Since when $\mu = 0$, we know for every $C > 0$, if \eqref{eqn:theorem-semismooth-ambient} holds, then \eqref{eqn:theorem-semismooth-submanifolds} holds with $\hat C = 2C$, which implies the semismoothness in Definition~\ref{def:semismooth-general}.
	The conclusion for the semismoothness with order $\mu \in (0, 1]$ directly follows from (i), (ii) and (iii).
	\hfill $\Box$
\end{proof}

\subsection{Proof of Proposition~\ref{prop:clarke-inclusion}}
\label{app:proof-clarke-inclusion}

First, we briefly review how to represent quantities on manifolds under a local coordinate.
Let $(U, \varphi)$ be a chart near $p \in \cM$, and $\{ e_i \}$ be the standard basis of $\R^n$, i.e., the $j$-th component of $e_i$ is $\delta_{ij}$, and let $E_i := (\dd\varphi)^{-1}[e_i]$, then $E_i$ is a smooth vector field near $p$ and $\{ E_i(q) \}$  is a basis of $T_q\cM$ for every $q \in U$. 
Let $g_{ij} := \inner{E_i}{E_j}$, $G := (g_{ij})_{i,j\in[n]}$ and $\Gamma_{ij}^k$ be the \emph{Christoffel symbols} such that $\nabla_{E_i}{E_j} = \sum_{k=1}^n \Gamma_{ij}^k E_k$, and $g^{ij}$ be such that $( g^{ij} )_{i,j\in[n]} = G^{-1}$, and we call $G$ the \emph{metric matrix}.
The gradient of a smooth function $f:\cM \to \R$ is $\grad f = \sum_{i, j=1}^n (g^{ij} E_i f)E_j$ (see \cite[p.~27]{lee2018introduction}).
Two smooth vector fields $X, Y$ defined near $p$ can be represented as $X = \sum_{i=1}^n X^i E_i$ and $Y = \sum_{i=1}^n Y^i E_i$, and their inner product is $\inner{X}{Y} = \sum_{i,j=1}^n g_{ij} X^i Y^j$.
The covariant derivative $\nabla_XY$ can be written as (see \cite[p.~92]{lee2018introduction})
\begin{equation}
	 \nabla_{X}Y = \sum_{i=1}^n (X Y^i) E_i + \sum_{i, j, k=1}^n X^iY^j \Gamma_{ij}^k E_k. 
\end{equation}

When $(U, \varphi)$ is the chart constructed in the proof of Lemma~\ref{lem:retraction-approx}, $\varphi^{-1}$ is the \emph{normal coordinate} at $p \in \cM$ (see \cite[p.~132]{lee2018introduction}). Proposition 5.24 in \cite{lee2018introduction} shows that $g_{ij}(p) = \delta_{ij}$, $\Gamma_{ij}^k(p) = 0$, and all partial derivatives of $g_{ij}$ vanish at $p$.
Thus, $\grad f(p) = \sum_{i=1}^n (E_i f)(p) E_i(p)$, $\nabla_XY(p) = \sum_{i=1}^n (X Y^i)(p) E_i(p)$ and $\inner{X(p)}{Y(p)} = \sum_{i=1}^n X^i(p) Y^i(p)$.

The following two lemmas are some chain rules on manifolds, 
which are proved by pulling functions on manifolds back to Euclidean spaces and then applying the Euclidean chain rules.

\begin{lemma}
	\label{lem:clarke-chain-rule}
	Let $\cM, \cN$ be two Riemannian manifolds, and $F: \cM \to \cN$ be a smooth map, $g: \cN \to \R$ be a map that is locally Lipschitz at $F(p) \in \cN$, then the following chain rule holds at $p \in \cM$
	\begin{equation}
		\partial (g \circ F)(p) \subseteq 
		\partial g(F(p)) \dd F(p),
	\end{equation}
	where the inclusion means that for every $\xi \in \partial (g \circ F)(p)$, there exists $\zeta \in \partial g(F(p))$ such that $\inner{\xi}{v} = \inner{\zeta}{\dd F|_p[v]}$ for every $v \in T_p\cM$.
\end{lemma}
\begin{proof}
	Let $(U, \varphi)$ and $(V, \psi)$ be charts near $p \in \cM$ and $F(p) \in \cN$ with $\varphi(0) = p$ and $\psi(0) = F(p)$, respectively. 
	We denote $G_\cN, G_\cM$ as the metric matrices on $\cN, \cM$ defined above, respectively.
	Define $\hat g := g \circ \psi^{-1}$ and $\hat F := \psi \circ F \circ \varphi^{-1}$. Then, Theorem 2.3.10 in \cite{clarke1990optimization} yields that
	\[ \partial (\hat g \circ \hat F)(0) \subseteq
	\partial \hat g(\hat F(0)) \dd \hat F(0). \]
	Fix $\xi \in \partial(g \circ F)(p)$, then from 
	\eqref{eqn:subgradient-from-euclidean} we know
	$\hat \xi := G_\cM\dd \varphi|_p[\xi] \in \partial(\hat g \circ \hat F)(0)$.
	The above display shows that there exists $\hat \zeta \in \partial \hat g(\hat F(0))$ such that for every $\hat v \in \R^n$, it holds that
	$ 
	\hat \xi^\trans \hat v = \hat \zeta^\trans (\dd \hat F(0)[\hat v])
	$.
	Let $v := (\dd\varphi|_p)^{-1}[\hat v]$, then
	$\hat\xi^\trans \hat v
	= (\dd \varphi|_p[\xi])^\trans G_\cM (\dd \varphi|_p[v]) = \inner{\xi}{v}$.
	Since $\zeta := (\dd \psi|_{F(p)})^{-1} G_\cN^{-1} \hat \zeta \in \partial g(F(p))$ and 
	$\dd \hat F = \dd \psi \circ \dd F \circ (\dd \varphi)^{-1}$, we find that
	$\hat \zeta^\trans (\dd \hat F(0)[\hat v])
	= 
	(\dd \psi|_{F(p)}[\zeta])^\trans
	G_\cN ( \dd \psi|_{F(p)} [\dd F|_p[v]])
	= \inner{\zeta}{\dd F|_p[v]}
	$.
	Therefore, $\inner{\xi}{v} = \inner{\zeta}{\dd F|_p[v]}$ for every $v \in T_p \cM$, and hence $\xi \in \partial g(F(p)) \dd F(p)$.
	\hfill $\Box$
\end{proof}

\begin{lemma}
	\label{lem:clarke-chain-rule-equal}
	Let $\cM$ be a Riemannian manifold, $X$ be a vector field on $\cM$ that is locally Lipschitz at~$p \in \cM$, and $Y$ be a smooth vector field on $\cM$. 
	Define $f := \inner{X}{Y}$, then the following chain rule holds.
	\begin{equation}
		\label{eqn:clarke-chain-rule-equal}
		\partial f(p) = \inner{\partial X(p)}{Y(p)} + \inner{X(p)}{\nabla Y(p)},
	\end{equation}
	where $\xi \in \partial f(p)$ is regarded as the operator $\zeta \mapsto \inner{\xi}{\zeta}$ for $\zeta \in T_p\cM$, and the right-hand side is understood as the set of operators $\zeta \mapsto \inner{H\zeta}{Y(p)} + \inner{X(p)}{\nabla_\zeta Y(p)}$ for $\zeta \in T_p\cM$ and $H \in \partial X(p)$.
\end{lemma}
\begin{proof}
	Let $(U_p, \varphi)$ be a chart such that $\varphi^{-1}$ is the normal coordinate.
	Suppose that $X = \sum_{i=1}^n X^i E_i$ and $Y = \sum_{i=1}^n Y^i E_i$, 
	we define $\hat f := f \circ \varphi^{-1}$,
	$\hat X^j := X^j \circ \varphi^{-1}$, $\hat X := (\hat X^1, \dots, \hat X^n)$, and $\hat Y$ likewise.
	Thus, 
	$\hat f(x) = \hat X(x)^\trans G(x) \hat Y(x)$ for $x \in \hat U_p := \varphi(U_p)$,
	and can be regarded as the composition $\cA \circ \tilde X$, where
	$\cA(x, v) := v^\trans G(x) \hat Y(x)$ and $\tilde X(x) := (x, \hat X(x))$ for $x \in U_p, v \in \R^n$.
	Since $\cA$ is smooth and~$G(0) = I_n, \nabla G(0) = 0$, we know
	$\nabla \cA(0, v)[w, \xi] = v^\trans \nabla \hat Y(0)[w] + \xi^\trans \hat Y(0)$ for $w, \xi \in \R^n$.
	Note that~$\partial \tilde X = (\mathrm{Id}, \partial \hat X)$, then the chain rule in Euclidean spaces~\cite[Theorem~2.6.6]{clarke1990optimization} implies that 
	\begin{equation}
		\label{eqn:lem-euc-gradient-inner}
		\partial \hat f(0) = \nabla \cA(0, \hat X(0))[\mathrm{Id}, \partial \hat X(0)] = \hat X(0)^\trans \nabla \hat Y(0) + \hat Y(0)^\trans \partial \hat X(0).   
	\end{equation}

	We define the linear bijection $J: \cL(\R^n) \to \cL(T_p\cM)$ as
	\[ (J\hat H)v := \sum_{i=1}^n [\hat H \hat v]_i E_i(p), \]
	where $\hat H \in \cL(\R^n)$, $v \in T_p \cM$ and $\hat v := \dd\varphi|_p[v]$.
	Below we show that $J(\partial \hat X(0)) = \partial X(p)$. 

	Suppose $\hat H \in \partial_B \hat X(0)$, then there exists $x_k \to 0$ such that $\hat X$ is differentiable at $x_k$ and $\nabla \hat X(x_k) \to \hat H$ as $k\to\infty$.
	Let $H := J\hat H$, $\hat v \in \R^n$, $v := (\dd\varphi|_p)^{-1}[\hat v]$, $p_k := \varphi^{-1}(x_k)$ and $v_k := (\dd \varphi|_{p_k})^{-1}[\hat v]$, then
	$\nabla \hat X^i(x_k)[\hat v] = \dd \hat X^i|_{x_k}[\hat v] = \dd X^i|_{p_k}[v_k] = v_kX^i(p_k)$,
	and therefore
	\[ \nabla_{v_k} X(p_k) = \sum_{i=1}^n \nabla \hat X^i(x_k)[\hat v] E_i + \sum_{i, j, k=1}^n X^iY^j\Gamma_{ij}^k E_k. \]
	Let $k \to \infty$ and note $\Gamma_{ij}^k(p) = 0$, we find $\nabla_{v_k} X(p_k) \to \sum_{i=1}^n [\hat H \hat v]_i E_i(p) = Hv$.
	Since $\hat v$ is arbitrary, we conclude that $\nabla X(p_k) \to  H \in \partial_B X(p)$.
	By reversing this argument, we can show that $J^{-1}H \in \partial \hat X(0)$ for every $H \in \partial X(p)$.
	Then, $J(\partial_B \hat X(0)) = \partial_B X(p)$, and
	by the linearity of $J$, we know $J(\partial \hat X(0)) = \partial X(p)$.

	Thus, $\hat Y(0)^\trans (\hat H\hat v)
	= \sum_{i=1}^n Y^i(p) \innerz{J\hat Hv}{E_i(p)} = \innerz{J\hat Hv}{Y(p)}$ for every $\hat H \in \partial \hat X(0)$, $v \in T_p \cM$ and $\hat v := \dd\varphi|_p[v]$.
	Similarly, it holds that 
	$\hat X(0)^\trans \nabla \hat Y(0)[ \hat v] = \inner{X(p)}{\nabla Y(p)[v]}$.
	Besides, \eqref{eqn:subgradient-from-euclidean} yields $\partial \hat f(0)[\hat v] = \partial f(p)[v]$.
	Therefore, the proof is completed by 
	plugging these equations into~\eqref{eqn:lem-euc-gradient-inner}.
	\hfill $\Box$
\end{proof}

\begin{proof}[Proof of Proposition~\ref{prop:clarke-inclusion}]
	Since Theorem~\ref{thm:semismoothness-from-euc} shows that $X$ is also locally Lipschitz at $p$, then $\partial X(p)$ is well-defined.
	Let $U_p \subset \cM$ be a neighborhood of $p$, $\xi \in T_p\cM$
	and $V$ be a smooth vector field on $\bar \cM$ such that $V_\top(q) = P_{pq}\xi$ for $q \in U_p$,
	and $\iota:  \cM \hookrightarrow \bar \cM$ be the embedding map.
	Define $\bar f_V := \inner{\bar X}{V}$ and $f_V := \bar f_V \circ \iota$,
	then Lemma~\ref{lem:clarke-chain-rule} shows that 
	$\partial f_V(p) \subseteq    \bar \partial \bar f_V(p) \dd \iota(p)   = \bar \partial \bar f_V(p)|_{T_p \cM}$. 
	
	Let $v \in T_p \cM$, then Lemma~\ref{lem:clarke-chain-rule-equal} gives that 
	$\bar\partial \bar f_V(p)[v] =  \inner{\bar \partial \bar X(p)[v]}{V(p)} + \inner{\bar X(p)}{\bar \nabla_v V(p)}$.
	Since $v, \bar X(p) \in T_p \cM$, Proposition 8.4 in \cite{lee2018introduction} shows that
	 \[\inner{\bar X(p)}{\bar \nabla_v V_\perp(p)} 
	=	\inner{\bar X(p)}{(\bar \nabla_v V_\perp(p))_\top} 
	=	\inner{X(p)}{\nabla_v V_\perp(p)} 
	 = -\inner{\II(v, X(p))}{V_\perp(p)}. \]
	Note that $\nabla_v V_\top(p) = \nabla_v (P_{pq}\xi)(p) = 0$, then $\inner{\bar X(p)}{\bar \nabla_v V_\top(p)} = \inner{X(p)}{\nabla_v V_\top(p)} = 0$,
	and thus, 
	\[ \inner{\bar X(p)}{\bar \nabla_vV(p)} 
	= \inner{\bar X(p)}{\bar \nabla_vV_\perp(p) + \bar \nabla_vV_\top(p)} 
	= -\inner{\II(v,X(p))}{V_\perp(p)}
	= -\inner{\II(v,X(p))}{V(p)},  \]
	where the last equation follows from $\inner{\II(v, X(p))}{V_\top(p)} = 0$.

	Combining these arguments, we know
	$\partial f_V(p)[v] \subseteq \innerz{\bar \partial \bar X(p)[v] - \II(v, X(p))}{V(p)}$.
	Note that $\bar X(q) \in T_q \cM$ for $q \in U_p$, then 
	$f_V(q) = \inner{X(q)}{V_\top(q)}$ for $q \in U_p$.
	Since $\nabla_v V_\top(p) = 0$ and $\partial X(p)[v] \subset T_p\cM$, then Lemma~\ref{lem:clarke-chain-rule-equal} yields
	$\partial f_V(p)[v] = \inner{\partial X(p)[v]}{V_\top(p)} = \inner{\partial X(p)[v]}{V(p)}$.
	Therefore, we know
	\[ \inner{\partial X(p)[v]}{V(p)} \subseteq\innerz{\bar \partial \bar X(p)[v] - \II(v, X(p))}{V(p)}. \]
	Note that $V(p)_\top = \xi$ and $V(p)_\perp$ are arbitrary and $\partial  X(p)[v], \bar \partial \bar X(p)[v]$ are compact convex sets, 
	the inclusion relationship \eqref{eqn:clarke-inclusion} follows from \cite[Corollary 13.1.1]{rockafellar1970convex}.
	\hfill $\Box$
\end{proof}

\end{appendices}

\bibliographystyle{plain}
\bibliography{references}
\end{document}